\RequirePackage{rotating}
\documentclass[a4paper,10pt]{amsart}
\usepackage[margin=1in]{geometry}
\usepackage{enumerate, amsmath, amsfonts, amssymb, amsthm, mathtools, thmtools, wasysym, graphics, graphicx, xcolor, frcursive,xparse,comment,ytableau,stmaryrd,bbm,array,colortbl,tensor,bbold,arydshln,leftidx}
\usepackage[all]{xy}
\usepackage[boxed,norelsize]{algorithm2e}
\usepackage{hyperref}
\usepackage{cancel}
\usepackage{calligra}
\DeclareMathAlphabet{\mathcalligra}{T1}{calligra}{m}{n}

\makeatletter
\newcommand{\specialcell}[1]{\ifmeasuring@#1\else\omit$\displaystyle#1$\ignorespaces\fi}
\makeatother

\usepackage{url, hypcap}
\hypersetup{colorlinks=true, citecolor=darkblue, linkcolor=darkblue}
\definecolor{join}{RGB}{0,77,178}
\usepackage{tikz}
\usetikzlibrary{calc,through,backgrounds,shapes,matrix}

\definecolor{darkblue}{rgb}{0.0,0,0.7} 
\newcommand{\darkblue}{\color{darkblue}} 
\definecolor{darkred}{rgb}{0.7,0,0} 
 \definecolor{lightgrey}{rgb}{0.7,0.7,0.7}

\definecolor{meet}{RGB}{255,205,111}
\definecolor{join}{RGB}{0,77,178}

\DeclareFontEncoding{OT2}{}{} 
\newcommand{\textcyr}[1]{{\fontencoding{OT2}\fontfamily{wncyr}\fontseries{m}\fontshape{n}\selectfont #1}}\newcommand{\Sha}{{\mbox{\textcyr{Sh}}}}

\newtheorem{theorem}{Theorem}[section]
\newtheorem{proposition}[theorem]{Proposition}
\newtheorem{corollary}[theorem]{Corollary}
\newtheorem{lemma}[theorem]{Lemma}

\theoremstyle{definition}
\newtheorem{definition}[theorem]{Definition}
\newtheorem{example}[theorem]{Example}

\usepackage[procnames]{listings}
\usepackage[nameinlink]{cleveref}
\usepackage[all]{xy}
\usepackage[T1]{fontenc}
\crefformat{footnote}{#2\footnotemark[#1]#3}
\Crefname{conjecture}{Conjecture}{Conjectures}

\usepackage[cal=boondoxo]{mathalfa}
\usepackage[colorinlistoftodos]{todonotes}
\newcommand{\defn}[1]{\emph{\darkblue #1}}

\usetikzlibrary{math}
\usepackage{xifthen}
\usepackage{xstring}
\SetKwInput{kwset}{set}
\usetikzlibrary{arrows,backgrounds,calc,trees}
\pgfdeclarelayer{background}
\pgfsetlayers{background,main}

\newcommand{\rr}{\mathcal{r}}
\newcommand{\calw}{\mathcal{w}}

\newcommand{\pop}{\mathsf{Pop}}

\newcommand{\w}{\mathbf{w}}
\newcommand{\s}{\mathbf{s}}
\newcommand{\des}{\mathrm{des}}
\newcommand{\Pop}{\mathsf{Pop}}
\newcommand{\Snap}{\mathsf{Snap}}
\newcommand{\Crackle}{\mathsf{Crackle}}
\newcommand{\Pow}{\mathsf{Pow}}

\newcommand{\iotab}{\overline{\iota}}

\newcommand{\HH}{\mathcal{H}}

\newcommand{\cov}{\mathrm{cov}}

\newcommand{\inv}{\mathrm{inv}}

\newcommand{\Inv}{\mathrm{Inv}}
\newcommand{\Sort}{\mathrm{Sort}}
\newcommand{\NC}{\mathrm{NC}}

\newcommand{\covsha}{\cov_{\text{\Sha}}}

\newcommand{\TTsha}{\mathcal L_{\text{\Sha}}}
\newcommand{\len}{\mathrm{len}}
\newcommand{\lift}{\mathrm{lift}}

\newcommand{\id}{\mathbbm{1}}

\renewcommand{\tt}{\ell}
\newcommand{\ttt}{\mathcal{l}}
\newcommand{\bb}{\delta}

\newcommand{\TT}{\mathcal{L}}

\renewcommand{\AA}{\mathcal{A}}
\renewcommand{\HH}{\mathcal{H}}
\newcommand{\RA}{\mathcal{R}}

\renewcommand{\t}{\mathbf{t}}
\newcommand{\PP}{\mathrm{Sal}}
\newcommand{\Sal}{\mathrm{Sal}}

\newcommand{\gal}{\mathrm{gal}}

\newcommand{\Shard}{\mathrm{Shard}}
\newcommand{\Weak}{\mathrm{Weak}}
\newcommand{\CC}{\mathbb{C}}
\newcommand{\RR}{\mathcal{R}}
\newcommand{\galb}{\delta}
\newcommand{\BB}{\mathbf{B}}
\newcommand{\PPP}{\mathbf{P}}

\title[Pop, Crackle, Snap (and Pow): Some Facets of Shards]{Pop, Crackle, Snap (and Pow): \\ Some Facets of Shards}

\author[C.~Defant]{Colin Defant}
\address[C.~Defant]{Massachusetts Institute of Technology}
\email{colindefant@gmail.com}

\author[N.~Williams]{Nathan Williams}
\address[N.~Williams]{University of Texas at Dallas}
\email{nathan.williams1@utdallas.edu}

\begin{document}

\begin{abstract}  
Reading cut the hyperplanes in a real central arrangement $\mathcal H$ into pieces called \emph{shards}, which reflect order-theoretic properties of the arrangement.  We show that shards have a natural interpretation as certain generators of the fundamental group of the complement of the complexification of $\HH$.  Taking only positive expressions in these generators yields a new poset that we call the \emph{pure shard monoid}.

When $\mathcal H$ is simplicial, its poset of regions is a lattice, so it comes equipped with a pop-stack sorting operator $\mathsf{Pop}$.  In this case, we use $\mathsf{Pop}$ to define an embedding $\mathsf{Crackle}$ of Reading's shard intersection order into the pure shard monoid.  When $\mathcal H$ is the reflection arrangement of a finite Coxeter group, we also define a poset embedding $\mathsf{Snap}$ of the shard intersection order into the positive braid monoid; in this case, our three maps are related by $\mathsf{Snap}=\mathsf{Crackle} \cdot \mathsf{Pop}$.
\end{abstract}

\maketitle

\section{Introduction}
\label{sec:intro}

\subsection{Introduction \texorpdfstring{(\Cref{sec:intro})}{}}
Throughout this paper, we let $\mathcal H$ be a finite central irreducible real hyperplane arrangement. Salvetti introduced a certain CW complex associated to $\mathcal H$ and used it to provide a presentation of the fundamental group of the complement of the complexification of $\mathcal H$. We prove that Salvetti's generating set is parameterized by shards, which Reading introduced and used to define his shard intersection order in the case when $\mathcal H$ is simplicial. We introduce the \emph{pure shard monoid}, which is the monoid generated by Salvetti's generators; it comes equipped with a natural partial order that we believe deserves further attention. In the case when $\HH$ is an arrangement of rank $2$ with $m$ hyperplanes, we prove that the interval from the identity element to the \emph{full twist} in the pure shard monoid is a planar lattice with rank generating function $1+\left(\sum_{k=1}^{m-1} \left(2\binom{m}{k}-2\right)q^k \right)+q^m$ and with $m 2^{m-2}$ maximal chains. 

We then assume $\HH$ is simplicial. In this case, there is a known characterization of the shard intersection order involving the pop-stack sorting operator $\mathsf{Pop}$ on the poset of regions of $\mathcal H$. We introduce a new map $\mathsf{Crackle}$, and we prove that it is a poset embedding of the shard intersection order into the pure shard monoid. 

Next, we specialize further to the case when $\mathcal H$ is the reflection arrangement of a finite Coxeter group $W$. We introduce another map $\mathsf{Snap}$, and we prove that it is a poset embedding of the shard intersection order on $W$ into the weak order on the positive braid monoid. The restriction of $\mathsf{Snap}$ to the set $\mathrm{Sort}(W,c)$ of $c$-sortable elements of $W$ originally arose in connection with Deodhar decompositions of noncrossing Catalan varieties. In this setting, we obtain as a corollary that $\mathsf{Snap}$ restricts to a poset embedding of the shard intersection order on $\mathrm{Sort}(W,c)$---which is isomorphic to the noncrossing partition lattice of $W$---into the weak order on the positive braid monoid. 

Finally, we turn back to arbitrary finite central irreducible real arrangements and define a fourth map $\Pow$. We prove that $\Pow$ is a poset embedding of the poset of regions of $\HH$ into the pure shard monoid.

\subsection{The Salvetti Complex \texorpdfstring{ (\Cref{sec:salvetti})}{}}\label{subsec:shards_and_loops}
We write $\RR$ for the set of regions (connected components of the complement) of $\HH$, and we fix a base region $B\in\RR$ and a point $x_B \in B$.  
Write $\CC^n \setminus \HH_\CC$ for the complexified hyperplane complement and $\pi_1(\CC^n \setminus \HH_\CC, x_B)$ for its fundamental group with base point $x_B$.    Let $\Weak(\HH,B):=(\RR,\leq)$ be the usual poset of regions of $\HH$ with respect to the base region $B$.  

Following~\cite{salvetti1987topology,deligne1972immeubles}, we construct a CW complex $\PP(\HH)$ by gluing together oriented dual zonotopes for $\HH$ along compatible faces---one zonotope for each choice of base region $B$, oriented from $B$ to $-B$.  The resulting CW complex $\PP(\HH)$ has the same fundamental group as the complexified hyperplane complement: \[\pi_1(\PP(\HH),B) = \pi_1(\CC^n \setminus \HH_\CC,x_B).\]  The 1-skeleton of $\Sal(\HH)$ is given by orienting all edges of $\Weak(\HH,B)$ away from $B$, and then for each edge $e$, adding in a reversed edge $e^*$. An illustration is given in~\Cref{fig:rank2}.

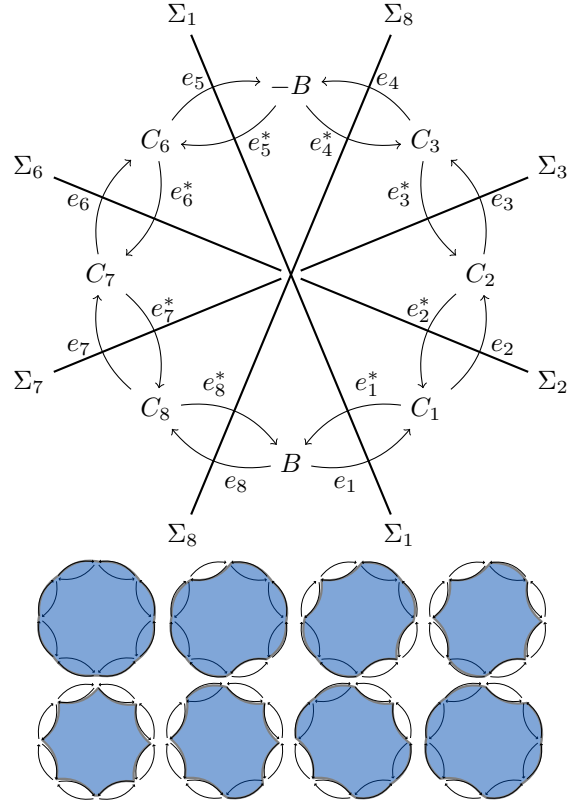
\begin{figure}[htbp]
\raisebox{-0.5\height}{\begin{tikzpicture}[scale=2.5,->]
\node (0) at (0,0) {};
\node (1) at (1,0) {$C_2$};
\node (2) at (1.38,.57) {$\Sigma_3$};
\node (3) at (.71,.71) {$C_3$};
\node (4) at (.57,1.38) {$\Sigma_8$};
\node (5) at (0,1) {$-B$};
\node (6) at (-.57,1.38) {$\Sigma_1$};
\node (7) at (-.71,.71) {$C_6$};
\node (8) at (-1.38,.57) {$\Sigma_6$};
\node (9) at (-1,0) {$C_7$};
\node (10) at (-1.38,-.57) {$\Sigma_7$};
\node (11) at (-.71,-.71) {$C_8$};
\node (12) at (-.57,-1.38) {$\Sigma_8$};
\node (13) at (0,-1) {$B$};
\node (14) at (.57,-1.38) {$\Sigma_1$};
\node (15) at (.71,-.71) {$C_1$};
\node (16) at (1.38,-.57) {$\Sigma_2$};
\draw[-,thick] (12) to (4);
\draw[-,thick] (14) to (6);
\draw[-,thick,shorten >=.1pt] (10) to (0);
\draw[-,thick,shorten >=.1pt] (8) to (0);
\draw[-,thick,shorten >=.1pt] (2) to (0);
\draw[-,thick,shorten >=.1pt] (16) to (0);
\path (13) edge[bend right] node [pos=0.3, below] {$e_1$} (15);
\path (15) edge[bend right] node [pos=0.3, above] {$e_1^*$} (13);
\path (15) edge[bend right] node [pos=0.5, right] {$e_2$} (1);
\path (1) edge[bend right] node [pos=0.2,left] {$e_2^*$} (15);
\path (1) edge[bend right] node [pos=0.5, right] {$e_3$} (3);
\path (3) edge[bend right] node [pos=0.3, left] {$e_3^*$} (1);
\path (3) edge[bend right] node [pos=0.3,above] {$e_4$} (5);
\path (5) edge[bend right] node [pos=0.2,below] {$e_4^*$} (3);
\path (7) edge[bend left] node [pos=0.3,above] {$e_5$} (5);
\path (5) edge[bend left] node [pos=0.2,below] {$e_5^*$} (7);
\path (9) edge[bend left] node [pos=0.5,left] {$e_6$} (7);
\path (7) edge[bend left] node [pos=0.3,right] {$e_6^*$} (9);
\path (11) edge[bend left] node [pos=0.5,left] {$e_7$} (9);
\path (9) edge[bend left] node [pos=0.2,right] {$e_7^*$} (11);
\path (13) edge[bend left] node [pos=0.3,below] {$e_8$} (11);
\path (11) edge[bend left] node [pos=0.3,above] {$e_8^*$} (13);
\end{tikzpicture}}

\scalebox{0.3}{\begin{tikzpicture}[scale=2.5,->]
\node (0) at (0,0) {};
\node (1) at (1,0) {};
\node (2) at (.92,.38) {};
\node (3) at (.71,.71) {};
\node (4) at (.38,.92) {};
\node (5) at (0,1) {};
\node (6) at (-.38,.92) {};
\node (7) at (-.71,.71) {};
\node (8) at (-.92,.38) {};
\node (9) at (-1,0) {};
\node (10) at (-.92,-.38) {};
\node (11) at (-.71,-.71) {};
\node (12) at (-.38,-.92) {};
\node (13) at (0,-1) {};
\node (14) at (.38,-.92) {};
\node (15) at (.71,-.71) {};
\node (16) at (.92,-.38) {};

\path (13) edge[bend right]  (15);
\path (15) edge[bend right]   (13);
\path (15) edge[bend right]   (1);
\path (1) edge[bend right]   (15);
\path (1) edge[bend right]  (3);
\path (3) edge[bend right]  (1);
\path (3) edge[bend right]  (5);
\path (5) edge[bend right]  (3);
\path (7) edge[bend left]  (5);
\path (5) edge[bend left]  (7);
\path (9) edge[bend left]  (7);
\path (7) edge[bend left]  (9);
\path (11) edge[bend left] (9);
\path (9) edge[bend left]  (11);
\path (13) edge[bend left]  (11);
\path (11) edge[bend left] (13);
\fill[color=join,line width=1mm,draw=black,line join=round,opacity=.5] 
                   (13)  to [bend right] (15) 
                          to [bend right] (1)
                          to [bend right] (3)
						  to [bend right] (5)
						  to [bend right] (7)
						  to [bend right] (9)
						  to [bend right] (11)
						  to [bend right] (13) --cycle;
\end{tikzpicture}}
\scalebox{0.3}{\begin{tikzpicture}[scale=2.5,->]
\node (0) at (0,0) {};
\node (1) at (1,0) {};
\node (2) at (.92,.38) {};
\node (3) at (.71,.71) {};
\node (4) at (.38,.92) {};
\node (5) at (0,1) {};
\node (6) at (-.38,.92) {};
\node (7) at (-.71,.71) {};
\node (8) at (-.92,.38) {};
\node (9) at (-1,0) {};
\node (10) at (-.92,-.38) {};
\node (11) at (-.71,-.71) {};
\node (12) at (-.38,-.92) {};
\node (13) at (0,-1) {};
\node (14) at (.38,-.92) {};
\node (15) at (.71,-.71) {};
\node (16) at (.92,-.38) {};

\path (13) edge[bend right]  (15);
\path (15) edge[bend right]  (13);
\path (15) edge[bend right]  (1);
\path (1) edge[bend right] (15);
\path (1) edge[bend right]  (3);
\path (3) edge[bend right]  (1);
\path (3) edge[bend right]  (5);
\path (5) edge[bend right]  (3);
\path (7) edge[bend left]  (5);
\path (5) edge[bend left]  (7);
\path (9) edge[bend left]  (7);
\path (7) edge[bend left]  (9);
\path (11) edge[bend left]  (9);
\path (9) edge[bend left] (11);
\path (13) edge[bend left]  (11);
\path (11) edge[bend left]  (13);
\fill[color=join,line width=1mm,draw=black,line join=round,opacity=.5] 
                   (13)  to [bend left] (15) 
                          to [bend right] (1)
                          to [bend right] (3)
              to [bend right] (5)
              to [bend left] (7)
              to [bend right] (9)
              to [bend right] (11)
              to [bend right] (13) --cycle;
\end{tikzpicture}}
\scalebox{0.3}{\begin{tikzpicture}[scale=2.5,->]
\node (0) at (0,0) {};
\node (1) at (1,0) {};
\node (2) at (.92,.38) {};
\node (3) at (.71,.71) {};
\node (4) at (.38,.92) {};
\node (5) at (0,1) {};
\node (6) at (-.38,.92) {};
\node (7) at (-.71,.71) {};
\node (8) at (-.92,.38) {};
\node (9) at (-1,0) {};
\node (10) at (-.92,-.38) {};
\node (11) at (-.71,-.71) {};
\node (12) at (-.38,-.92) {};
\node (13) at (0,-1) {};
\node (14) at (.38,-.92) {};
\node (15) at (.71,-.71) {};
\node (16) at (.92,-.38) {};

\path (13) edge[bend right](15);
\path (15) edge[bend right](13);
\path (15) edge[bend right](1);
\path (1) edge[bend right](15);
\path (1) edge[bend right](3);
\path (3) edge[bend right](1);
\path (3) edge[bend right](5);
\path (5) edge[bend right](3);
\path (7) edge[bend left](5);
\path (5) edge[bend left](7);
\path (9) edge[bend left](7);
\path (7) edge[bend left](9);
\path (11) edge[bend left](9);
\path (9) edge[bend left](11);
\path (13) edge[bend left](11);
\path (11) edge[bend left](13);
\fill[color=join,line width=1mm,draw=black,line join=round,opacity=.5] 
                   (13)  to [bend left] (15) 
                          to [bend left] (1)
                          to [bend right] (3)
              to [bend right] (5)
              to [bend left] (7)
              to [bend left] (9)
              to [bend right] (11)
              to [bend right] (13) --cycle;
\end{tikzpicture}}
\scalebox{0.3}{\begin{tikzpicture}[scale=2.5,->]
\node (0) at (0,0) {};
\node (1) at (1,0) {};
\node (2) at (.92,.38) {};
\node (3) at (.71,.71) {};
\node (4) at (.38,.92) {};
\node (5) at (0,1) {};
\node (6) at (-.38,.92) {};
\node (7) at (-.71,.71) {};
\node (8) at (-.92,.38) {};
\node (9) at (-1,0) {};
\node (10) at (-.92,-.38) {};
\node (11) at (-.71,-.71) {};
\node (12) at (-.38,-.92) {};
\node (13) at (0,-1) {};
\node (14) at (.38,-.92) {};
\node (15) at (.71,-.71) {};
\node (16) at (.92,-.38) {};

\path (13) edge[bend right](15);
\path (15) edge[bend right](13);
\path (15) edge[bend right](1);
\path (1) edge[bend right](15);
\path (1) edge[bend right](3);
\path (3) edge[bend right](1);
\path (3) edge[bend right](5);
\path (5) edge[bend right](3);
\path (7) edge[bend left](5);
\path (5) edge[bend left](7);
\path (9) edge[bend left](7);
\path (7) edge[bend left](9);
\path (11) edge[bend left](9);
\path (9) edge[bend left](11);
\path (13) edge[bend left](11);
\path (11) edge[bend left](13);
\fill[color=join,line width=1mm,draw=black,line join=round,opacity=.5] 
                   (13)  to [bend left] (15) 
                          to [bend left] (1)
                          to [bend left] (3)
              to [bend right] (5)
              to [bend left] (7)
              to [bend left] (9)
              to [bend left] (11)
              to [bend right] (13) --cycle;
\end{tikzpicture}}

\scalebox{0.3}{\begin{tikzpicture}[scale=2.5,->]
\node (0) at (0,0) {};
\node (1) at (1,0) {};
\node (2) at (.92,.38) {};
\node (3) at (.71,.71) {};
\node (4) at (.38,.92) {};
\node (5) at (0,1) {};
\node (6) at (-.38,.92) {};
\node (7) at (-.71,.71) {};
\node (8) at (-.92,.38) {};
\node (9) at (-1,0) {};
\node (10) at (-.92,-.38) {};
\node (11) at (-.71,-.71) {};
\node (12) at (-.38,-.92) {};
\node (13) at (0,-1) {};
\node (14) at (.38,-.92) {};
\node (15) at (.71,-.71) {};
\node (16) at (.92,-.38) {};

\path (13) edge[bend right](15);
\path (15) edge[bend right](13);
\path (15) edge[bend right](1);
\path (1) edge[bend right](15);
\path (1) edge[bend right](3);
\path (3) edge[bend right](1);
\path (3) edge[bend right](5);
\path (5) edge[bend right](3);
\path (7) edge[bend left](5);
\path (5) edge[bend left](7);
\path (9) edge[bend left](7);
\path (7) edge[bend left](9);
\path (11) edge[bend left](9);
\path (9) edge[bend left](11);
\path (13) edge[bend left](11);
\path (11) edge[bend left](13);
\fill[color=join,line width=1mm,draw=black,line join=round,opacity=.5] 
                   (13)  to [bend left] (15) 
                          to [bend left] (1)
                          to [bend left] (3)
              to [bend left] (5)
              to [bend left] (7)
              to [bend left] (9)
              to [bend left] (11)
              to [bend left] (13) --cycle;
\end{tikzpicture}}
\scalebox{0.3}{\begin{tikzpicture}[scale=2.5,->]
\node (0) at (0,0) {};
\node (1) at (1,0) {};
\node (2) at (.92,.38) {};
\node (3) at (.71,.71) {};
\node (4) at (.38,.92) {};
\node (5) at (0,1) {};
\node (6) at (-.38,.92) {};
\node (7) at (-.71,.71) {};
\node (8) at (-.92,.38) {};
\node (9) at (-1,0) {};
\node (10) at (-.92,-.38) {};
\node (11) at (-.71,-.71) {};
\node (12) at (-.38,-.92) {};
\node (13) at (0,-1) {};
\node (14) at (.38,-.92) {};
\node (15) at (.71,-.71) {};
\node (16) at (.92,-.38) {};

\path (13) edge[bend right](15);
\path (15) edge[bend right](13);
\path (15) edge[bend right](1);
\path (1) edge[bend right](15);
\path (1) edge[bend right](3);
\path (3) edge[bend right](1);
\path (3) edge[bend right](5);
\path (5) edge[bend right](3);
\path (7) edge[bend left](5);
\path (5) edge[bend left](7);
\path (9) edge[bend left](7);
\path (7) edge[bend left](9);
\path (11) edge[bend left](9);
\path (9) edge[bend left](11);
\path (13) edge[bend left](11);
\path (11) edge[bend left](13);
\fill[color=join,line width=1mm,draw=black,line join=round,opacity=.5] 
                   (13)  to [bend right] (15) 
                          to [bend left] (1)
                          to [bend left] (3)
              to [bend left] (5)
              to [bend right] (7)
              to [bend left] (9)
              to [bend left] (11)
              to [bend left] (13) --cycle;
\end{tikzpicture}}
\scalebox{0.3}{\begin{tikzpicture}[scale=2.5,->]
\node (0) at (0,0) {};
\node (1) at (1,0) {};
\node (2) at (.92,.38) {};
\node (3) at (.71,.71) {};
\node (4) at (.38,.92) {};
\node (5) at (0,1) {};
\node (6) at (-.38,.92) {};
\node (7) at (-.71,.71) {};
\node (8) at (-.92,.38) {};
\node (9) at (-1,0) {};
\node (10) at (-.92,-.38) {};
\node (11) at (-.71,-.71) {};
\node (12) at (-.38,-.92) {};
\node (13) at (0,-1) {};
\node (14) at (.38,-.92) {};
\node (15) at (.71,-.71) {};
\node (16) at (.92,-.38) {};

\path (13) edge[bend right](15);
\path (15) edge[bend right](13);
\path (15) edge[bend right](1);
\path (1) edge[bend right](15);
\path (1) edge[bend right](3);
\path (3) edge[bend right](1);
\path (3) edge[bend right](5);
\path (5) edge[bend right](3);
\path (7) edge[bend left](5);
\path (5) edge[bend left](7);
\path (9) edge[bend left](7);
\path (7) edge[bend left](9);
\path (11) edge[bend left](9);
\path (9) edge[bend left](11);
\path (13) edge[bend left](11);
\path (11) edge[bend left](13);
\fill[color=join,line width=1mm,draw=black,line join=round,opacity=.5] 
                   (13)  to [bend right] (15) 
                          to [bend right] (1)
                          to [bend left] (3)
              to [bend left] (5)
              to [bend right] (7)
              to [bend right] (9)
              to [bend left] (11)
              to [bend left] (13) --cycle;
\end{tikzpicture}}
\scalebox{0.3}{\begin{tikzpicture}[scale=2.5,->]
\node (0) at (0,0) {};
\node (1) at (1,0) {};
\node (2) at (.92,.38) {};
\node (3) at (.71,.71) {};
\node (4) at (.38,.92) {};
\node (5) at (0,1) {};
\node (6) at (-.38,.92) {};
\node (7) at (-.71,.71) {};
\node (8) at (-.92,.38) {};
\node (9) at (-1,0) {};
\node (10) at (-.92,-.38) {};
\node (11) at (-.71,-.71) {};
\node (12) at (-.38,-.92) {};
\node (13) at (0,-1) {};
\node (14) at (.38,-.92) {};
\node (15) at (.71,-.71) {};
\node (16) at (.92,-.38) {};

\path (13) edge[bend right](15);
\path (15) edge[bend right](13);
\path (15) edge[bend right](1);
\path (1) edge[bend right](15);
\path (1) edge[bend right](3);
\path (3) edge[bend right](1);
\path (3) edge[bend right](5);
\path (5) edge[bend right](3);
\path (7) edge[bend left](5);
\path (5) edge[bend left](7);
\path (9) edge[bend left](7);
\path (7) edge[bend left](9);
\path (11) edge[bend left](9);
\path (9) edge[bend left](11);
\path (13) edge[bend left](11);
\path (11) edge[bend left](13);
\fill[color=join,line width=1mm,draw=black,line join=round,opacity=.5] 
                   (13)  to [bend right] (15) 
                          to [bend right] (1)
                          to [bend right] (3)
              to [bend left] (5)
              to [bend right] (7)
              to [bend right] (9)
              to [bend right] (11)
              to [bend left] (13) --cycle;
\end{tikzpicture}}

\caption{{\it Top:} the hyperplane arrangement for the dihedral group $I_2(4)$ has eight regions, and its four hyperplanes are cut into six shards.  {\it Bottom:} the eight 2-cells of $\Sal(\HH)$, indicated in blue.  One 2-cell is attached for each of the eight homotopies $e_1e_2e_3e_4 \cong e_8e_7e_6e_5, e_2e_3e_4e_5^*\cong e_1^*e_8e_7e_6, \ldots, e_8^*e_1e_2e_3 \cong e_7e_6e_5e_4^*$.}
\label{fig:rank2}
\end{figure}

Given regions $D,D'\in\RR$, we fix a positive minimal gallery $\gal(D',D)$ from $D'$ to $D$ in $\Sal(\HH)$; any two such galleries from $D'$ to $D$ are homotopic. If $e$ is the edge $C' \xrightarrow{e} C$, we define the corresponding loop $\tt_e\in\pi_1(\Sal(\HH),B)$ by
\begin{equation}\label{eq:tt}
\tt_e := \gal(B,C') \cdot e e^*\cdot \gal(B,C')^{-1}\in\pi_1(\Sal(\HH),B).
\end{equation}
Because of homotopies, this definition of the loop does not depend on the choice of the gallery. Write $\TT_{\mathrm{edge}}=\TT_{\mathrm{edge}}(\HH,B)$ for the set of all such loops $\tt_e$. By definition, the group $\pi_1(\PP(\HH),B)$ is generated by $\TT_{\mathrm{edge}}$.

\subsection{Shards \texorpdfstring{(\Cref{sec:shards})}{}}
\emph{Shards} are certain closed polyhedral subsets of the hyperplanes in $\HH$ that were introduced by Reading in \cite{reading2003order}; we write $\Sha(\HH,B)$ for the set of shards.   Each cover relation $C'\lessdot C$ in $\Weak(\HH,B)$ can be labeled by a shard $\Sigma(C'\lessdot C)$, which is the unique shard separating the region $C'$ from the region $C$; in this case, we call $\Sigma(C'\lessdot C)$ a \defn{lower shard} of $C$. Let $\covsha(C)$ be the set of lower shards of $C$.

Now assume $\HH$ is simplicial. Then $\Weak(\HH,B)$ is a semidistributive lattice~\cite{bjorner1990hyperplane}, and the set of shards forms an elegant geometric realization of the set of join-irreducible elements. Furthermore, shard intersections encode the canonical join representations of $\Weak(\HH,B)$. The map $C\mapsto \bigcap\covsha(C)$ defines a bijection from $\RR$ to the set of arbitrary intersections of shards. In~\cite{reading2011noncrossing}, Reading introduced another poset $\Shard(\HH,B):=(\RR,\preceq)$ called the \defn{shard intersection order}, which is defined by \[C\preceq D\quad\text{if and only if}\quad \bigcap \covsha(C) \supseteq\bigcap \covsha(D).\]

As with $\Weak(\HH,B)$, the poset $\Shard(\HH,B)$ is a lattice---but while $\Weak(\HH,B)$ is ``tall and slender'' (with height equal to the number of hyperplanes and with the number of atoms equal to the dimension), $\Shard(\HH,B)$ is ``short and wide'' (with height equal to dimension and with the number of atoms equal to the number of shards). When $\HH$ is the reflection arrangement of a finite Coxeter group $W$, the relationship between noncrossing partitions and sortable elements allowed Reading to embed the $W$-noncrossing partition lattice into the shard intersection order, thereby giving a uniform proof that the noncrossing partition lattice is indeed a lattice.   An example is illustrated in~\Cref{fig:shard}.

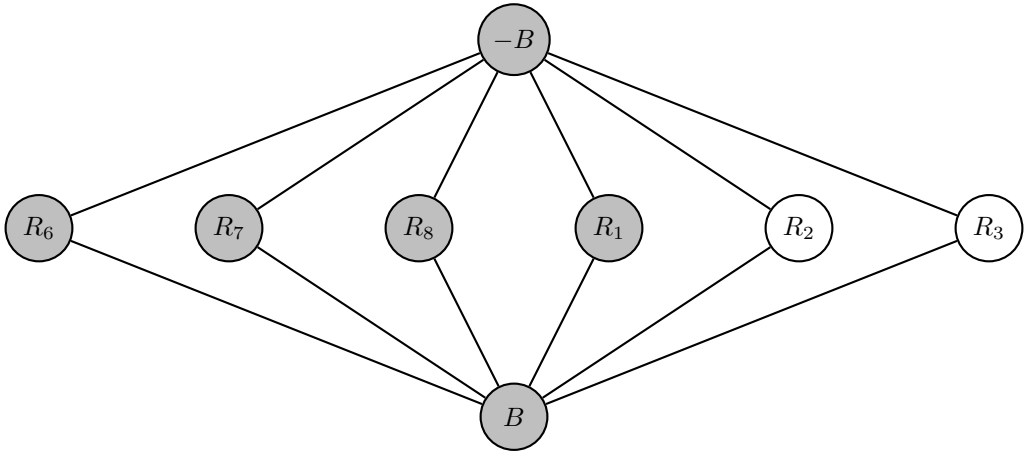
\begin{figure}[htbp]
\begin{tikzpicture}[scale=2.5]
\node[circle,draw,thick,minimum size=2.5em,thick,fill=lightgray] (B) at (0,0) {$B$};
\node[circle,draw,thick,minimum size=2.5em,thick,fill=lightgray] (R1) at (.5,1) {$R_1$};
\node[circle,draw,thick,minimum size=2.5em,] (R2) at (1.5,1) {$R_2$};
\node[circle,draw,thick,minimum size=2.5em,] (R3) at (2.5,1) {$R_3$};
\node[circle,draw,thick,minimum size=2.5em,fill=lightgray] (mB) at (0,2) {$-B$};\node[circle,draw,thick,minimum size=2.5em,fill=lightgray] (R6) at (-2.5,1) {$R_6$};
\node[circle,draw,thick,minimum size=2.5em,fill=lightgray] (R7) at (-1.5,1) {$R_7$};
\node[circle,draw,thick,minimum size=2.5em,fill=lightgray] (R8) at (-.5,1) {$R_8$};
\draw[-,thick] (B) to (R1) to (mB);
\draw[-,thick] (B) to (R2) to (mB);
\draw[-,thick] (B) to (R3) to (mB);
\draw[-,thick] (B) to (R6) to (mB);
\draw[-,thick] (B) to (R7) to (mB);
\draw[-,thick] (B) to (R8) to (mB);
\end{tikzpicture}
\caption{The poset $\Shard(\HH,B)$ for the arrangement of~\Cref{fig:rank2}, which is the reflection arrangement of the dihedral group $I_2(4)$.   Gray indicates elements in the image of Reading's embedding of the $I_2(4)$-noncrossing partition lattice (with respect to a certain Coxeter element); see \Cref{sec:noncrossing}.}
\label{fig:shard}
\end{figure}

\subsection{Shards and the Salvetti Complex \texorpdfstring{(\Cref{sec:shard_generators})}{}}
Whenever we have an edge $C'\xrightarrow{e} C$ in $\Sal(\HH)$, we will write $\Sigma(e)$ for the associated shard $\Sigma(C'\lessdot C)$.  Although the generators in $\TT_{\mathrm{edge}}$ are a priori indexed by cover relations in $\Weak(\HH,B)$, our first theorem says that they are really indexed by the much smaller set of shards.

\begin{theorem}\label{thm:shard_generators}
Let $\HH$ be a central real hyperplane arrangement. Given edges $C' \xrightarrow{e} C$ and $D' \xrightarrow{f} D$ in $\Sal(\HH)$, we have $\tt_{e} \simeq \tt_f$ if and only if $\Sigma(e) = \Sigma(f)$.
\end{theorem}

By~\Cref{thm:shard_generators}, it makes sense write $\TTsha=\TT_{\mathrm{edge}}$, indexing the loops in $\TT_{\mathrm{edge}}$ by shards.  Thus, for any shard $\Sigma \in \Sha(\HH,B)$, we define the \defn{shard loop} $\tt_\Sigma=\tt_e$, where $e$ is any edge such that $\Sigma(e)=\Sigma$. 

We define the \defn{pure shard monoid}, denoted $\PPP^+(\HH,B)$, to be the submonoid of $\pi_1(\PP(\HH),B)$ generated by $\TTsha$.  That is, an element of $\pi_1(\Sal(\HH),B)$ is in $\PPP^+(\HH,B)$ if and only if it can be represented by a word in the alphabet $\TTsha$. This allows us to endow $\PPP^+(\HH,B)$ with a partial order $\leq$ by declaring that $p\leq p'$ if there is a word over $\TTsha$ representing $p'$ that contains a prefix representing $p$. \Cref{fig:e_to_delta_P} shows the interval $[\id,\Delta^2]_{\PPP^+}$ between the identity element and the \emph{full twist} $\Delta^2$ (see \Cref{sec:pure_shard}) in the pure shard monoid of the arrangement from \Cref{fig:rank2}.  This interval is ``tall and wide,'' but it is \emph{not} a lattice in general (see \Cref{fig:intervala3}), preventing the use of Garside theory to study $\pi_1(\Sal(\HH),B)$.  

Our first main result about the pure shard monoid is the following. 

\begin{theorem}\label{thm:self-dual}
Let $\HH$ be a central arrangement. The interval $[\id,\Delta^2]_{\PPP^+}$ in the pure shard monoid is self-dual. 
\end{theorem}

In general, the combinatorics of the pure shard monoid can be quite involved. However, for a rank-2 arrangement with $m$ hyperplanes, we will prove the following precise theorem, which tells us that there are exactly $m2^{m-2}$ words over the alphabet $\TTsha$ representing $\Delta^2$.

\begin{theorem}\label{thm:rank2interval}
Let $\HH$ be hyperplane arrangement of rank $2$ with $m$ hyperplanes.  The interval $[\id,\Delta^2]_{\PPP^+}$ is a planar lattice with rank generating function $1+\left(\sum_{k=1}^{m-1} \left(2\binom{m}{k}-2\right)q^k \right)+q^m$ and with $m 2^{m-2}$ maximal chains. 
\end{theorem}

\subsection{Pop \texorpdfstring{(\Cref{sec:pop})}{}}\label{subsec:pop_intro}
In the next three subsections, we discuss three incarnations of the shard intersection order given by three maps $\Pop$, $\Crackle$, and $\Snap$. The first one, defined using the map $\Pop$, is not new, but it inspired our terminology for the other two. For these three subsections, we assume that $\HH$ is simplicial so that $\Weak(\HH,B)$ is a semidistributive lattice. 

Let $L$ be a locally finite meet-semilattice with meet operation denoted by $\wedge$. Motivated by work on \emph{pop-stacks} from enumerative combinatorics and theoretical computer science~\cite{Ungar,ClaessonPop,AlbertVatter}, the first author defined the \defn{pop-stack sorting operator} $\pop\colon L\to L$ in \cite{defant2022meeting} (see also \cite{defant2021stack}) by
\begin{equation}
    \pop(x):=x\wedge\bigwedge\{y\in L:y\lessdot x\}.
\end{equation}
In the case when $L=\Weak(\HH,B)$, we can use $\pop$ to characterize $\Shard(\HH,B)$. For $C\in\RR$, let $\Sigma([\pop(C),C])$ be the set of shards that label the cover relations in the interval $[\pop(C),C]$; that is,
\[\Sigma([\pop(C),C]):=\left\{\Sigma(D'\lessdot D):\pop(C)\leq D'\lessdot D\leq C\right\}.\]
By~\cite[Proposition 5.7]{reading2011noncrossing}, we have
\begin{equation}\label{eq:pop_shard_intersection}
    C'\preceq C\quad\text{if and only if}\quad \Sigma([\pop(C'),C'])\subseteq\Sigma([\pop(C),C]).
\end{equation}

\subsection{Crackle \texorpdfstring{(\Cref{sec:crackle})}{}}\label{subsec:crackle_intro}

We define the \defn{crackle map} $\Crackle\colon\RR\to \pi_1(\Sal(\HH),B)$ by
\begin{equation}\label{eq:crackle}    
\Crackle(C) := \gal(B,\pop(C))\cdot\gal(\pop(C),C)\cdot\gal(C,\pop(C))\cdot \gal(B,\pop(C))^{-1}.
\end{equation}
This map generalizes the shard loops $\tt_\Sigma$ of~\Cref{eq:tt} and~\Cref{thm:shard_generators}: if $J$ is a join-irreducible region of $\Weak(\HH,B)$, then $\pop(J)$ is the unique region covered by $J$, so $\Crackle(J)=\tt_{\Sigma(\pop(J)\lessdot J)}$. 

Just as~\Cref{eq:pop_shard_intersection} characterized $\Shard(\HH,B)$ using $\pop$, we can characterize $\Shard(\HH,B)$ using $\Crackle$. As above, we write $[\id,\Delta^2]_{\PPP^+}$ for the interval between the identity element $\id$ and the full twist $\Delta^2$ in the pure shard monoid $\PPP^+(\HH,B)$. Recall that if $P$ and $Q$ are posets, then a map $\psi\colon P\to Q$ is called a \defn{poset embedding} if it is a poset isomorphism from $P$ to its image $\psi(P)\subseteq Q$. 

\begin{theorem}\label{thm:crackle_hom}
The map $\Crackle$ is a poset embedding from $\Shard(\HH,B)$ into the interval $[\id,\Delta^2]_{\PPP^+}$.
\end{theorem}

\Cref{thm:crackle_hom} is illustrated in~\Cref{fig:e_to_delta_P,fig:intervala3}.

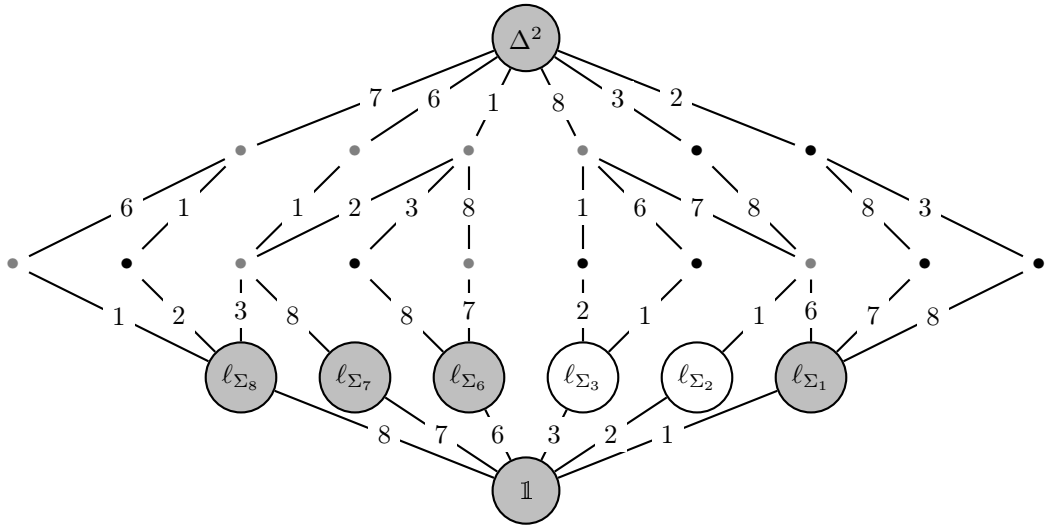
\begin{figure}[htbp]
\begin{tikzpicture}[scale=1.5]
\node[circle,draw,thick,minimum size=2.5em,fill=lightgray] (1) at (0,0) {$\id$};

\node[circle,draw,thick,minimum size=2.5em,fill=lightgray] (5) at (-.5,1) {$\tt_{\Sigma_6}$};
\node[circle,draw,thick,minimum size=2.5em,fill=lightgray] (11) at (-1.5,1) {$\tt_{\Sigma_7}$};
\node[circle,draw,thick,minimum size=2.5em,fill=lightgray] (9) at (-2.5,1) {$\tt_{\Sigma_8}$};
\node[circle,draw,thick,minimum size=2.5em,fill=lightgray] (14) at (2.5,1) {$\tt_{\Sigma_1}$};
\node[circle,draw,thick,minimum size=2.5em,fill=white] (23) at (1.5,1) {$\tt_{\Sigma_2}$};
\node[circle,draw,thick,minimum size=2.5em,fill=white] (19) at (.5,1) {$\tt_{\Sigma_3}$};

\node[gray] (4) at (-.5,2) {$\bullet$};
\node (6) at (-1.5,2) {$\bullet$};
\node[gray] (13) at (-2.5,2) {$\bullet$};
\node (7) at (-3.5,2) {$\bullet$};
\node[gray] (10) at (-4.5,2) {$\bullet$};
\node (17) at (4.5,2) {$\bullet$};
\node (15) at (3.5,2) {$\bullet$};
\node[gray] (24) at (2.5,2) {$\bullet$};
\node (21) at (1.5,2) {$\bullet$};
\node (18) at (.5,2) {$\bullet$};

\node[gray] (3) at (-.5,3) {$\bullet$};
\node[gray] (12) at (-1.5,3) {$\bullet$};
\node[gray] (8) at (-2.5,3) {$\bullet$};
\node (16) at (2.5,3) {$\bullet$};
\node (22) at (1.5,3) {$\bullet$};
\node[gray] (20) at (.5,3) {$\bullet$};

\node[circle,draw,thick,minimum size=2.5em,fill=lightgray] (2) at (0,4) {$\Delta^2$};

\draw[-,thick] (1) to node[midway, fill=white] {6} (5);
\draw[-,thick] (1) to node[midway, fill=white] {7} (11);
\draw[-,thick] (1) to node[midway, fill=white] {8} (9);
\draw[-,thick] (1) to node[midway, fill=white] {1} (14);
\draw[-,thick] (1) to node[midway, fill=white] {2} (23);
\draw[-,thick] (1) to node[midway, fill=white] {3} (19);

\draw[-,thick] (5) to node[midway, fill=white] {7} (4);
\draw[-,thick] (5) to node[midway, fill=white] {8} (6);
\draw[-,thick] (11) to node[midway, fill=white] {8} (13);
\draw[-,thick] (9) to node[midway, fill=white] {3} (13);
\draw[-,thick] (9) to node[midway, fill=white] {2} (7);
\draw[-,thick] (9) to node[midway, fill=white] {1} (10);
\draw[-,thick] (14) to node[midway, fill=white] {8} (17);
\draw[-,thick] (14) to node[midway, fill=white] {7} (15);
\draw[-,thick] (14) to node[midway, fill=white] {6} (24);
\draw[-,thick] (23) to node[midway, fill=white] {1} (24);
\draw[-,thick] (19) to node[midway, fill=white] {1} (21);
\draw[-,thick] (19) to node[midway, fill=white] {2} (18);

\draw[-,thick] (4) to node[midway, fill=white] {8} (3);
\draw[-,thick] (6) to node[midway, fill=white] {3} (3);
\draw[-,thick] (13) to node[midway, fill=white] {2} (3);
\draw[-,thick] (13) to node[midway, fill=white] {1} (12);
\draw[-,thick] (7) to node[midway, fill=white] {1} (8);
\draw[-,thick] (10) to node[midway, fill=white] {6} (8);
\draw[-,thick] (17) to node[midway, fill=white] {3} (16);
\draw[-,thick] (15) to node[midway, fill=white] {8} (16);
\draw[-,thick] (24) to node[midway, fill=white] {8} (22);
\draw[-,thick] (24) to node[midway, fill=white] {7} (20);
\draw[-,thick] (21) to node[midway, fill=white] {6} (20);
\draw[-,thick] (18) to node[midway, fill=white] {1} (20);

\draw[-,thick] (3) to node[midway, fill=white] {1} (2);
\draw[-,thick] (12) to node[midway, fill=white] {6} (2);
\draw[-,thick] (8) to node[midway, fill=white] {7} (2);
\draw[-,thick] (16) to node[midway, fill=white] {2} (2);
\draw[-,thick] (22) to node[midway, fill=white] {3} (2);
\draw[-,thick] (20) to node[midway, fill=white] {8} (2);
\end{tikzpicture}
\caption{The interval $[\id,\Delta^2]_{\PPP^+}$ in the pure shard monoid $\PPP^+(\HH,B)$ between the identity element and the full twist $\Delta^2$, where $(\HH,B)$ is as in~\Cref{fig:rank2}.  An edge $p \lessdot p'$ is labeled $i$ when $p'=p \cdot \tt_{\Sigma_i}$. Circled elements are in the image of $\Crackle$.  An element is colored gray if it appears as a prefix of a word for $\Delta^2$ using only generators corresponding to noncrossing shards; see \Cref{sec:noncrossing}.}
\label{fig:e_to_delta_P}
\end{figure}

\subsection{Snap \texorpdfstring{(\Cref{sec:snap})}{}}\label{subsec:snap_intro}
We now specialize to the case when $\HH$ is the reflection arrangement of a finite Coxeter group $W$.  We identify the base region $B$ with the identity element of $W$; the free transitive action of $W$ on $\RR$ then allows us to identity regions of $\HH$ with elements of $W$. We write ${\Shard(W):=\Shard(\HH,B)}$, $\Weak(W):=\Weak(\HH,B)$, $\PPP^+(W)=\PPP^+(\HH,B)$, etc.

In this setting, the group $\PPP(W) := \pi_1(\mathbb C^n\setminus \HH_{\mathbb C},x_B)$ is called the \defn{pure braid group} of $W$, while the group $\BB(W) := \pi_1((\mathbb C^n\setminus \HH_{\mathbb C})/W,x_B)$ is called the \defn{braid group} of $W$. The Coxeter group $W$ fits into the following well-known exact sequence with its braid and pure braid groups:
\[1 \to \PPP(W) \to \BB(W) \xrightarrow{\varphi} W \to 1.\]  Let ${\bf S}$ be the set of simple generators of $\BB(W)$ obtained by lifting the set $S$ of simple reflections of $W$. The generators in $\bf S$ satisfy the same braid relations as the corresponding simple reflections of $W$; the difference is that $W$ also includes the relations stating that the simple reflections are involutions. Thus, the projection $\varphi\colon \BB(W)\to W$ is the quotient map that sends each generator ${\bf s}\in{\bf S}$ to the corresponding $s\in S$ and imposes these additional relations. The submonoid of $\BB(W)$ generated by ${\bf S}$ is called the \defn{positive braid monoid} of $W$ and is denoted by $\BB^+(W)$.  The weak order $(W,\leq)$ is defined by saying $u \leq v$ if and only if any reduced word for $u$ appears as a prefix of some reduced word for $v$.  Analogously, the weak order $\Weak(\BB^+(W)):=(\BB^+(W),\leq)$ is defined by saying ${\bf u} \leq \mathbf{v}$ if and only if any word over $\bf S$ representing $\bf u$ appears as a prefix of some word over $\bf S$ representing $\mathbf{v}$. In this setting, the full twist $\Delta^2$ is equal to the lift $\w_\circ^2$ of the long element $w_\circ$ of $W$; we will also write $\Delta=\w_\circ$.

We can rephrase \Cref{thm:shard_generators} when $\HH$ is the reflection arrangement of $W$ as follows.
\begin{corollary}\label{thm:shard_generators2}
Suppose $u,v\in W$ and $s,t\in S$ are such that $u \lessdot us$ and $v \lessdot vt$.  Let ${\bf u}, {\bf v}, {\bf s}, {\bf t}$ be the lifts of $u,v,s,t$, respectively, to $\BB^+(W)$. We have ${\bf u} \s {\bf u}^{-1} = \mathbf{v} \mathbf{t} \mathbf{v}^{-1}$ if and only if $\Sigma(u\lessdot us) = \Sigma(v \lessdot vt)$.
\end{corollary}

For $u,v\in W$, we let $u^v=v^{-1}uv$. Write $\mathbf{c}=(\s_1,\ldots,\s_n)$ for an ordering of the elements of $\bf S$. For $\s \in{\bf S}$ and a positive braid $\w=\s_{i_1} \cdots \s_{i_k}\in \BB^+(W)$ with projection $w=\varphi(\w) \in W$, let $\s^\w = (t,j)$, where $t=s^w$ and $j$ counts the number of times $t$ appears in the sequence $s^{s_{i_k}},s^{s_{i_k}s_{i_{k-1}}},\ldots,s^{s_{i_k}s_{i_{k-1}}\cdots s_1}$.

In~\cite{galashin2022rational}, a new set of noncrossing $W$-Catalan objects was introduced as the set of subwords of $\mathbf{c}^{h+1}$ that represent the full twist $\Delta^2=\w_\circ^2$ and satisfy an additional \emph{Deodhar condition}. When interpreted in the positive braid monoid $\BB^+(W)$, this Deodhar condition is equivalent to restricting to those $c$-sortable elements $\w$ in the interval $[\id,\w_\circ^2]_{\BB^+}$ with the property that for each descent $\s$ of $\w$, we have $\s^\w = (t,j)$ with $j$ even---this is a nonstandard \emph{Deodhar embedding} of the $c$-sortable elements into the interval $[\id,\w_\circ^2]_{\BB^+}$ (different from the usual lift of $W$ into $\BB^+(W)$).  The second author speculated that restricting the $2$nd $c$-Fuss--Cambrian lattice to the image of this Deodhar embedding would recover the noncrossing partition lattice $\NC(W,c)$.  As the $2$nd $c$-Fuss--Cambrian lattice is a subposet of $\Weak(\BB^+(W))$, it makes sense to generalize this Deodhar embedding of $c$-sortable elements to all elements of $W$.

The \defn{snap map} $\Snap\colon W\to \BB^+(W)$ is our generalization of the Deodhar embedding.  We write $\pop$ for the pop-stack sorting operators on the lattice $\Weak(W)$ and the meet-semilattice $\Weak(\BB^+(W))$, relying on the argument of the operator to indicate the context. For $w\in W$, let $\des(w)$ denote the right descent set of $w$, let $w_\circ(\des(w))$ be the longest element of the parabolic subgroup of $W$ generated by $\des(w)$, and write $\w$ and $\w_\circ(\des(w))$ for the usual lifts of $w$ and $w_\circ(\des(w))$ to $\BB^+(W)$.  Define
\begin{equation}\label{eq:snap}\Snap(w):=\pop(\w) \cdot (\w_\circ(\des(w)))^2.\end{equation}  Since $\Crackle(w)=\pop(\w) \cdot (\w_\circ(\des(w)))^2 \cdot \pop(\w)^{-1}$ in $\PPP^+(W) \subseteq \BB(W)$, it follows that
\[\Snap(w)=\Crackle(w)\pop(\w).\]

\begin{figure}[htbp]
\begin{tikzpicture}[scale=1.5]
\node[circle,draw,thick,minimum size=2.5em,fill=lightgray] (e) at (0,0) {$\id$};

\node (s) at (-1,1) {$\s$};
\node[circle,draw,thick,minimum size=2.5em,fill=lightgray] (ss) at (-2,2) {$\s\s$};
\node (sst) at (-3,3) {$\s\s\t$};
\node (ssts) at (-4,4) {$\s\s\t\s$};

\node (st) at (-1,2) {$\s\t$};
\node[circle,draw,thick,minimum size=2.5em,fill=lightgray] (stt) at (-2,3) {$\s\t\t$};
\node (stts) at (-3,4) {$\s\t\t\s$};
\node (sttst) at (-4,5) {$\s\t\t\s\t$};

\node (sts) at (-1,3) {$\s\t\s$};
\node[circle,draw,thick,minimum size=2.5em,fill=lightgray] (stss) at (-2,4) {$\s\t\s\s$};
\node (stsst) at (-3,5) {$\s\t\s\s\t$};
\node (stssts) at (-4,6) {$\s\t\s\s\t$};

\node (d) at (0,4) {$\Delta$};
\node (ds) at (-1,5) {$\Delta \s$};
\node (dst) at (-1,6) {$\Delta \s\t$};
\node (dsts) at (-1,7) {$\Delta \s\t\s$};
\node[circle,draw,thick,minimum size=2.5em,fill=lightgray] (dd) at (0,8) {$\Delta^2$};
\node (dt) at (1,5) {$\Delta \t$};
\node (dts) at (1,6) {$\Delta \t\s$};
\node (dtst) at (1,7) {$\Delta \t\s\t$};

\node (t) at (1,1) {$\t$};
\node[circle,draw,thick,minimum size=2.5em,fill=lightgray] (tt) at (2,2) {$\t\t$};
\node (tts) at (3,3) {$\t\t\s$};
\node (ttst) at (4,4) {$\t\t\s\t$};

\node (ts) at (1,2) {$\t\s$};
\node[circle,draw,thick,minimum size=2.5em] (tss) at (2,3) {$\t\s\s$};
\node (tsst) at (3,4) {$\t\s\s\t$};
\node (tssts) at (4,5) {$\t\s\s\t\s$};

\node (tst) at (1,3) {$\t\s\t$};
\node[circle,draw,thick,minimum size=2.5em] (tstt) at (2,4) {$\t\s\t\t$};
\node (tstts) at (3,5) {$\t\s\t\t\s$};
\node (tsttst) at (4,6) {$\t\s\t\t\s\t$};

\draw[-,thick] (e) to (t) to (ts) to (tst) to (d) to (ds) to (dst) to (dsts) to (dd);
\draw[-,thick] (e) to (s) to (st) to (sts) to (d) to (dt) to (dts) to (dtst) to (dd);
\draw[-,thick] (s) to (ss) to (sst) to (ssts) to (ds);
\draw[-,thick] (ts) to (tss) to (tsst) to (tssts) to (dts);
\draw[-,thick] (sts) to (stss) to (stsst) to (stssts) to (dsts);
\draw[-,thick] (t) to (tt) to (tts) to (ttst) to (dt);
\draw[-,thick] (st) to (stt) to (stts) to (sttst) to (dst);
\draw[-,thick] (tst) to (tstt) to (tstts) to (tsttst) to (dtst);
\end{tikzpicture}
\caption{The interval $[\id,\Delta^2]_{\BB^+}$ in $\Weak(\BB^+(I_2(4)))$, where $I_2(4)$ is the dihedral group of order $8$ with simple reflections $s$ and $t$. The reflection arrangement of $I_2(4)$ is shown in~\Cref{fig:rank2}.  Circled elements are in the image of $\Snap$.  Gray indicates that the element is the image of a $c$-sortable element under $\Snap$, where $c=st$; see~\Cref{sec:noncrossing}.}
\label{fig:e_to_delta_B}
\end{figure}
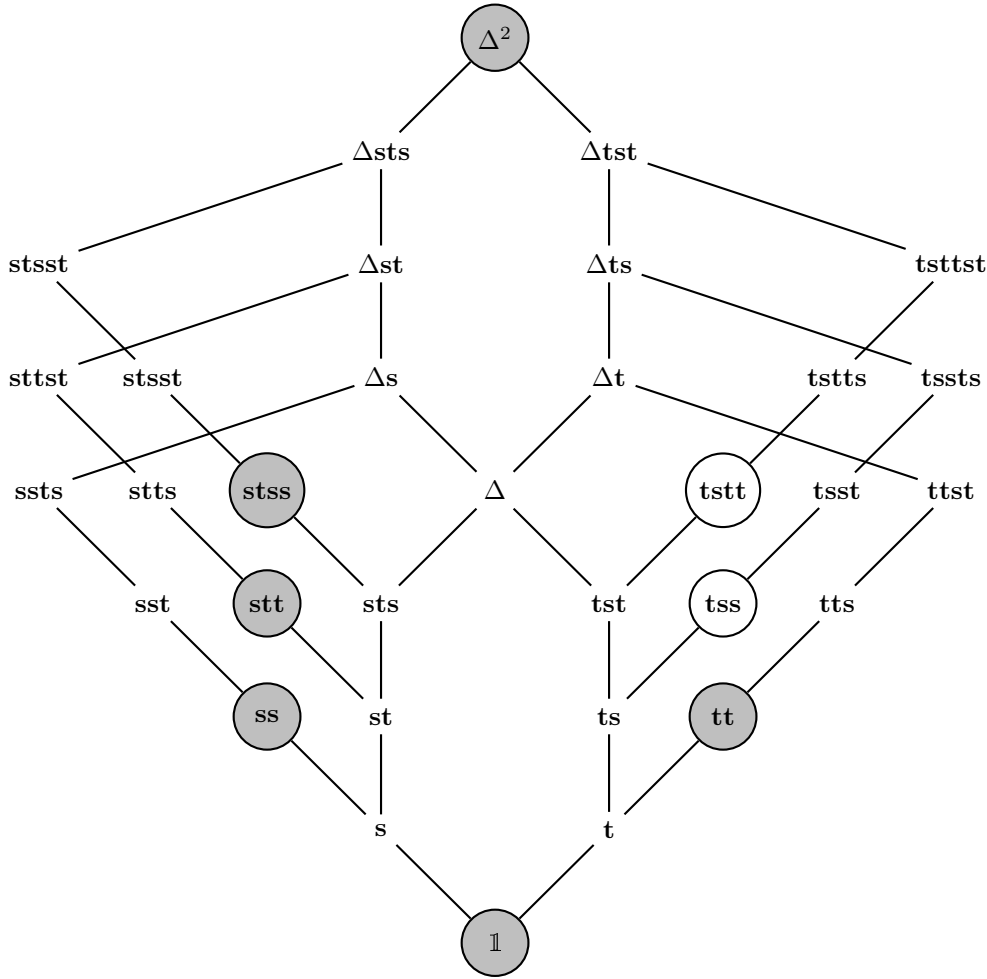

Just as the shard intersection order was characterized via $\pop$ in~\Cref{eq:pop_shard_intersection} and via $\Crackle$ in~\Cref{thm:crackle_hom}, it is also characterized via $\Snap$.

\begin{theorem}\label{thm:snap}
The map $\Snap$ is a poset embedding from $\Shard(W)$ into $[\id,\Delta^2]_{\BB^+}$.
\end{theorem}

\Cref{thm:snap} is illustrated in \Cref{fig:e_to_delta_B}. 

\subsection{Pow \texorpdfstring{(\Cref{sec:pow})}{}}
In her history ``The Untold Tale of Pow!, the Fourth Rice Krispies Elf: A look into the era when the cereal mascots were more than just Snap!, Crackle! and Pop!''~\cite{smith}, Smith writes: 
\begin{quotation}
``Lost in the shuffle, however, was a fourth Rice Krispies elf named Pow! His short life is a time-capsule of an era when everyone was dreaming big.''
\end{quotation}
Inspired by this fourth mascot, we introduce a map $\Pow: \RR \to \PPP^+(\HH,B)$ in the general setting when $\HH$ is a central irreducible real hyperplane arrangement. Given a region $C\in\RR$ and a positive minimal gallery \[B = C_0 \xrightarrow{e_1} C_{1} \xrightarrow{e_{2}} \cdots \xrightarrow{e_{k-1}} C_{k-1}\xrightarrow{e_k}C_k=C,\] we let \[\Pow(C):=\tt_{\Sigma(e_k)}\tt_{\Sigma(e_{k-1})}\cdots \tt_{\Sigma(e_1)}.\] 
(We prove in \Cref{prop:pow_well_defined} that $\Pow$ is well-defined.)  Just as $\Crackle$ embeds the ``short and wide'' poset $\Shard(\HH,B)$ into the ``tall and wide'' interval $[\id,\Delta^2]_{\PPP^+}$, the map $\Pow$ embeds the ``tall and slender'' poset $\Weak(\HH,B)$ into $[\id,\Delta^2]_{\PPP^+}$. 

\begin{theorem}\label{thm:pow_hom}
The map $\Pow$ is a poset embedding from $\Weak(\HH,B)$ into $[\id,\Delta^2]_{\PPP^+}$.
\end{theorem}

\Cref{thm:pow_hom} is illustrated in \Cref{fig:e_to_delta_P2}. 

\subsection{Future Work \texorpdfstring{(\Cref{sec:future})}{}}
The final section of the paper proposes some ideas for further investigation. 

\section{The Salvetti Complex}\label{sec:salvetti}

\subsection{Real Hyperplane Arrangements}
We recall some definitions regarding real hyperplane arrangements, referring the reader to~\cite{stanley2004introduction} for additional details.

A real \defn{hyperplane} is an affine subspace of $\mathbb R^n$ of codimension 1, and a \defn{hyperplane arrangement} (or just \defn{arrangement}, for short) is a set of distinct hyperplanes in $\mathbb R^n$.  A hyperplane is \defn{linear} if it contains the origin, and an arrangement is \defn{central} if all of its hyperplanes are linear. A \defn{subarrangement} of an arrangement $\HH$ is simply a subset of $\HH$. We say $\HH$ is \defn{irreducible} if there do not exist a linear automorphism $\theta$ of $\mathbb R^n$ and a partition $\HH=\HH_1\sqcup\HH_2$ with $\HH_1,\HH_2\neq\emptyset$ such that all of the normal vectors of hyperplanes in $\theta(\HH_1)$ are orthogonal to all of the normal vectors of hyperplanes in $\theta(\HH_2)$. Throughout this paper, we always assume that $\HH$ is a finite central irreducible real hyperplane arrangement in $\mathbb R^n$; note that this forces $\HH$ to contain at least $3$ hyperplanes. 

The \defn{rank} of an arrangement (or subarrangement) is the codimension of the intersection of its hyperplanes.    We call a subarrangement $\AA \subseteq \HH$ \defn{full} if it contains all hyperplanes of $\HH$ containing a particular subset of $\mathbb R^n$.   

A \defn{region} (or \defn{chamber}) of $\HH$ is a connected component of $\mathbb R^n \setminus \bigcup\HH$; we write $\RR=\RR_\HH$ for the set of regions of $\HH$.  We will fix a choice of a \defn{base region} $B$.  A \defn{bounding hyperplane} of a region $C\in\RR$ is a hyperplane in $\HH$ whose intersection with the boundary of $C$ is $(n-1)$-dimensional.  We say $\HH$ is \defn{simplicial} if every region has exactly $n$ bounding hyperplanes.

\subsection{The Poset of Regions}
Let $C, C' \in \RA$ be two regions of $\HH$, and choose points $x \in C$ and $x' \in C'$.  The hyperplanes intersecting the line segment with endpoints $x$ and $x'$ are said to \defn{separate} $C$ from $C'$, and we write $\inv(C,C')$ for the set of all hyperplanes separating $C$ from $C'$. We write $-C$ for the unique region such that $\inv(C,-C)=\HH$.  Let $\inv(C)=\inv(B,C)$, where $B$ is our fixed base region.  The \defn{poset of regions} of $\HH$ is the poset $\Weak(\HH,B):=(\RR,\leq)$, where we write $C\leq C'$ if $\inv(C)\subseteq\inv(C')$.

\begin{theorem}[{\cite[Theorem 3.4]{bjorner1990hyperplane}}]\label{thm:simplicial_lattice}
If $\HH$ is a simplicial hyperplane arrangement, then the poset of regions $\Weak(\HH,B)$ is a lattice.
\end{theorem}

We write $C' \xrightarrow{e} C$ when $e$ is the edge in the Hasse diagram of $\Weak(\HH,B)$ corresponding to a cover relation $C'\lessdot C$.  In this case, we write $H_e$ for the unique hyperplane in $\inv(C)\setminus \inv(C')$, and we call $H_e$ a \defn{lower cover} of $C$. Let $\cov(C)$ be the set of lower covers of $C$.

\subsection{The Salvetti Complex}

Let $\HH = \{H_i\}_{i=1}^N$. The \defn{Salvetti complex} is a combinatorial model for the complement in $\mathbb C^n$ of the complexified hyperplane arrangement $\HH_\CC$~\cite{salvetti1987topology}. Choose $\alpha_i$ to be a normal vector to $H_i$. We write $Z(\HH)$ for the \defn{dual zonotope} of $\HH$, defined as the Minkowski sum $\alpha_1+\cdots+\alpha_N$.  As a polytope, the zonotope $Z(\HH)$ comes equipped with the structure of a CW complex.  For $C \in \RR$, write $Z(\HH,C)$ for an oriented copy of $Z(\HH)$, where each face of dimension at least $1$ is given an orientation pointing from $C$ to $-C$.  The Salvetti complex is then the CW complex defined as the union of the CW complexes $Z(\HH,C)$ over all $C \in \RR$:
\begin{equation}
    \Sal(\HH) := \bigcup_{C \in \RR} Z(\HH,C).
\end{equation}
This is not a disjoint union: for any $C,C'\in\RR$, we identify faces in $Z(\HH,C)$ and $Z(\HH,C')$ that have equivalent orientations. Vertices of $\Sal(\HH)$ correspond to regions of $\HH$, so it makes sense to think of our fixed base region $B$ as a point in $\Sal(\HH)$. 

\begin{theorem}[{\cite{salvetti1987topology}}]
    If $\HH$ is a central real arrangement, then $\Sal(\HH)$ embeds in $\CC^n \setminus \HH_\CC$ and is a deformation retract of $\CC^n \setminus \HH_\CC$.  Furthermore, \[ \pi_1(\Sal(\HH),B)= \pi_1(\CC^n \setminus \HH_\CC,x_B),\] where $B \in \RR$ and $x_B\in B$. 
\label{thm:salvetti_kpi1}
\end{theorem}

Salvetti further described the abelianization $H_1(\PP(\HH),B)$ of the group $\pi_1(\PP(\HH),B)$ (see also~\Cref{thm:gens_to_other_shards} below).

\begin{theorem}[{\cite{salvetti1987topology}}]  \label{thm:abelianization}
  The abelianization $H_1(\PP(\HH),B)$ of $\pi_1(\PP(\HH),B)$ is a free abelian group with generating set $\{\overline{\tt}_H\}_{H \in \HH}$ indexed by the hyperplanes in $\HH$. The natural homomorphism $\pi_1(\PP(\HH),B) \to H_1(\PP(\HH),B)$ sends $\tt_e$ to $\overline{\tt}_{H_e}$.
\end{theorem}

A direct description of the restriction of $\Sal(\HH)$ to its 1-skeleton can be built directly from the Hasse diagram of $\Weak(\HH,B)$ (see~\cite[Part Two]{salvetti1987topology}).  The vertices are given by the regions of $\HH$.  The edges are given by the cover relations of the Hasse diagram of $\Weak(\HH,B)$ as follows.  For each cover relation $C' \lessdot C$ in $\Weak(\HH,B)$, introduce an oriented edge $e$ from $C'$ to $C$ and an oriented edge $e^*$ from $C$ to $C'$.   We extend the notation $C' \xrightarrow{e} C$ to accommodate these extra (oriented) edges using $C \xrightarrow{e^*} C'$. Note that $e^*$ is not the same as $e^{-1}$.  We will always use un-starred letters for edges oriented away from $B$ and starred letter for edges oriented toward $B$.

Given arbitrary regions $C,C'\in\RA$, we define a \defn{gallery} from $C'$ to $C$ to be a sequence of edges (of the form $e$, $e^*$, $e^{-1}$, and $(e^*)^{-1}$) that starts at $C'$ and ends at $C$.  A gallery from $C'$ to $C$ is \defn{positive} if it only uses edges of the form $e$ and $e^*$ (not $e^{-1}$ or $(e^*)^{-1}$). The gallery is \defn{minimal} if its length is equal to $|\inv(C',C)|$, and it is called a \defn{loop} if $C=C'$.  If $\alpha$ and $\beta$ are galleries such that the ending point of $\alpha$ is the starting point of $\beta$, then $\alpha\cdot\beta$ is the gallery that we traverse by first traversing $\alpha$ and then traversing $\beta$.  
We say two galleries $\alpha,\beta$ in $\Sal(\HH)$ are \defn{homotopic} and write $\alpha \cong \beta$ if $\alpha$ can be obtained from $\beta$ by repeated insertion or deletion of boundary paths of 1-cells and 2-cells in $\Sal(\HH)$.  A gallery is minimal if and only if it has the shortest length among all galleries in its homotopy class.  Any two positive minimal galleries from a region $C'$ to a region $C$ are homotopic, and we write $\gal(C',C)$ for a choice of one such gallery.  We abbreviate the concatenation of a sequence of positive minimal galleries with the notation $\gal(C_1,C_2,C_3\ldots,C_{k-1},C_k)=\gal(C_1,C_2)\cdot \gal(C_2,C_3)\cdot \cdots\cdot\gal(C_{k-1},C_k)$.

\subsection{Generators and Relations of the Fundamental Group}\label{subsec:generators_and_relations}

This subsection follows~\cite[Part Two]{salvetti1987topology}. If $C' \xrightarrow{e} C$, we define the corresponding loop $\tt_e$ by \[\tt_e := \gal(B,C') \cdot e e^* \cdot \gal(B,C')^{-1}\in\pi_1(\Sal(\HH),B).\]
 Note that $\tt_e$ does not depend on the choice of the minimal gallery $\gal(B,C')$. Write $\TT_{\mathrm{edge}}$ for the set of all such loops $\tt_e$; this set generates $\pi_1(\PP(\HH),B)$, but it is redundant.  It turns out to be enough to reduce to a single generator from each hyperplane, as we now explain. 

Fix a positive minimal gallery \[\galb=\left(B=C_0 \xrightarrow{e_1} C_1 \xrightarrow{e_2} C_2 \xrightarrow{e_3} \cdots \xrightarrow{e_N} C_N=-B\right)\] from $B$ to $-B$ in $\Sal(\HH)$.  Writing $\bb_i:=\tt_{e_i}$ and $H_i:=H_{e_i}$ specifies the subset $\TT_\galb=\{\bb_i\}_{i=1}^N \subseteq \TT_{\mathrm{edge}}$
and totally orders the hyperplanes by \[H_1 <_\galb H_2 <_\galb \cdots <_\galb H_N.\]

\begin{example}
\label{ex:rank2_1}
	Continuing the example drawn in~\Cref{fig:rank2}, we fix \[\galb=\left(B \xrightarrow{e_1} C_1 \xrightarrow{e_2} C_2 \xrightarrow{e_3} C_3 \xrightarrow{e_4} -B\right),\] which gives the loops in $\TT_\galb$:
	\begin{align*}
		\bb_1&=\tt_{e_1}=e_1e_1^*,& 
		\bb_2&=\tt_{e_2}=e_1e_2e_2^*e_1^{-1},\\
		\bb_3&=\tt_{e_3}=e_1e_2e_3e_3^*e_2^{-1}e_1^{-1}, &
		\bb_4&=\tt_{e_4}=e_1e_2e_3e_4e_4^*e_3^{-1}e_2^{-1}e_1^{-1}.
	\end{align*} 
\end{example}

The next theorem shows that any generator in $\TT_{\mathrm{edge}}$ is conjugate to an element of $\TT_\galb$.

\begin{theorem}[{\cite[Lemma 12, Corollary 12]{salvetti1987topology}}]\label{thm:gens_to_other_shards}
For any positive minimal gallery $\galb$ from $B$ to $-B$, $\TT_\galb$ is a generating set of $\pi_1(\PP(\HH),B)$.  Specifically, for any edge $D' \xrightarrow{e} D$, there is a unique edge $C_{k-1} \xrightarrow{e_k} C_k$ such that $H_{e_k}=H_e$, and we have \[\tt_e = \left(\bb_{i_s}\bb_{i_{s-1}}\cdots\bb_{i_1} \right)^{-1} \bb_k\left(\bb_{i_s}\bb_{i_{s-1}}\cdots\bb_{i_1}   \right), \] where $i_1<\cdots<i_s$ are the elements $i$ of $\{1,\ldots,k-1\}$ such that $H_i\not\in\inv(D)$.
\end{theorem}

\begin{example}\label{ex:rank2_2}
Continuing~\Cref{ex:rank2_1}, we have
\begin{align*}
	\tt_{e_5}&=\bb_{1}, &
	\tt_{e_6}&=\bb_{1}^{-1}\bb_{2}\bb_{1}, \\
	\tt_{e_7}&=\bb_{1}^{-1}\bb_{2}^{-1}\bb_{3}\bb_{2}\bb_{1}, &
	\tt_{e_8}&=\bb_{1}^{-1}\bb_{2}^{-1}\bb_{3}^{-1}\bb_{4}\bb_{3}\bb_{2}\bb_{1}=\bb_{4},
\end{align*}
where the identity $\bb_{1}^{-1}\bb_{2}^{-1}\bb_{3}^{-1}\bb_{4}\bb_{3}\bb_{2}\bb_{1}=\bb_{4}$ follows from \Cref{thm:relations} below.
\end{example}

Having found a set $\TT_\galb$ of generators for the group $\pi_1(\Sal(\HH),B)$, we now need relations.  To present $\pi_1(\Sal(\HH),B)$, we need one family of relations for each full rank-2 subarrangement $\AA$ of $\HH$. 

\begin{theorem}[{\cite[Page 616]{salvetti1987topology}}]
Fix a positive minimal gallery $\galb$ from $B$ to $-B$ in $\PP(\HH)$.  For each full rank-2 subarrangement $\AA$ of $\HH$, choose one 2-cell in $\PP(\HH)$ with edges named $e_1,e_2,\ldots,e_m$ and $e_{m+1},\ldots,e_{2m}$ as in~\Cref{fig:rank2} so that $\{H_{e_i}\}_{i=1}^m = \AA$ and $H_{e_1} <_\galb \cdots <_\galb H_{e_m}$.  Write $[\AA]_{\TT_\galb}$ for the family of relations \[\tt_{e_m}\cdots \tt_{e_2}\tt_{e_1}=\tt_{e_1}\tt_{e_m}\cdots\tt_{e_2}=\cdots=\tt_{e_{m{-}1}} \cdots \tt_{e_2}\tt_{e_1}\tt_{e_m},\] where each $\tt_{e_i}$ has been rewritten using the elements of $\TT_{\galb}$ using~\Cref{thm:gens_to_other_shards}.  Then \[\pi_1(\PP(\HH),B) = \left\langle \TT_{\galb} : [\AA]_{\TT_\galb} \right\rangle,\] where the relations range over all full rank-2 subarrangements $\AA \subseteq \HH$.
\label{thm:relations}
\end{theorem}

\begin{example}\label{ex:rank2_3}
	Continuing~\Cref{ex:rank2_1}, we have the relations
\[\bb_4 \bb_3 \bb_2 \bb_1 = \bb_1 \bb_4 \bb_3 \bb_2  = \bb_2 \bb_1 \bb_4 \bb_3  = \bb_3 \bb_2 \bb_1 \bb_4;\] by~\Cref{thm:gens_to_other_shards}, these relations imply the further relations \[\tt_{e_5} \tt_{e_6} \tt_{e_7} \tt_{e_8} = \tt_{e_8} \tt_{e_5} \tt_{e_6} \tt_{e_7} =  \tt_{e_7} \tt_{e_8} \tt_{e_5} \tt_{e_6}=\tt_{e_6}\tt_{e_7}\tt_{e_8}\tt_{e_5}.\]
\end{example}

We end this section with a corollary of \Cref{thm:gens_to_other_shards} that we will find useful in subsequent sections. 

\begin{corollary}\label{lem:reverse_order_pow}
Let $C \in \RR$. Consider a positive minimal gallery \[B=C_0\xrightarrow{e_1}C_{1}\xrightarrow{e_{2}} \cdots \xrightarrow{e_{k-1}}C_{k-1}\xrightarrow{e_k}C_k= C\] from $B$ to $C$, and let \[-C=-C_k\xrightarrow{f_k} -C_{k-1}\xrightarrow{f_{k-1}}\cdots \xrightarrow{f_{2}}-C_{1}\xrightarrow{f_1} -C_0= -B\] be the corresponding positive minimal gallery from $-C$ to $-B$.  Then 
\[\tt_{e_k}\tt_{e_{k-1}}\cdots \tt_{e_1}=\tt_{f_1}\tt_{f_{2}}\cdots \tt_{f_k}.\]
\end{corollary}
\begin{proof}
We can extend the given positive minimal gallery from $B$ to $C$ to a positive minimal gallery $\delta$ from $B$ to $-B$. In the notation of \Cref{thm:gens_to_other_shards}, we have $\delta_i=\tt_{e_{i}}$ for all $1\leq i\leq k$. Therefore, that theorem tells us that \[\tt_{f_{i}}=\big(\tt_{e_1}^{-1}\cdots \tt_{e_{i-1}}^{-1}\big)\tt_{e_{i}} \big(\tt_{e_{i-1}} \cdots \tt_{e_1}\big)\] for all $1\leq i\leq k$. The result follows by multiplying and canceling terms. 
\end{proof}

\section{Shards}
\label{sec:shards}
We recall some constructions and results from~\cite{reading2003order,reading2011noncrossing}.   As before, fix a base region $B \in \RA$.  For a full rank-2 subarrangement $\AA$ of $\HH$, let $B_\AA$ be the region of $\AA$ containing $B$. We say a hyperplane $H\in\AA$ is \defn{basic} if its intersection with the boundary of $B_\AA$  has dimension $n-1$. For $H,H'\in\AA$, we say $H$ \defn{cuts} $H'$ if $H$ is basic in $\AA$ but $H'$ is not. Each hyperplane $H\in \HH$ is broken into a number of connected pieces if we remove all points in $H$ contained the the hyperplanes of $\HH$ that cut $H$---a \defn{shard} of $H$ is then the closure of one of these connected pieces.  We write $H_\Sigma$ for the unique hyperplane containing a shard $\Sigma$. Let $\Sha(\HH,B)$ denote the set of all shards. 

Extending the notation for hyperplanes, when $C' \xrightarrow{e} C$, we write $\Sigma(e)$ for the shard crossed by $e$ and call $\Sigma(e)$ a \defn{lower shard} of $C$.  Let $\covsha(C)$ be the set of lower shards of $C$. The \defn{shard intersection order} is the partial order $\preceq$ on $\RR$ defined by saying \[C'\preceq C\quad\text{if and only if}\quad \bigcap\covsha(C')\supseteq \bigcap\covsha(C).\]

For each shard $\Sigma$, we fix a minimal element $J_\Sigma$ of  $\{C\in\RR:\Sigma\in\covsha(C)\}$ in the poset $\Weak(\HH,B)$. Recall that an element in a finite lattice is called \defn{join-irreducible} if it covers exactly one element. When $\HH$ is simplicial, we know by \Cref{thm:simplicial_lattice} that $\Weak(\HH,B)$ is a lattice, and the following results of Reading show that shards give a geometric interpretation of this lattice structure.
 
\begin{theorem}[{\cite[Proposition 2.2]{reading2003order},\cite[Proposition 3.3]{reading2011noncrossing}}]
  If $\HH$ is simplicial, then there is a bijection between $\Sha(\HH,B)$ and the set of join-irreducible regions in $\Weak(\HH,B)$.  The shard $\Sigma$ corresponds to the join-irreducible region $J_\Sigma$, which is the unique minimal element of $\{ C\in\RR : \Sigma \in \covsha(C)\}$. The shard
  corresponding to a join-irreducible region $J$ is the unique element of $\covsha(J)$.
\label{thm:shards_and_join_irrs}
\end{theorem}

\begin{theorem}[{\cite[Theorem 3.6]{reading2011noncrossing}}]  If $\HH$ is simplicial and $C \in \RA$, then
  \[C = \bigvee_{\Sigma \in \covsha(C)} J_\Sigma.\]  
\label{thm:shards}
\end{theorem}
(There is actually a stronger version of \Cref{thm:shards}, which says that $\{J_\Sigma:\Sigma\in\covsha(C)\}$ is the \emph{canonical join representation} of $C$ in the semidistributive lattice $\Weak(\HH,B)$---the lattice property is strongly tied to the simpliciality of $\HH$; see \cite{reading2011noncrossing}.)

\section{Shards and the Salvetti Complex}
\label{sec:shard_generators}
Our goal in this section is to establish \Cref{thm:shard_generators}, which tells us that two loops $\tt_e$ and $\tt_f$ in $\TT_{\mathrm{edge}}$ are homotopic if and only if they go around the same shard (that is, $\Sigma(e)=\Sigma(f)$). Thus, while \Cref{thm:gens_to_other_shards,thm:abelianization} tell us that the elements of $\TT_{\mathrm{edge}}$ up to conjugation are indexed by the hyperplanes in $\HH$,  \Cref{thm:shard_generators} tells us that the elements of $\TT_{\mathrm{edge}}$ (not up to conjugation) are indexed by the shards of $\HH$.

\subsection{Proof of \texorpdfstring{\Cref{thm:shard_generators}}{Theorem 1.1}}
We work toward understanding general arrangements by first understanding rank-2 arrangements. We need the following general lemma, which will also be useful elsewhere in the paper. Given a subarrangement $\AA$ of $\HH$, we let \[\iota_\AA\colon\RR\to\RR_\AA\] be the map that sends a region $D\in\RR$ to the unique region of $\AA$ containing $D$, and we let $B_\AA=\iota_\AA(B)$.

\begin{lemma}\label{lem:quotient_sec4}
Let $\AA$ be a subarrangement of a central arrangement $\HH$. The map $\iota_\AA\colon\RR\to\RR_\AA$ induces a quotient map \[\iotab_\AA\colon\pi_1(\Sal(\HH),B)\to\pi_1(\Sal(\AA),B_\AA)\] that sends the generating set $\TT_{\mathrm{edge}}(\HH,B)$ to $\TT_{\mathrm{edge}}(\AA,B_\AA)\cup\{\id\}$, where $\id$ denotes the identity element of $\pi_1(\Sal(\AA),B_\AA)$. 
\end{lemma}

\begin{proof}
By identifying regions via $\iota_\AA$, we collapse all edges $C' \xrightarrow{e} C$ and $C\xrightarrow{e^*}C'$ with $\iota_{\AA}(C') = \iota_{\AA}(C)$ and identify all edges $C' \xrightarrow{e} C$ and $D' \xrightarrow{f} D$ with $\iota_{\AA}(C')=\iota_{A}(D')$ and $\iota_{A}(C)=\iota_{A}(D)$.  The quotient therefore introduces some relations of the form $\tt_e=\id$ (when $H_e \not \in \AA$) and some relations of the form $\tt_e = \tt_{e'}$.
\end{proof}

\begin{lemma}
\label{lem:A2}
	Suppose $\HH=\{H_1,H_2,H_3\}$ is a rank-2 hyperplane arrangement with three hyperplanes that has base region $B$ and edges named $e_1,e_2,e_3$ and $e_4,e_5,e_6$ (in the same manner as in~\Cref{fig:rank2}).  Then $\tt_{e_2}\neq\tt_{e_5}$.
\end{lemma}

\begin{proof}
If follows from \Cref{thm:gens_to_other_shards} that $\tt_{e_5}=\tt_{e_1}^{-1}\tt_{e_2}\tt_{e_1}$, so we just need to show that $\tt_{e_1}$ and $\tt_{e_2}$ do not commute in $\pi_1(\PP(\HH),B)$. By \Cref{thm:relations}, this group has presentation $\langle \tt_{e_1},\tt_{e_2},\tt_{e_3}:\tt_{e_1}\tt_{e_2}\tt_{e_3}=\tt_{e_3}\tt_{e_1}\tt_{e_2}=\tt_{e_2}\tt_{e_3}\tt_{e_1}\rangle$. Consider the free group $F_2=\langle x,y\rangle$. The map $\{\tt_{e_1},\tt_{e_2},\tt_{e_3}\}\to F_2$ given by $\tt_{e_1}\mapsto y^{-1}x^{-1}$, $\tt_{e_2}\mapsto x$, $\tt_{e_3}\mapsto y$ extends to a homomorphism from $\pi_1(\PP(\HH),B)$ to $F_2$; this implies that $\tt_{e_2}$ and $\tt_{e_3}$ do not commute. 
\end{proof}

\begin{lemma}
\label{lem:rank22}
	Let $\HH$ be a rank-2 hyperplane arrangement with base region $B$ and with edges named $e_1,e_2,\ldots,e_m$ and $e_{m+1},\ldots,e_{2m}$ as in~\Cref{fig:rank2}.  For all distinct $i,j\in\{1,\ldots,2m\}$ such that $\{i,j\}$ is not $\{m,2m\}$ or $\{1,m+1\}$, we have $\tt_{e_i}\neq\tt_{e_j}$.
\end{lemma}

\begin{proof}
We may assume $i<j$. If $j\neq i+m$, then the hyperplanes $H_{e_i}$ and $H_{e_j}$ are distinct. In this case, \Cref{thm:abelianization} tells us that the natural homomorphism $\pi_1(\PP(\HH),B) \to H_1(\PP(\HH))$ sends $\tt_{e_i}$ and $\tt_{e_j}$ to different elements, so $\tt_{e_i}\neq\tt_{e_j}$. Now suppose $j=m+1$ so that $H_{e_i}=H_{e_j}$. By hypothesis, we have $1<i<m$. Consider the subarrangement $\AA=\{H_{e_1},H_{e_i},H_{e_m}\}$. It follows from \Cref{lem:quotient_sec4} (and its proof) and \Cref{lem:A2} that the quotient map $\iotab_\AA$ sends $\tt_{e_i}$ and $\tt_{e_j}$ to distinct elements of $\TT_{\mathrm{edge}}(\Sal(\AA),B_\AA)$. Hence, $\tt_{e_i}\neq\tt_{e_j}$.  
\end{proof}

\begin{proposition}	
Let $\HH$ be a central arrangement.	 If $C' \xrightarrow{e} C$ and $D'\xrightarrow{f}D$ are edges in $\Sal(\HH)$ such that $\Sigma(e)=\Sigma(f)$, then $\tt_e = \tt_{f}$.
\label{prop:shard1}
\end{proposition}

\begin{proof}
Write $\Sigma := \Sigma(e)=\Sigma(f)$.  Choose a positive gallery $g$ from $C$ to $D$, and let $C=E_1,E_2,\ldots,E_r=D$ be the regions through which $g$ passes. Let $1=k_1<k_2<\cdots<k_s=r$ be the indices such that $\Sigma\in\cov(E_{k_i})$. We may assume $g$ is chosen so that for each $1\leq i\leq s-1$, the hyperplanes that $g$ crosses when passing from $E_{k_i}$ to $E_{k_{i+1}}$ all belong to a single full rank-2 subarrangement $\AA_i$ with $H_\Sigma\in\AA_i$. For $1\leq i\leq s-1$, there is an edge $E_{k_i}'\xrightarrow{d_i} E_{k_i}$ with $\Sigma(d_i)=\Sigma$. It suffices to prove that $\tt_{d_i}=\tt_{d_{i+1}}$ for each such $i$. 

Fix $1\leq i\leq s-1$, and let $B_{\AA_i}=\iota_{\AA_i}(B)$. Let $m=|\AA_i|$. The hyperplane $H_\Sigma$ must be basic in $\AA_i$ (otherwise, the basic hyperplanes in $\AA_i$ would cut $\Sigma$), so either $E_{k_i}$ or $E_{k_{i+1}}$ is contained in $-B_{\AA_i}$. For simplicity, let us assume $E_{k_{i+1}}\subseteq-B_{\AA_i}$; the other case is virtually the same. Then $E_{k_i}'\subseteq B_{\AA_i}$. We can find a positive minimal gallery $E_{k_i}'\xrightarrow{e_1}\cdots\xrightarrow{e_{m-1}}E_{k_{i+1}}'$ that only passes through hyperplanes in $\AA_i$. Let $E_{k_i}\xrightarrow{e_{2m-1}}\cdots\xrightarrow{e_{m+1}}E_{k_{i+1}}$ be the part of the gallery $g$ from $E_{k_i}$ to $E_{k_{i+1}}$, and note that this is a positive minimal gallery from $E_{k_i}$ to $E_{k_{i+1}}$. Let $e_{2m}=d_i$ and $e_m=d_{i+1}$. The collection of edges $e_1,\ldots,e_m,e_{m+1},\ldots,e_{2m}$ forms a $2$-cell in $\Sal(\HH)$, so we have the homotopy \[e_1e_2\cdots e_m\cong e_{2m}e_{2m-1}\cdots e_{m+1}.\] Similarly, the edges $e_1,\ldots,e_{m-1},e_m^*,e_{m+1},\ldots, e_{2m-1},e_{2m}^*$ form a $2$-cell in $\Sal(\HH)$, so we have the homotopy \[e_{2m-1}\cdots e_{m+1}e_m^*\cong e_{2m}^*e_1\cdots e_{m-1}.\] Hence, 
\begin{align*}
\tt_{d_{i+1}}&\cong\gal(B,E_{k_i}')(e_1e_2\cdots e_m)e_m^*e_{m-1}^{-1}\cdots e_1^{-1}\gal(B,E_{k_i}')^{-1} \\
&\cong\gal(B,E_{k_i}')e_{2m}(e_{2m-1}\cdots e_{m+1}e_m^*)e_{m-1}^{-1}\cdots e_1^{-1}\gal(B,E_{k_i}')^{-1} \\
&\cong\gal(B,E_{k_i}')e_{2m}(e_{2m}^*e_{1}\cdots e_{m-1})e_{m-1}^{-1}\cdots e_1^{-1}\gal(B,E_{k_i}')^{-1} \\
&\cong \gal(B,E_{k_i}')e_{2m}e_{2m}^*\gal(B,E_{k_i}')^{-1} \\
&\cong\tt_{d_i}.
\end{align*}
Thus, the generators $\tt_{d_i}$ and $\tt_{d_{i+1}}$ are the same. 
\end{proof}

\begin{comment}
\begin{figure}[htbp]
\begin{tikzpicture}[scale=2.5,->]
\node (0) at (0,0) {};
\node (1) at (1,0) {};
\node (2) at (.92,.38) {};
\node (3) at (.71,.71) {$X$};
\node (4) at (.38,.92) {};
\node (5) at (0,1) {$-B_{\AA_i}$};
\node (6) at (-.38,.92) {};
\node (7) at (-.71,.71) {};
\node (8) at (-.92,.38) {};
\node (9) at (-1,0) {};
\node (10) at (-.92,-.38) {};
\node (11) at (-.71,-.71) {$Y$};
\node (12) at (-.38,-.92) {};
\node (13) at (0,-1) {$B_{\AA_i}$};
\node (14) at (.38,-.92) {};
\node (15) at (.71,-.71) {};
\node (16) at (.92,-.38) {};
\draw[-,thick] (12) to (4);
\draw[-,thick] (14) to (6);
\draw[-,dotted,thick,shorten >=.1pt] (10) to (0);
\draw[-,dotted,thick,shorten >=.1pt] (8) to (0);
\draw[-,dotted,thick,shorten >=.1pt] (1) to (0);
\draw[-,dotted,thick,shorten >=.1pt] (9) to (0);
\draw[-,dotted,thick,shorten >=.1pt] (2) to (0);
\draw[-,dotted,thick,shorten >=.1pt] (16) to (0);
\path (13) edge[bend right] node [pos=0.5, below] {$\scriptstyle e_1$} (15);
\path (3) edge[bend right] node [pos=0.5,above] {$\scriptstyle d=e_m$} node [pos=0.5,below]  {$H_\Sigma$} (5);
\path (7) edge[bend left] node [pos=0.5,above] {$\scriptstyle e_{m+1}$} (5);
\path (13) edge[bend left] node [pos=0.4,below] {$\scriptstyle e_{2m}=d_i$} node [pos=0.5,above] {$H_\Sigma$} (11);
\end{tikzpicture}
\caption{A schematic illustration of the part of the proof of~\Cref{prop:shard1} giving a homotopy between $\tt_{d_i}$ and $\tt_{d_{i+1}}$.}
\label{fig:shard1}
\end{figure}
\end{comment}

\begin{proposition}	
Let $\HH$ be a central arrangement.	If $C' \xrightarrow{e} C$ and $D'\xrightarrow{f}D$ are edges in $\Sal(\HH)$ such that $\Sigma(e)\neq\Sigma(f)$, then $\tt_e \not\cong \tt_{f}$.
\label{prop:shard2}
\end{proposition}

\begin{proof}
Let $\Sigma=\Sigma(e)$ and $\Sigma'=\Sigma(f)$. If the hyperplanes $H_\Sigma$ and $H_{\Sigma'}$ are different, then the result follows from \Cref{thm:abelianization} by passing to the abelianization $H_1(\PP(\HH),B)$.

Now suppose $\Sigma$ and $\Sigma'$ are shards for the same hyperplane $H$. Let $H'$ be a hyperplane that cuts $H$ so that $\Sigma$ and $\Sigma'$ are on opposite sides of $H'$, and let $\AA$ be the full rank-2 subarrangement of $\HH$ containing $H$ and $H'$. Note that $H$ is not basic in $\AA$. It follows from \Cref{lem:rank22} and the proof of \Cref{lem:quotient_sec4} that the quotient map $\iotab_\AA$ sends $\tt_{e}$ and $\tt_{f}$ to different elements of $\TT_{\mathrm{edge}}(\AA,B_\AA)$, so $\tt_{e}\neq\tt_{f}$.
\end{proof}

Put together, \Cref{prop:shard1,prop:shard2} imply \Cref{thm:shard_generators}.

\subsection{The Pure Shard Monoid}\label{sec:pure_shard}
We now know that the generators in $\TT_{\mathrm{edge}}$ are indexed by shards, so we may write $\tt_\Sigma$ to refer to the \defn{shard loop} indexed by the shard $\Sigma$ and write $\TTsha=\TT_{\mathrm{edge}}$.

\begin{definition}
The \defn{pure shard monoid} $\PPP^+(\HH,B)$ is the submonoid of $\pi_1(\Sal(\HH),B)$ generated by $\TTsha$.
\end{definition}
Our name for $\PPP^+(\HH,B)$ comes from the fact that the generating set $\TTsha$ is indexed by shards and the fact that if $\HH$ is the reflection arrangement of a finite Coxeter group $W$, then $\pi_1(\Sal(\HH),B)$ is isomorphic to the pure braid group of $W$.

An \defn{$\TTsha$-word} is a word over the alphabet $\TTsha$. By definition, $\PPP^+(\HH,B)$ is the set of elements of $\pi_1(\Sal(\HH),B)$ represented by $\TTsha$-words. We view $\PPP^+(\HH,B)$ as a poset, where the partial order $\leq$ is defined by saying $p\leq p'$ if there is an $\TTsha$-word representing $p'$ that contains an $\TTsha$-word representing $p$ as a prefix. 

The \defn{full twist} is the element $\Delta^2$ of $\pi_1(\Sal(\HH),B)$ defined by \[\Delta^2:=\gal(B,-B)\cdot \gal(-B,B).\] For $\HH$ a finite irreducible simplicial arrangement, it is known that the center of $\pi_1(\Sal(\HH),B)$ is an infinite cyclic group generated by $\Delta^2$~\cite{cordovil1994center}.  We are especially interested in $[\id,\Delta^2]_{\PPP^+}$, the interval between the identity element $\id$ and the full twist $\Delta^2$ in $\PPP^+(\HH,B)$. This interval is ``tall and wide,'' but it is \emph{not} a lattice in general, preventing the use of Garside theory to study $\pi_1(\Sal(\HH),B)$ (see~\Cref{fig:a3}).  Note that if \[B=C_0\xrightarrow{e_1}C_{1}\xrightarrow{e_{2}} \cdots \xrightarrow{e_{k-1}}C_{k-1}\xrightarrow{e_k}C_k= -B\] is a positive minimal gallery from $B$ to $-B$, then 
\begin{equation}\label{eq:YY}
\Delta^2=\ell_{e_k}\ell_{e_{k-1}}\cdots\ell_{e_2}\ell_{e_1}. 
\end{equation}

\subsection{Proof of \texorpdfstring{\Cref{thm:self-dual}}{Theorem 1.2}}\label{sec:self-dual}
To initiate the study of the pure shard monoid, we consider an arbitrary central arrangement $\HH$ and prove that the interval $[\id,\Delta^2]_{\PPP^+}$ is self-dual. 

Observe that the set $\Sha(\HH,B)$ of shards defined with respect to the base region $B$ is equal to the set $\Sha(\HH,-B)$ of shards defined with respect to the base region $-B$. Moreover, for each shard $\Sigma\in\Sha(\HH,B)$, the polyhedral cone $-\Sigma$ is also a shard in $\Sha(\HH,B)$. Let us write $\widehat\ell_{e}$ for the loop in $\pi_1(\Sal(\HH),-B)$ corresponding to an edge $C' \xrightarrow{e} C$; that is, 
\[\widehat\ell_{\Sigma}=\gal(C,-B)^{-1}e^{-1}(e^*)^{-1}\gal(C,-B).\] It follows from \Cref{thm:shard_generators} that $\widehat\ell_e$ only depends on the shard $\Sigma(e)$. Using Salvetti's presentation in \Cref{thm:gens_to_other_shards,thm:relations}, we find that there is an isomorphism $\Psi\colon\pi_1(\Sal(\HH),B)\to\pi_1(\Sal(\HH),-B)$ satisfying $\Psi(\ell_\Sigma)=\widehat\ell_{-\Sigma}$ for all $\Sigma\in\Sha(\HH,B)$. 

For $\Sigma\in\Sha(\HH,B)$, we have 
\begin{equation}\label{eq:conjugate_delta}
\gal(B,-B)\,\widehat\ell_\Sigma\,\gal(B,-B)^{-1}\cong\ell_{\Sigma}^{-1}. 
\end{equation}
Indeed, if $C'\xrightarrow{e} C$ is an edge such that $\Sigma=\Sigma(e)$, then
\begin{align*}
\gal(B,-B)\,\widehat\ell_\Sigma\,\gal(B,-B)^{-1} &\cong \gal(B,-B)\gal(C,-B)^{-1}e^{-1}(e^*)^{-1}\gal(C,B)\gal(B,-B)^{-1} \\ 
&\cong \gal(B,C)e^{-1}(e^*)^{-1}\gal(B,C')^{-1} \\
&\cong \gal(B,C')(e^*)^{-1}e^{-1}\gal(B,C')^{-1} \\ 
&=\ell_\Sigma^{-1}. 
\end{align*}

Define a map $\kappa\colon\PPP^+(\HH,B)\to\PPP^+(\HH,B)$ by $\kappa(\ell_{\Sigma_1}\cdots\ell_{\Sigma_k})=\ell_{-\Sigma_k}\cdots\ell_{-\Sigma_1}$. Let us check that this map is well defined. Suppose $\Sigma_1,\ldots,\Sigma_k,\Sigma_1',\ldots,\Sigma_{k'}'\in\Sha(\HH,B)$ are such that ${\ell_{\Sigma_1}\cdots\ell_{\Sigma_k}=\ell_{\Sigma_1'}\cdots\ell_{\Sigma_{k'}'}}$. Applying the isomorphism $\Psi$, we find that \[\widehat\ell_{-\Sigma_1}\cdots\widehat\ell_{-\Sigma_k}=\widehat\ell_{-\Sigma_1'}\cdots\widehat\ell_{-\Sigma_{k'}'}.\] If we conjugate each side of this equation by $\gal(B,-B)$ and use the identity \eqref{eq:conjugate_delta}, we obtain 
\[
\ell_{-\Sigma_1}^{-1}\cdots\ell_{-\Sigma_k}^{-1}\cong\ell_{-\Sigma_1'}^{-1}\cdots\ell_{-\Sigma_{k'}'}^{-1}.
\]
This shows that $\kappa(\ell_{\Sigma_1}\cdots\ell_{\Sigma_k})^{-1}\cong \kappa(\ell_{\Sigma_1'}\cdots\ell_{\Sigma_{k'}'})^{-1}$. Since elements of $\pi_1(\Sal(\HH),B)$ are defined up to homotopy, it follows that $\kappa$ is well defined. It is also clear that $\kappa$ is an antihomomorphism (meaning $\kappa(xy)=\kappa(y)\kappa(x)$ for all $x,y\in\PPP^+(\HH,B)$) and an involution. 

Let \[B=C_0\xrightarrow{e_1}C_{1}\xrightarrow{e_{2}} \cdots \xrightarrow{e_{k-1}}C_{k-1}\xrightarrow{e_k}C_k= -B\] be a positive minimal gallery from $B$ to $-B$. For $1\leq i\leq k$, there is an edge $-C_{i}\xrightarrow{-e_{i}}-C_{i-1}$, and we have $\Sigma(-e_i)=-\Sigma(e_i)$. Note that \[B=-C_{k}\xrightarrow{-e_k}-C_{k-1}\xrightarrow{-e_{k-1}}\cdots\xrightarrow{-e_2}-C_1\xrightarrow{-e_1}-C_0=-B\] is a positive minimal gallery from $B$ to $-B$. It follows from \eqref{eq:YY} that 
\[\ell_{\Sigma(e_k)}\cdots\ell_{\Sigma(e_1)}=\Delta^2=\ell_{-\Sigma(e_1)}\cdots\ell_{-\Sigma(e_k)}.\] Hence, 
\begin{equation}\label{eq:kappa_Delta}
\kappa(\Delta^2)=\Delta^2. 
\end{equation}
If $x\in[\id,\Delta^2]_{\PPP^+}$, then 
\begin{equation}\label{eq:ZZ}
\Delta^2=\kappa(xx^{-1}\Delta^2)=\kappa(x^{-1}\Delta^2)\kappa(x), 
\end{equation}
so $\kappa(x^{-1}\Delta^2)\in[\id,\Delta^2]_{\PPP^+}$. Define $\Omega\colon[\id,\Delta^2]_{\PPP^+}\to[\id,\Delta^2]_{\PPP^+}$ by $\Omega(x)=\kappa(x^{-1}\Delta^2)$. The following proposition implies \Cref{thm:self-dual}. 

\begin{proposition}
The map $\Omega$ is an antiautomorphism of the interval $[\id,\Delta^2]_{\PPP^+}$. 
\end{proposition}

\begin{proof}
For $x\in[\id,\Delta^2]_{\PPP^+}$, it follows from \eqref{eq:ZZ} that 
\[\Omega^2(x)=\kappa(\kappa(x^{-1}\Delta^2)^{-1}\Delta^2)=\kappa(\kappa(x))=x.\] This shows that $\Omega$ is an involution; in particular, it is a bijection. 

To see that $\Omega$ is order-reversing, suppose $x,y\in[\id,\Delta^2]_{\PPP^+}$ are such that $x\leq y$, and let $z\in\PPP^+(\HH,B)$ be such that $y=xz$. We have 
\[\Omega(x)=\kappa(x^{-1}\Delta^2)=\kappa(y^{-1}\Delta^2)\kappa(z)=\Omega(y)\kappa(z), \]
so $\Omega(y)\leq \Omega(x)$. 
\end{proof}

\subsection{Proof of \texorpdfstring{\Cref{thm:rank2interval}}{Theorem 1.3}}\label{sec:rank2_enumeration}
We now delve deeper into the combinatorics of the interval $[\id,\Delta^2]_{\PPP^+}$ when $\HH$ is an arrangement of rank $2$ with $m$ hyperplanes.  As in~\Cref{fig:rank2}, index the shards of $\HH$ as $\Sigma_1,\Sigma_2,\ldots,\Sigma_{2m}$ so that $\Sigma_1=\Sigma_{m+1}$ and $\Sigma_m=\Sigma_{2m}$ are basic hyperplanes and so that $H_{\Sigma_i}=H_{\Sigma_{m+i}}$ for all $1 \leq i \leq m$. For simplicity, let us write $H_i=H_{\Sigma_i}$.  Index the generators of $\TTsha$ as $\rr_i=\tt_{\Sigma_i}$ and $\ttt_i=\tt_{\Sigma_{m+i}}$, and note that $\rr_1=\ttt_{1}$ and $\rr_m=\ttt_{m}$. Let $\binom{[m]}{k}$ be the collection of $k$-element subsets of $[m]$. For each $I=\{i_1<\cdots<i_k\}\in\binom{[m]}{k}$, let \[\rr_I = \rr_{i_k}\rr_{i_{k-1}} \cdots \rr_{i_1} \quad \text{and}\quad\ttt_I = \ttt_{i_1} \ttt_{i_{2}} \cdots \ttt_{i_k}.\] \Cref{fig:e_to_delta_P2,fig:e_to_delta_P3} depict the interval $[\id,\Delta^2]_{\PPP^+}$ when $m=4$ and $m=5$, respectively. 

A finite poset is called \defn{planar} if its Hasse diagram can be drawn in the $xy$-plane so that each cover relation is represented by a curve with a strictly increasing $y$-coordinate and so that no two curves representing cover relations intersect except possibly at their endpoint. Our goal in this subsection is to prove \Cref{thm:rank2interval}, which states that $[\id,\Delta^2]_{\PPP^+}$ is a planar lattice with rank generating function $1+\left(\sum_{k=1}^{m-1} \left(2\binom{m}{k}-2\right)q^k \right)+q^m$ and with $m 2^{m-2}$ maximal chains. We break the proof into \Cref{prop:rank2maxchain,prop:rank2rankgen,prop:planar} below. 

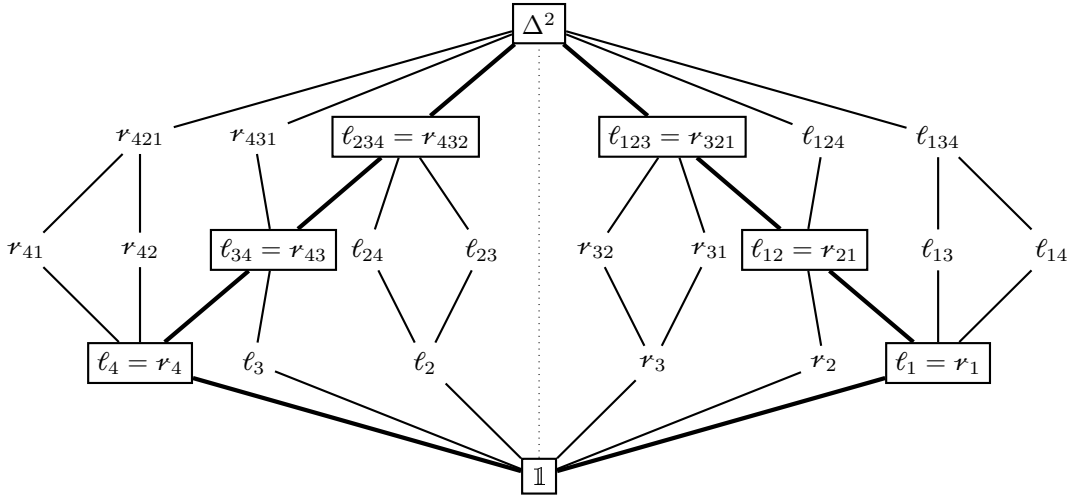
\begin{figure}[htbp]
\begin{tikzpicture}[scale=1.5]
\node[draw,thick] (1) at (0,0) {$\id$};

\node (l2) at (-1,1) {$\ttt_2$};
\node (l3) at (-2.5,1) {$\ttt_3$};
\node[draw,thick] (l4) at (-3.5,1) {$\ttt_4=\rr_4$};
\node[draw,thick] (r1) at (3.5,1) {$\ttt_1=\rr_1$};
\node (r2) at (2.5,1) {$\rr_2$};
\node (r3) at (1,1) {$\rr_3$};

\node (23) at (-.5,2) {$\ttt_{23}$};
\node (24) at (-1.5,2) {$\ttt_{24}$};
\node[draw,thick] (34) at (-2.33,2) {$\ttt_{34}=\rr_{43}$};
\node (42) at (-3.5,2) {$\rr_{42}$};
\node (41) at (-4.5,2) {$\rr_{41}$};
\node (14) at (4.5,2) {$\ttt_{14}$};
\node (13) at (3.5,2) {$\ttt_{13}$};
\node[draw,thick] (12) at (2.33,2) {$\ttt_{12}=\rr_{21}$};
\node (31) at (1.5,2) {$\rr_{31}$};
\node (32) at (.5,2) {$\rr_{32}$};

\node[draw,thick] (234) at (-1.17,3) {$\ttt_{234}=\rr_{432}$};
\node (431) at (-2.5,3) {$\rr_{431}$};
\node (421) at (-3.5,3) {$\rr_{421}$};
\node (134) at (3.5,3) {$\ttt_{134}$};
\node (124) at (2.5,3) {$\ttt_{124}$};
\node[draw,thick] (123) at (1.17,3) {$\ttt_{123}=\rr_{321}$};

\node[draw,thick] (d) at (0,4) {$\Delta^2$};

\draw[-,thick] (1) to node[midway] {} (l2);
\draw[-,thick] (1) to node[midway] {} (l3);
\draw[-,ultra thick] (1) to node[midway] {} (l4);
\draw[-,ultra thick] (1) to node[midway] {} (r1);
\draw[-,thick] (1) to node[midway] {} (r2);
\draw[-,thick] (1) to node[midway] {} (r3);

\draw[-,thick] (l2) to node[midway] {} (23);
\draw[-,thick] (l2) to node[midway] {} (24);
\draw[-,thick] (l3) to node[midway] {} (34);
\draw[-,ultra thick] (l4) to node[midway] {} (34);
\draw[-,thick] (l4) to node[midway] {} (42);
\draw[-,thick] (l4) to node[midway] {} (41);
\draw[-,thick] (r1) to node[midway] {} (14);
\draw[-,thick] (r1) to node[midway] {} (13);
\draw[-,ultra thick] (r1) to node[midway] {} (12);
\draw[-,thick] (r2) to node[midway] {} (12);
\draw[-,thick] (r3) to node[midway] {} (31);
\draw[-,thick] (r3) to node[midway] {} (32);

\draw[-,thick] (23) to node[midway] {} (234);
\draw[-,thick] (24) to node[midway] {} (234);
\draw[-,ultra thick] (34) to node[midway] {} (234);
\draw[-,thick] (34) to node[midway] {} (431);
\draw[-,thick] (42) to node[midway] {} (421);
\draw[-,thick] (41) to node[midway] {} (421);
\draw[-,thick] (14) to node[midway] {} (134);
\draw[-,thick] (13) to node[midway] {} (134);
\draw[-,thick] (12) to node[midway] {} (124);
\draw[-,ultra thick] (12) to node[midway] {} (123);
\draw[-,thick] (31) to node[midway] {} (123);
\draw[-,thick] (32) to node[midway] {} (123);

\draw[-,ultra thick] (234) to node[midway] {} (d);
\draw[-,thick] (431) to node[midway] {} (d);
\draw[-,thick] (421) to node[midway] {} (d);
\draw[-,thick] (134) to node[midway] {} (d);
\draw[-,thick] (124) to node[midway] {} (d);
\draw[-,ultra thick] (123) to node[midway] {} (d);
\draw[-,dotted] (1) to (d);
\end{tikzpicture}
\caption{A redrawing of the interval $[\id,\Delta^2]_{\PPP^+}$ in the pure shard monoid $\PPP^+(\HH,B)$ from~\Cref{fig:e_to_delta_P}, where elements are labeled using the conventions of~\Cref{sec:rank2_enumeration}. Boxed elements are in the image of $\Pow$ (see~\Cref{sec:pow}).}
\label{fig:e_to_delta_P2}
\end{figure}

\begin{figure}[htbp]
\scalebox{0.6}{
\begin{tikzpicture}[scale=1.5]
\node[draw,thick] (e) at (0,0) {$\id$};

\node[draw,thick] (l5) at (-8+.5,1) {$5$};
\node (l4) at (-6+.5,1) {$\ttt_4$};
\node (l3) at (-4+.5,1) {$\ttt_3$};
\node (l2) at (-2+.5,1) {$\ttt_2$};
\node (r4) at (2-.5,1) {$\rr_4$};
\node (r3) at (4-.5,1) {$\rr_3$};
\node (r2) at (6-.5,1) {$\rr_2$};
\node[draw,thick] (r1) at (8-.5,1) {$1$};

\node (51) at (-9+.5,2) {$51$};
\node (52) at (-8+.5,2) {$52$};
\node (53) at (-7+.5,2) {$53$};
\node[draw,thick] (45) at (-6+.5,2) {$45$};
\node (35) at (-5+.5,2) {$35$};
\node (34) at (-4+.5,2) {$34$};
\node (25) at (-3+.5,2) {$25$};
\node (24) at (-2+.5,2) {$24$};
\node (23) at (-1+.5,2) {$23$};

\node (15) at (9-.5,2) {$15$};
\node (14) at (8-.5,2) {$14$};
\node (13) at (7-.5,2) {$13$};
\node[draw,thick] (21) at (6-.5,2) {$21$};
\node (31) at (5-.5,2) {$31$};
\node (32) at (4-.5,2) {$32$};
\node (41) at (3-.5,2) {$41$};
\node (42) at (2-.5,2) {$42$};
\node (43) at (1-.5,2) {$43$};

\node (521) at (-9+.5,3) {$521$};
\node (531) at (-8+.5,3) {$531$};
\node (532) at (-7+.5,3) {$532$};
\node (541) at (-6+.5,3) {$541$};
\node (542) at (-5+.5,3) {$542$};
\node[draw,thick] (345) at (-4+.5,3) {$345$};
\node (245) at (-3+.5,3) {$245$};
\node (235) at (-2+.5,3) {$235$};
\node (234) at (-1+.5,3) {$234$};

\node (145) at (9-.5,3) {$145$};
\node (135) at (8-.5,3) {$135$};
\node (134) at (7-.5,3) {$134$};
\node (125) at (6-.5,3) {$125$};
\node (124) at (5-.5,3) {$124$};
\node[draw,thick] (321) at (4-.5,3) {$321$};
\node (421) at (3-.5,3) {$421$};
\node (431) at (2-.5,3) {$431$};
\node (432) at (1-.5,3) {$432$};

\node (5321) at (-8+.5,4) {$5321$};
\node (5421) at (-6+.5,4) {$5421$};
\node (5431) at (-4+.5,4) {$5431$};
\node[draw,thick] (2345) at (-2+.5,4) {$2345$};
\node[draw,thick] (4321) at (2-.5,4) {$4321$};
\node (1235) at (4-.5,4) {$1235$};
\node (1245) at (6-.5,4) {$1245$};
\node (1345) at (8-.5,4) {$1345$};

\node[draw,thick] (d) at (0,5) {$\Delta^2$};
\draw[-,thick] (l5) to (51) to (521) to (5321) to (d);
\draw[-,thick] (l5) to (52) to (521);
\draw[-,ultra thick] (e) to (l5) to (45) to (345) to (2345) to (d) to (4321) to (321) to (21) to (r1) to (e);
\draw[-,thick] (l5) to (53) to (531) to (5321);
\draw[-,thick] (53) to (532) to (5321);
\draw[-,thick] (e) to (l4) to (45) to (541) to (5421) to (d) to (5431) to (345) to (35) to (l3) to (e);
\draw[-,thick] (5421) to (542) to (45);
\draw[-,thick] (345) to (34) to (l3);
\draw[-,thick] (e) to (l2) to (25) to (245) to (2345) to (235) to (23) to (234) to (2345);
\draw[-,thick] (l2) to (24) to (245);
\draw[-,thick] (l2) to (23);
\draw[-,thick] (e) to (r1) to (15) to (145) to (1345) to (d) to (1245) to (125) to (21) to (r2) to (e) to (r3) to (31) to (321) to (1235) to (d);
\draw[-,thick] (e) to (r4) to (41) to (421) to (4321) to (432) to (43) to (r4) to (42) to (421);
\draw[-,thick] (r1) to (13) to (135) to (1345) to (134) to (13);
\draw[-,thick] (r1) to (14) to (145);
\draw[-,thick] (21) to (124) to (1245);
\draw[-,thick] (r3) to (32) to (321);
\draw[-,thick] (43) to (431) to (4321);
\draw[-,dotted] (e) to (d);
\end{tikzpicture}}
\caption{The interval $[\id,\Delta^2]_{\PPP^+}$ in the pure shard monoid $\PPP^+(\HH,B)$, where $\HH$ is a rank-2 arrangement with $m=5$ hyperplanes. Elements are labeled using the conventions of~\Cref{sec:rank2_enumeration}.  With the exception of singleton sets $I=\{i\}$ for $1<i<m$, elements $\rr_I$ are abbreviated as decreasing sequences, elements $\ttt_I$ are abbreviated as increasing sequences, and boxed elements are in the image of $\Pow$ (see~\Cref{sec:pow}) and may be viewed as increasing or decreasing.} 
\label{fig:e_to_delta_P3}
\end{figure}
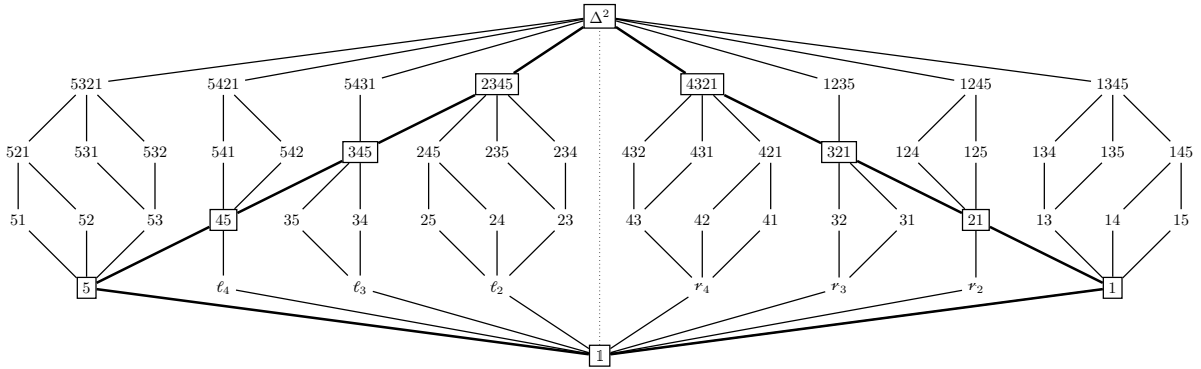

\begin{figure}[htbp]
\includegraphics[width=.9\textwidth]{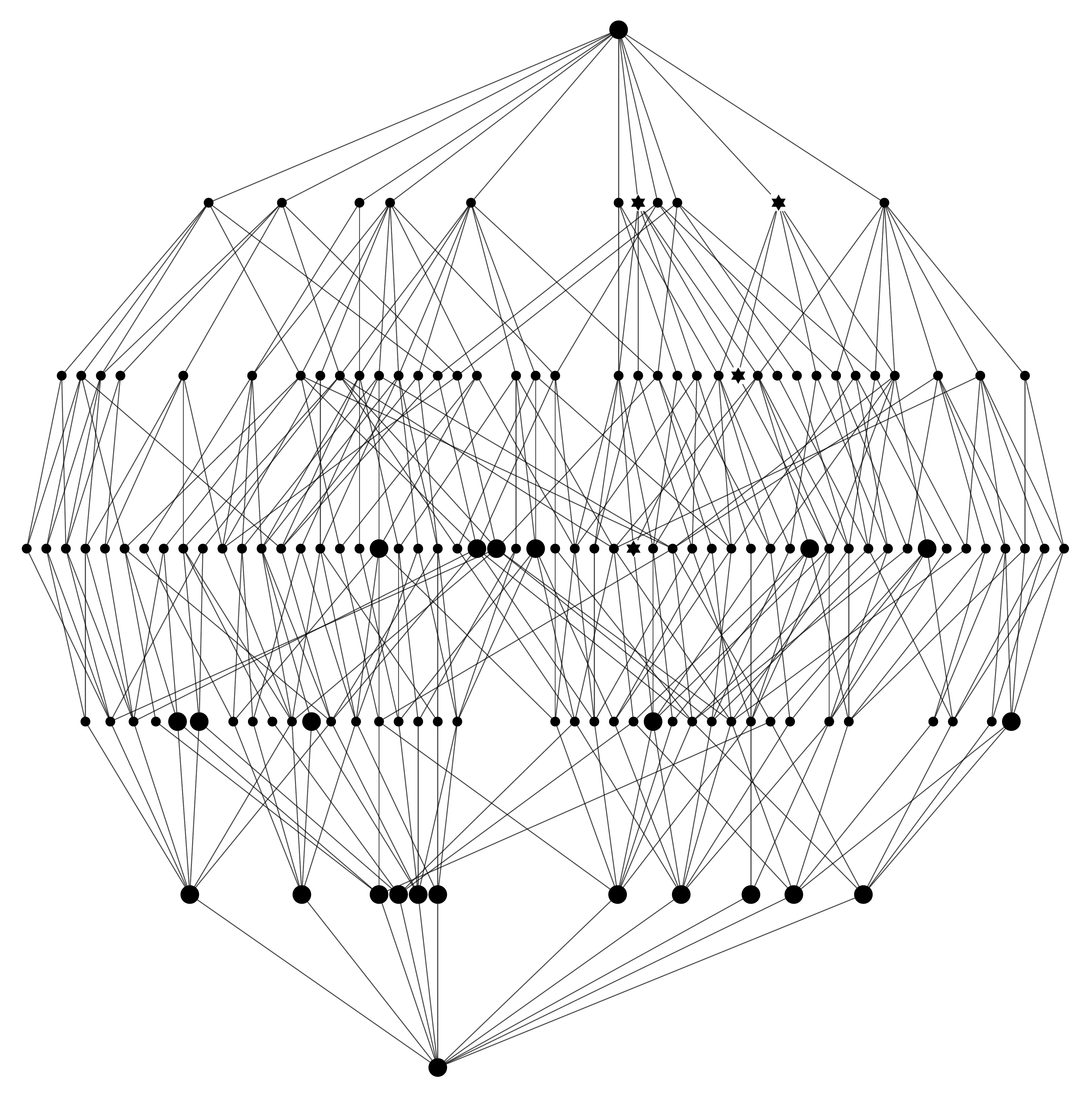}
\caption{The interval $[\id,\Delta^2]_{\PPP^+}$ in $\PPP^+(\HH,B)$ for the braid arrangement of type $A_3$.  It has 152 elements and 588 maximal chains.  Large circles represent elements in the image of $\Crackle$.  This interval is not a lattice, as witnessed by the 4 elements that have been drawn as stars.}
\label{fig:intervala3}
\end{figure}

Maximal chains in $[\id,\Delta^2]_{\PPP^+}$ correspond to $\TTsha(\HH,B)$-words for $\Delta^2$; we will first characterize these words, showing that there are $m2^{m-2}$ of them. Let us say an $\TTsha$-word is \defn{unimodal} if it is of the form $\ttt_I\rr_{[m]\setminus I}$, where $I\subseteq[m]$ is such that $1\in I$ and $m\not\in I$. The name comes from the fact that the indices in this word increase from $1$ up to $\max(I)$ and then decrease from $m$ down to $\min([m]\setminus I)$. Note that there are $2^{m-2}$ unimodal $\TTsha$-words. 

\begin{proposition}\label{prop:rank2maxchain}
Let $\HH$ be a rank-2 arrangement with $m$ hyperplanes. An $\TTsha$-word represents $\Delta^2$ if and only if it is a cyclic rotation of a unimodal $\TTsha$-word. Hence, the interval $[\id,\Delta^2]_{\PPP^+}$ has exactly $m2^{m-2}$ maximal chains.  
\end{proposition}
\begin{proof}
We will write $\TTsha(\HH,B)$ instead of $\TTsha$ since we will need to pass to subarrangements. For each subarrangement we consider, we will employ the quotient map from \Cref{lem:quotient_sec4}. One can check directly that the desired result holds when $m=3$, so we may assume $m\geq 4$ and proceed by induction on $m$. By direct computation, we find that $\Delta^2$ is represented by the unimodal $\TTsha(\HH,B)$-word ${\ttt_{[m-1]}\rr_{m}=\ttt_1\ttt_2\cdots\ttt_{m-1}\rr_m}$. Therefore, we can pass to the abelianization of $\pi_1(\PP(\HH),B)$ and invoke \Cref{thm:abelianization} to see that for each $1\leq i\leq m$, an $\TTsha(\HH,B)$-word for $\Delta^2$ must use exactly one of $\rr_i$ or $\ttt_i$; in particular, such a word has no repeated letters. Since $\Delta^2$ is in the center of $\pi_1(\Sal(\HH),B)$, the set of $\TTsha(\HH,B)$-words for $\Delta^2$ is closed under taking cyclic rotations. Thus, it suffices to prove that the $\TTsha(\HH,B)$-words for $\Delta^2$ that begin with $\ttt_1$ are precisely the unimodal $\TTsha(\HH,B)$-words. As illustrated in \Cref{fig:tree}, each unimodal $\TTsha(\HH,B)$-word can be obtained from the word $\ttt_1\ttt_2\cdots\ttt_{m-1}\rr_m$ by repeatedly using relations of the form $\ttt_i\ttt_{i+1}\cdots\ttt_{m-1}\rr_m=\rr_m\rr_{m-1}\cdots\rr_i$, which follow from \Cref{lem:reverse_order_pow}. Hence, all unimodal $\TTsha(\HH,B)$-words represent $\Delta^2$.

Now suppose we have an arbitrary $\TTsha(\HH,B)$-word $\calw_1\cdots\calw_m$ for $\Delta^2$ such that $\calw_1=\ttt_1$; we will show that it is unimodal. We consider two cases depending on whether $\rr_2$ or $\ttt_2$ appears in $\calw_1\cdots\calw_m$.

\medskip

\noindent {\bf Case 1.} Assume $\calw_p=\rr_2$ for some $p\geq 2$. We claim that $p=m$. Suppose instead that $p<m$, and let $i$ be such that $\calw_m\in\{\ttt_i,\rr_i\}$. Consider the subarrangement $\AA=\{H_1,H_2,H_i\}$. Let $\iotab_\AA$ be the quotient map from \Cref{lem:quotient_sec4}. The basic hyperplanes of $\AA$ are $H_{1}$ and $H_{i}$, so $\iotab_\AA(\ttt_i)=\iotab_\AA(\rr_i)$. It follows that $\iotab_\AA(\calw_1\cdots\calw_m)=\iotab_\AA(\ttt_1)\iotab_\AA(\rr_2)\iotab_\AA(\rr_i)$, so $\iotab_\AA(\ttt_1)\iotab_\AA(\rr_2)\iotab_\AA(\rr_i)$ is an $\TTsha(\AA,B_\AA)$-word representing $\iotab_\AA(\Delta^2)$, which is the full twist in $\pi_1(\Sal(\AA),B_\AA)$. However, this word is not a cyclic shift of a unimodal $\TTsha(\AA,B_\AA)$-word, so this contradicts our induction hypothesis. Hence, $p=m$. 

We now know that our $\TTsha(\HH,B)$-word is of the form $\ttt_1\calw_2\cdots\calw_{m-1}\rr_2$. Let $\HH'$ be the subarrangement $\HH\setminus\{H_1\}$ of $\HH$. The basic hyperplanes of $\HH'$ are $H_{2}$ and $H_{m}$, so $\iotab_{\HH'}(\ttt_{2})=\iotab_{\HH'}(\rr_{2})$. Now, $\iotab_{\HH'}(\calw_2)\cdots\iotab_{\HH'}(\calw_{m-1})\iotab_{\HH'}(\rr_{2})$ is an $\TTsha(\HH',B_{\HH'})$-word representing the full twist in $\pi_1(\Sal(\HH'),B_{\HH'})$, so by induction, it is a cyclic rotation of a unimodal $\TTsha(\HH',B_{\HH'})$-word. This shows that the $\TTsha(\HH',B_{\HH'})$-word $\iotab_{\HH'}(\rr_{2})\iotab_{\HH'}(\calw_2)\cdots\iotab_{\HH'}(\calw_{m-1})$ is unimodal, so $\ttt_1\calw_2\cdots\calw_{m-1}\rr_2$ is a unimodal $\TTsha(\HH,B)$-word. 

\medskip 

\noindent{\bf Case 2.} Assume $\calw_p=\ttt_{2}$ for some $p\geq 2$. We claim that $p=2$. Suppose instead that $p>2$, and let $i$ be such that $\calw_2\in\{\ttt_i,\rr_i\}$. Consider the subarrangement $\AA=\{H_{1},H_{2},H_{i}\}$. The basic hyperplanes of $\AA$ are $H_{1}$ and $H_{i}$, so $\iotab_\AA(\ttt_i)=\iotab_\AA(\rr_i)$. It follows that $\iotab_\AA(\ttt_1)\iotab_\AA(\ttt_i)\iotab_\AA(\ttt_2)$ is an $\TTsha(\AA,B_\AA)$-word representing $\iotab_\AA(\Delta^2)$, which is the full twist in $\pi_1(\Sal(\AA),B_\AA)$. This word is not a cyclic rotation of a unimodal $\TTsha(\AA,B_\AA)$-word, so this contradicts our induction hypothesis. Hence, $p=2$.

We now know that our $\TTsha(\HH,B)$-word is of the form $\ttt_1\ttt_2\calw_3\cdots\calw_{m}$. Let $\HH'$ be the subarrangement $\HH\setminus\{H_1\}$ of $\HH$. The basic hyperplanes of $\HH'$ are $H_{2}$ and $H_{m}$, so $\iotab_{\HH'}(\ttt_{2})=\iotab_{\HH'}(\rr_{2})$. Now, $\iotab_{\HH'}(\ttt_2)\iotab_{\HH'}(\calw_3)\cdots\iotab_{\HH'}(\calw_{m})$ is an $\TTsha(\HH',B_{\HH'})$-word representing the full twist in $\pi_1(\Sal(\HH'),B_{\HH'})$. By induction, we know that $\iotab_{\HH'}(\ttt_2)\iotab_{\HH'}(\calw_3)\cdots\iotab_{\HH'}(\calw_{m})$ is a unimodal $\TTsha(\HH',B_{\HH'})$-word. It follows that $\ttt_1\ttt_2\calw_3\cdots\calw_{m}$ is a unimodal $\TTsha(\HH,B)$-word. 
\end{proof}

\begin{figure}[htbp]
\begin{tikzpicture}[scale=2.5]
\node (432) at (0,0) {$\ttt_1\ttt_2 \ttt_3 \ttt_4 \rr_5$};
\node (43) at (-1,-1) {$\ttt_1\ttt_2\ttt_3\rr_5\rr_4$};
\node (4) at (0,-1) {$\ttt_1\ttt_2\rr_5\rr_4\rr_3$};
\node (0) at (1,-1) {$\ttt_1\rr_5\rr_4\rr_3\rr_2$};
\node (42) at (0,-2) {$\ttt_1\ttt_2\ttt_4\rr_5\rr_3$};
\node (32) at (1,-2) {$\ttt_1\ttt_3\ttt_4\rr_5\rr_2$};
\node (2) at (2,-2) {$\ttt_1\ttt_4\rr_5\rr_3\rr_2$};
\node (3) at (1,-3) {$\ttt_{1}\ttt_3\rr_5\rr_4\rr_2$};
\draw[-,thick] (432) to node[midway] {} (43);
\draw[-,thick] (432) to node[midway] {} (4);
\draw[-,thick] (432) to node[midway] {} (0);
\draw[-,thick] (4) to node[midway] {} (42);
\draw[-,thick] (0) to node[midway] {} (2);
\draw[-,thick] (0) to node[midway] {} (32);
\draw[-,thick] (32) to node[midway] {} (3);
\end{tikzpicture}
\caption{The $2^{m-2}=8$ different unimodal $\TTsha$-words when $\HH$ is a rank-2 hyperplane arrangement with $m=5$ hyperplanes.  All other words representing $\Delta^2$ are cyclic rotations of these.}
\label{fig:tree}
\end{figure}
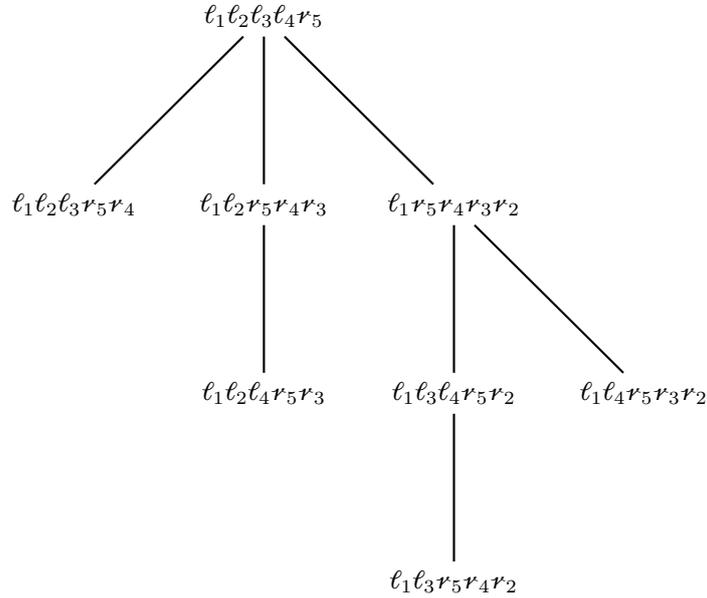

\begin{proposition}\label{prop:rank2rankgen}
Let $\HH$ be a rank-2 arrangement with $m$ hyperplanes. The elements of $[\id,\Delta^2]_{\PPP^+}$ are precisely the elements of the form $\ttt_I$ and $\rr_I$ for $I\subseteq [m]$, and all of these elements are distinct except that $\ttt_{\{1,\ldots,k\}}=\rr_{\{1,\ldots,k\}}$ and $\ttt_{\{m-k+1,\ldots,m\}}=\rr_{\{m-k+1,\ldots,m\}}$ for all $0\leq k\leq m$. In particular, the rank generating function of $[\id,\Delta^2]_{\PPP^+}$ is $1+\left(\sum_{k=1}^{m-1} \left(2\binom{m}{k}-2\right)q^k \right)+q^m$.
\end{proposition}

\begin{proof}
It is clear that there is $1$ element of rank $0$ (the identity $\id$) and that there is $1$ element of rank $m$ (the full twist $\Delta^2$). Now fix $1\leq k\leq m-1$. For each $I\in\binom{[m]}{k}$, we can use the fact that $\rr_1=\ttt_{1}$ and $\rr_m=\ttt_{m}$ to see that $\ttt_I\rr_{[m]\setminus I}$ and $\rr_I\ttt_{[m]\setminus I}$ (viewed as words) are cyclic rotations of unimodal $\TTsha$-words. Therefore, it follows from \Cref{prop:rank2maxchain} that $\ttt_I\leq\Delta^2$ and $\rr_I\leq\Delta^2$. If $I,I'\in\binom{[m]}{k}$ as distinct, then we can pass to the abelianiztion of $\pi_1(\PP(\HH),B)$ and invoke \Cref{thm:abelianization} to see that $\ttt_I\neq\ttt_{I'}$, $\rr_I\neq\rr_{I'}$, and $\ttt_I\neq\rr_{I'}$. It follows from \Cref{thm:gens_to_other_shards,thm:relations} that $\ttt_{\{1,\ldots,k\}}=\rr_{\{1,\ldots,k\}}$ and $\ttt_{\{m-k+1,\ldots,m\}}=\rr_{\{m-k+1,\ldots,m\}}$. We claim that if $I$ is not $\{1,\ldots,k\}$ or $\{m-k+1,\ldots,m\}$, then $\ttt_I\neq\rr_I$. It will then follow that we have found $2\binom{m}{k}-2$ distinct elements of rank $k$ in $[\id,\Delta^2]_{\PPP^+}$. 

We can check the claim directly when $m=3$, so we may assume $m\geq 4$ and proceed by induction on $m$. Suppose $I$ is not $\{1,\ldots,k\}$ or $\{m-k+1,\ldots,m\}$. Let $x=\min(I)$. If there exist $y\in [m]\setminus I$ and $z\in I$ with $x<y<z$, then we can let $\AA=\{H_x,H_y,H_z\}$ and use the quotient map $\iotab_\AA$ from \Cref{lem:quotient_sec4}. In this case, it follows by induction that $\iotab_\AA(\ttt_I)=\iotab_\AA(\ttt_x)\iotab_\AA(\ttt_z)\neq\iotab_\AA(\rr_z)\iotab_\AA(\rr_x)=\iotab_\AA(\rr_I)$, so $\ttt_I\neq\rr_I$. Now suppose no such $y$ and $z$ exist. This prohibits $1$ and $m$ from being in $I$. Let $\AA'=\{H_1,H_x,H_m\}$, and consider the quotient map $\iotab_{\AA'}$. In this case, it follows by induction (or, more directly, from \Cref{lem:A2}) that $\iotab_{\AA'}(\ttt_I)=\iotab_{\AA'}(\ttt_x)\neq\iotab_{\AA'}(\rr_x)=\iotab_{\AA'}(\rr_I)$, so $\ttt_I\neq\rr_I$. This establishes the claim.

To complete the proof, we need to show that every element of rank $k$ in $[\id,\Delta^2]_{\PPP^+}$ is equal to $\ttt_I$ or $\rr_I$ for some $I\in\binom{[m]}{k}$. Thus, let us choose an arbitrary element $u\in[\id,\Delta^2]_{\PPP^+}$ of rank $k$. Let $\calw_1\cdots\calw_k$ be an $\TTsha$-word representing $u$. 
Let $I=\{i_1<\cdots<i_k\}\in\binom{[m]}{k}$ be such that $\{H_{i_1},\ldots,H_{i_k}\}$ is the set of hyperplanes containing the shards that index the shard loops $\calw_1,\ldots,\calw_k$. 
It is possible to extend the word $\calw_1\cdots\calw_k$ to a word $\calw_1\cdots\calw_k\calw_{k+1}\cdots\calw_m$ representing $\Delta^2$, and we know by \Cref{prop:rank2maxchain} that $\calw_1\cdots\calw_k\calw_{k+1}\cdots\calw_m$ is a cyclic rotation of a unimodal $\TTsha$-word. If $1$ and $m$ are not in $I$, then this immediately implies that $\calw_1\cdots\calw_k$ is either $\ttt_I$ or $\rr_I$. Let us now assume $1$ and $m$ are both in $I$; we omit the proof in the case when exactly one of these indices is in $I$ because it is very similar. Let us also assume $\ttt_1$ appears before $\rr_m$ in the word $\calw_1\cdots\calw_k$; the other case is similar. Then $\calw_1\cdots\calw_k$ is of the form $\rr_{K_1}\ttt_1\ttt_{K_2}\rr_m\rr_{K_3}$, where $K_1\sqcup K_2\sqcup K_3\sqcup\{1,m\}=I$ and $x<y<z$ for all $x\in K_1$, $y\in [m]\setminus I$, and $z\in K_3$. By repeatedly applying relations of the form $\ttt_i\ttt_{i+1}\cdots\ttt_{m-1}\rr_m=\rr_m\rr_{m-1}\cdots\rr_i$ (which follow from \Cref{lem:reverse_order_pow}), we can transform $\rr_{K_1}\ttt_1\ttt_{K_2}\rr_m\rr_{K_3}$ into the word $\rr_{K_1}\ttt_1\ttt_{K_2\cup K_3}\rr_m$ (see \Cref{ex:rank2_4}). Then, by repeatedly applying relations of the form $\ttt_1\ttt_{2}\cdots\ttt_{i}=\rr_i\rr_{i-1}\cdots\rr_2\ttt_1$ (which follow from \Cref{lem:reverse_order_pow}), we can transform $\rr_{K_1}\ttt_1\ttt_{K_2\cup K_3}\rr_m$ into $\ttt_1\ttt_{K_1\cup K_2\cup K_3}\rr_m$, which is the same as $\ttt_I$ because $\rr_m=\ttt_m$. Hence, $u$ is represented by $\ttt_I$. 
\end{proof}

\begin{example}\label{ex:rank2_4}
Suppose $m=10$. The word $\rr_2\ttt_1\ttt_4\ttt_7\ttt_8\rr_{10}\rr_9\rr_6$ represents an element $u$ of rank $8$ in $[\id,\Delta^2]_{\PPP^+}$ because it is a prefix of $\rr_2\ttt_1\ttt_4\ttt_7\ttt_8\rr_{10}\rr_9\rr_6\rr_5\rr_3$, which is a cyclic rotation of a unimodal $\TTsha$-word. In the notation of the proof of \Cref{prop:rank2rankgen}, we have $I=\{1,2,4,6,7,8,9,10\}$, $K_1=\{2\}$, $K_2=\{4,7,8\}$, and $K_3=\{6,9\}$. We can apply relations of the form $\ttt_i\ttt_{i+1}\cdots\ttt_{9}\rr_{10}=\rr_{10}\rr_{9}\cdots\rr_i$ to find that \[\rr_2\ttt_1\ttt_4\ttt_7\ttt_8\rr_{10}\rr_9\rr_6=\rr_2\ttt_1\ttt_4\ttt_7\ttt_8\ttt_9\rr_{10}\rr_6=\rr_2\ttt_1\ttt_4\rr_{10}\rr_9\rr_8\rr_7\rr_6=\rr_2\ttt_1\ttt_4\ttt_6\ttt_7\ttt_8\ttt_9\rr_{10};\] this last expression is $\rr_{K_2}\ttt_1\ttt_{K_2\cup K_3}\rr_m$. Applying the relation $\ttt_1\ttt_{2}\cdots\ttt_{i}=\rr_i\rr_{i-1}\cdots\rr_2\ttt_1$ when $i=2$, we find that \[\rr_2\ttt_1\ttt_4\ttt_6\ttt_7\ttt_8\ttt_9\rr_{10}=\ttt_1\ttt_2\ttt_4\ttt_6\ttt_7\ttt_8\ttt_9\rr_{10};\] this last expression is the same as $\ttt_I$ because $\rr_{10}=\ttt_{10}$. 
\end{example}

\begin{proposition}\label{prop:planar}
If $\HH$ is a rank-2 arrangement, then the interval $[\id,\Delta^2]_{\PPP^+}$ is a planar lattice. 
\end{proposition}

\begin{proof}
According to \cite[Ex.~7, p.~20]{birkhoff1940lattice}, a finite planar poset with a unique minimal element and a unique maximal element is automatically a lattice. Hence, we just need to prove planarity. We are going to describe an explicit drawing of the Hasse diagram of $[\id,\Delta^2]_{\PPP^+}$, and we will give a full description of the edges in this Hasse diagram. Using the ideas discussed earlier in this subsection, one can verify that these are in fact all of the edges. 

Let $m$ be the number of hyperplanes in $\HH$. Suppose $1\leq k\leq m-1$. \Cref{prop:rank2rankgen} gives us an explicit description of the elements of rank $k$ in $[\id,\Delta^2]_{\PPP^+}$; we will construct a sequence that lists these elements. The sequence consists of four subsequences, where the first two subsequences overlap in a single element and the last two subsequences also overlap in a single element. The first subsequence lists the elements $\rr_I$ with $m\in I$ so that the indexing subsets appear in lexicographically-increasing order; thus, this subsequence starts with $\rr_{\{1,\ldots,k-1,m\}}$ and ends with $\rr_{\{m-k+1,\ldots,m\}}$. The second subsequence lists the elements $\ttt_I$ with $1\not\in I$ so that the indexing subsets appear in lexicographically-decreasing order; this subsequence starts with $\ttt_{\{m-k+1,\ldots,m\}}$ (which is the same as the last element in the first subsequence) and ends with $\ttt_{\{2,\ldots,k+1\}}$. The third subsequence lists the elements $\rr_I$ with $m\not\in I$ so that the indexing subsets appear in lexicographically-decreasing order; this subsequence starts with $\rr_{\{m-k,\ldots,m-1\}}$ and ends with $\rr_{\{1,\ldots,k\}}$. The fourth subsequence lists the elements $\ttt_I$ with $1\in I$ so that the indexing subsets appear in lexicographically-increasing order; this subsequence starts with $\ttt_{\{1,\ldots,k\}}$ (which is the same as the last element in the third subsequence) and ends with $\ttt_{\{1,m-k+2,\ldots,m\}}$. 

Let us draw the Hasse diagram of $[\id,\Delta^2]_{\PPP^+}$ so that elements of rank $k$ are drawn at height $k$ and so that all edges are line segments. For each $1\leq k\leq m-1$, we can draw the elements of rank $k$ so that they appear from left to right in the order specified by the sequence in the preceding paragraph; see \Cref{fig:e_to_delta_P2,fig:e_to_delta_P3}. We claim that this drawing is planar. To see this, we note that we can break this drawing of the Hasse diagram into four pieces, where some pieces overlap on their boundaries. We have a \emph{northwest} piece consisting of elements of the form $\rr_I$ with $m\in I$, a \emph{southwest} piece consisting of elements $\ttt_I$ with $1\not\in I$, a \emph{southeast} piece consisting of elements $\rr_I$ with $m\not\in I$, and a \emph{northeast} piece consisting of elements $\ttt_I$ with $1\in I$. If we can show that each of these pieces is planar, then it will follow that the entire drawing of the Hasse diagram is planar. We will prove that the southwest piece is planar; the other pieces are similar.  

For $2\leq i\leq m$, let $Q_i$ be the interval $[\ttt_i,\ttt_{\{i,i+1,\ldots,m\}}]_{\PPP^+}$. The set of elements in the southwest piece is the disjoint union $\{\id\}\sqcup Q_2\sqcup\cdots\sqcup Q_m$. For $2\leq i<i'\leq m$, there are no cover relations between elements of $Q_i$ and elements of $Q_{i'}$ unless $i'=i+1$, in which case the only such cover relation is $\ttt_{\{i+1,\ldots,m\}}\lessdot\ttt_{\{i,\ldots,m\}}$. Thus, it suffices to show that our drawing of each interval $Q_i$ is planar. We do so by reverse induction on $i$, noting first that our drawing of $Q_m$ is certainly planar because $Q_m$ is the single element $\ttt_m$. Now suppose $2\leq i\leq m-1$. The interval $Q_i$ is the disjoint union of the two subintervals $Q_i^{\swarrow}=[\ttt_i,\ttt_{\{i,i+2,i+3,\ldots,m\}}]_{\PPP^+}$ and $Q_i^{\nearrow}=[\ttt_{\{i,i+1\}},\ttt_{\{i,\ldots,m\}}]_{\PPP^+}$. The subinterval $Q_i^\swarrow$ consists of all elements $\ttt_I\in Q_i$ with $i+1\not\in I$, while the subinterval $Q_I^\nearrow$ consists of all elements $\ttt_I\in Q_i$ with $i+1\in I$. Both $Q_i^\swarrow$ and $Q_i^\nearrow$ are isomorphic to $Q_{i+1}$, so by induction, they are planar (and our drawing represents them in a planar manner). Finally, the only cover relations between elements of $Q_i^\swarrow$ and elements of $Q_i^\nearrow$ are $\ttt_i\lessdot\ttt_{\{i,i+1\}}$ and $\ttt_{\{i,i+2,i+3,\ldots,m\}}\lessdot\ttt_{\{i,\ldots,m\}}$. Thus, our drawing of $Q_i$ is planar.  
\end{proof}

\section{Pop}\label{sec:pop}

As in \Cref{subsec:pop_intro}, we consider the \defn{pop-stack sorting operator} $\pop\colon L\to L$, where $L$ is a locally finite meet-semilattice, defined by
\begin{equation}
    \pop(x):=x\wedge\bigwedge\{y\in L:y\lessdot x\}
\end{equation}
for all $x\in L$. For a simplicial arrangement $\HH$ with $C\in\RR$, let 
\[\Sigma([\pop(C),C]):=\left\{\Sigma(D'\lessdot D):\pop(C)\leq D'\lessdot D\leq C\right\}.\]
In \cite{reading2011noncrossing}, Reading provided the following characterization of the shard intersection order using $\Pop$.

\begin{proposition}[{\cite[Proposition 5.7]{reading2011noncrossing}}]\label{prop:shard_intersection_via_labels}
Suppose $\HH$ is a simplicial arrangement. For $C,C'\in\RR$, we have
\[C'\preceq C\quad\text{if and only if}\quad \Sigma([\pop(C'),C'])\subseteq\Sigma([\pop(C),C]),
\] where $\pop$ denotes the pop-stack sorting operator on the lattice $\Weak(\HH,B)$. 
\end{proposition}

Generalizing \cite[Theorem 2.10.5]{stump2015cataland}, Reading gave a second characterization of the shard intersection order.

\begin{theorem}[{\cite[Theorem 9-7.24]{reading2016lattice},\cite[Theorem 2.10.5]{stump2015cataland}}]\label{thm:shard_intersection_via_hyperplanes}
  Suppose $\HH$ is a simplicial arrangement.  For $C,C' \in \RR$, we have
  \[C' \preceq C \text{ if and only if } C' \leq C \text{ and } \bigcap \cov(C') \supseteq \bigcap \cov(C) .\]
\end{theorem}

\section{Crackle}\label{sec:crackle}

We assume throughout this section that $\HH$ is a simplicial real arrangement with base region $B$.

\subsection{Full Twists, Shard Generators, and Crackle}\label{subsec:full_twist}

We have $[\pop(-B),-B]=\Weak(\HH,B)$ since $\pop(-B)=B$ by \cite[Corollary 1.6]{edelman1984partial}.  Given a region $C\in\RR$, we can think of the interval $[\pop(C),C]$ in $\Weak(\HH,B)$ as a ``nonstandard parabolic subarrangement'' of $\HH$. This motivates us to define $\Crackle(C)$ to be the ``parabolic full twist'' of this subarrangement.

\begin{definition}\label{def:crackle}  Let $C \in \RR$.  We define
\[\Crackle(C) := \gal(B,\pop(C))\cdot \gal(\pop(C),C,\pop(C))\cdot \gal(B,\pop(C))^{-1} \in \pi_1(\Sal(\HH),B).
\]
\end{definition}

\Cref{def:crackle} generalizes both the full twist and the shard loops of~\Cref{sec:shard_generators}:  on the one hand, we have $\Crackle(-B)=\Delta^2$ by construction, so $\Crackle$ generalizes the full twist; on the other hand, if $J$ is a join-irreducible region of $\Weak(\HH,B)$, then $\pop(J)$ is the unique region covered by $J$, so $\Crackle(J)=\tt_{\Sigma(\pop(J)\lessdot J)}$.  

Our primary goal in this section is to prove \Cref{thm:crackle_hom}, which uses $\Crackle$ to characterize the shard intersection order. Recall that this theorem states that $\Crackle$ is a poset embedding from $\Shard(\HH,B)$ into $[\id,\Delta^2]_{\PPP^+}$ and is illustrated in~\Cref{fig:e_to_delta_P,fig:intervala3}.

\subsection{Crackle as a Product of Shard Generators}
Our first task is to prove \Cref{cor:crackle_shard_generators}, which tells us that $\Crackle$ maps $\RR$ into the interval $[\id,\Delta^2]_{\PPP^+}$ in the pure shard monoid. 

\begin{lemma}\label{lem:crackle_shard_generators}
Consider $E',E\in\RR$ with $E'\leq E$. If
\[E'=E_0 \xrightarrow{e_1} E_{1} \xrightarrow{e_{2}}  \cdots \xrightarrow{e_{k-1}} E_{k-1} \xrightarrow{e_{k}} E_k = E\] is a positive minimal gallery in $\Sal(\HH)$ from $E'$ to $E$, then 
\[\gal(B,E')\cdot\gal(E',E,E')\cdot\gal(B,E')^{-1} = \tt_{\Sigma(e_k)} \tt_{\Sigma(e_{k-1})}\cdots \tt_{\Sigma(e_1)}.\] 
\end{lemma}
  
\begin{proof}
We have 
\begin{align*}
\gal(B,E')\cdot\gal(E',E,E')\cdot\gal(B,E')^{-1}
    &= \gal(B,E')\cdot \left(e_1\cdots e_k e_k^*\cdots e_1^*\right) \cdot\gal(B,E')^{-1}\\
    &\cong \prod_{i=k}^1\left(\gal(B,E') \cdot \left(e_1\cdots e_{i-1}e_ie_i^*e_{i-1}^{-1} \cdots e_k^{-1}\right) \cdot  \gal(B,E')^{-1}\right)\\
    &= \tt_{\Sigma(e_k)} \tt_{\Sigma(e_{k-1})}\cdots \tt_{\Sigma(e_1)}. \qedhere
\end{align*}
\end{proof}  

\begin{proposition}\label{cor:crackle_shard_generators}
Let $C\in\RR$. If
\[\pop(C)=E_0 \xrightarrow{e_1} E_{1} \xrightarrow{e_{2}}  \cdots \xrightarrow{e_{k-1}} E_{k-1} \xrightarrow{e_{k}} E_k = C\] is a positive minimal gallery in $\Sal(\HH)$ from $\pop(C)$ to $C$, then 
\[\Crackle(C) = \tt_{\Sigma(e_k)} \tt_{\Sigma(e_{k-1})}\cdots \tt_{\Sigma(e_1)}.\] Moreover, $\Crackle(C)\leq\Delta^2$ in $\PPP^+(\HH,B)$.  
\end{proposition}    
\begin{proof}
    The first statement is immediate from  \Cref{lem:crackle_shard_generators} (with $E=C$ and $E'=\pop(C)$). Now let $B \xrightarrow{f_1} \cdots \xrightarrow{f_{r}} \pop(C)$ be a positive minimal gallery from $B$ to $\pop(C)$, and let $C \xrightarrow{f_{r+1}} \cdots \xrightarrow{f_m} -B$ be a positive minimal gallery from $C$ to $-B$.  Applying~\Cref{lem:crackle_shard_generators} (with $E'=B$ and $E=-B$), we find that
\begin{align*}
\Delta^2 =\Crackle(-B) &=  \tt_{\Sigma(f_{m})} \cdots \tt_{\Sigma(f_{r+1})} \tt_{\Sigma(e_k)}\cdots\tt_{\Sigma(e_1)} \tt_{\Sigma(f_{r})} \cdots \tt_{\Sigma(f_1)} \\
&=\tt_{\Sigma(e_k)}\cdots\tt_{\Sigma(e_1)} \tt_{\Sigma(f_{r})} \cdots \tt_{\Sigma(f_1)}\tt_{\Sigma(f_m)}\cdots\tt_{\Sigma(f_{r+1})},
\end{align*}
where the last equality follows from the fact that $\Delta^2$ is central in $\pi_1(\Sal(\HH),B)$.  This proves the second statement since $\tt_{\Sigma(e_k)}\cdots \tt_{\Sigma(e_1)}=\Crackle(C)$. 
\end{proof}

\subsection{Intervals, Subarrangements, and Crackle}
Let $\AA$ be a subarrangement of $\HH$. Let ${\iota_\AA\colon\RR\to\RR_\AA}$ denote the map that sends a region $D$ of $\HH$ to the unique region of $\AA$ containing $D$, and let $\iotab_\AA$ be the quotient map from \Cref{lem:quotient_sec4}. For $A \in \RR$, let $\HH_A$ be the subarrangement of $\HH$ consisting of all hyperplanes that contain the intersection of the lower covers of $A$; in this case, we write $\RR_A$, $B_A$, $\iota_A$, and $\iotab_A$ instead of $\RR_\AA$, $B_\AA$, $\iota_\AA$, and $\iotab_\AA$. We will often intentionally confuse $\HH_A$ with its essentialization, which is simplicial; thus, it also makes sense to write $\pop_A$ and $\Crackle_A$ for the pop-stack sorting operator and the crackle map on $\Weak(\HH_A,B_A)$. 

\begin{lemma}\label{lem:subarrangement_quotient}
For all $A,C\in\RR$, we have \[\iotab_A(\Crackle(C))=\Crackle_A(\iota_A(C)).\]
\end{lemma}

\begin{proof}
The quotient $\iotab_A$ introduces the relation $\tt_\Sigma=\id$ when $H_\Sigma \not \in \HH_A$ and the relation $\tt_\Sigma = \tt_{\Sigma'}$ when $\Sigma$ and $\Sigma'$ are contained in the same shard of $\HH_A$, so the desired identity follows from~\Cref{cor:crackle_shard_generators}.
\end{proof}

For $C\in\RR$, it follows from Reading's work in \cite[Propositions 5.7 and 5.8]{reading2011noncrossing} that there are two natural ways to identify some of the regions of $\HH$ with the regions of $\HH_C$: the map $\iota_C$ restricts to a bijection $\iota_{C}^\preceq \colon\{D \in \RR: D \preceq C\}\to\RR_C$ and also restricts to a bijection $\iota_C^\pop \colon[\pop(C),C]\to\RR_C$. The next result, which is illustrated in \Cref{fig:two_ways}, also follows from Reading's work. 

\begin{theorem}[{\cite[Propositions 5.7 and 5.8]{reading2011noncrossing}}]\label{thm:two_ways}
Consider regions $C,D\in\RR$ with $D\preceq C$. There is a poset isomorphism \[\omega_C:[\pop(D),D]\to[(\iota^\pop_C)^{-1}(\pop_C(\iota_C(D))),(\iota^\pop_C)^{-1}(\iota_C(D))],\] where both intervals are taken in $\Weak(\HH,B)$. Suppose $E'\xrightarrow{e}E$ is an edge in $\Sal(\HH)$ such that \[\pop(D)\leq E'\lessdot E\leq D,\] and let $\omega_C(e)$ denote the corresponding edge $\omega_C(E')\xrightarrow{\omega_C(e)}\omega_C(E)$. Then \[\Sigma(e)=\Sigma(\omega_C(e)).\]
\end{theorem}

\begin{figure}
\begin{tikzpicture}[scale=1.5]
\draw[red,thick,dotted] (-.5,2) to (2,3.5);
\draw[red,thick,dotted] (-.5,1) to (2,2.5);
\draw[dotted,thick,bend left=20] (0,0) to (-.5,1);
\draw[dotted,thick,bend right] (0,0) to (2,1.5);
\draw[red,thick,bend right=60] (-.5,1) to (-.5,2);
\draw[red,thick,bend left=60] (-.5,1) to (-.5,2);
\draw[red,thick,bend right=60] (2,2.5) to (2,3.5);
\draw[red,thick,bend left=60] (2,2.5) to (2,3.5);
\draw[thick,bend right=60] (2,1.5) to (2,5.5);
\draw[thick,bend left=60] (2,1.5) to (2,5.5);
\draw[dotted,thick,bend right=20] (2,5.5) to (1,6.5);
\draw[dotted,thick,bend left=80] (0,0) to (1,6.5);
\draw[dotted,thick,bend left=25] (-.5,2) to (1,6.5);
\node (popd) at (-.5,1) {$\bullet$};
\node (d) at (-.5,2) {$\bullet$};
\node at (popd)[below] {$\pop(D)$};
\node at (d)[above] {$D$};
\node (mB) at (1,6.5) {$\bullet$};
\node at (mB)[above] {$-B$};
\node (c) at (2,5.5) {$\bullet$};
\node at (c)[above] {$C$};
\node (popc) at (2,1.5) {$\bullet$};
\node at (popc)[below] {$\pop(C)$};
\node (B) at (0,0) {$\bullet$};
\node at (B)[below] {$B$};
\node (wpopd) at (2,2.5) {$\bullet$};
\node (wd) at (2,3.5) {$\bullet$};
\node at (wpopd)[below] {$\omega_C(\pop(D))$};
\node at (wd)[above] {$\omega_C(D)$};
\end{tikzpicture}
\caption{An illustration of~\Cref{thm:two_ways}.}
\label{fig:two_ways}
\end{figure}
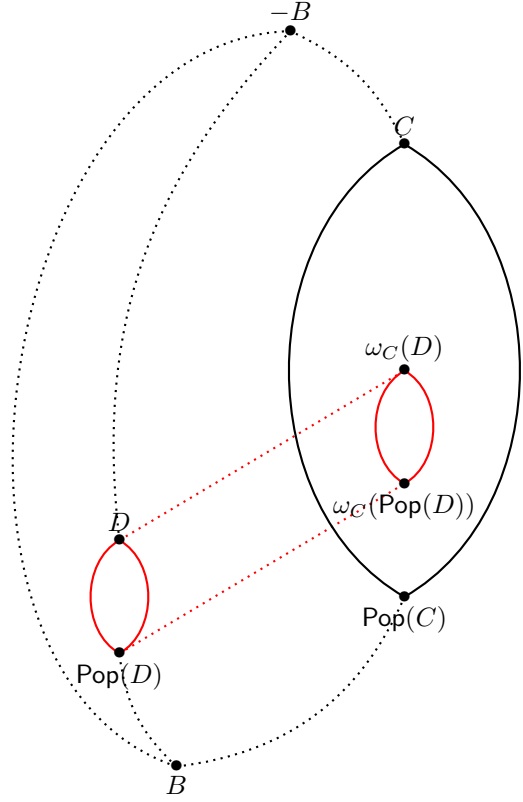

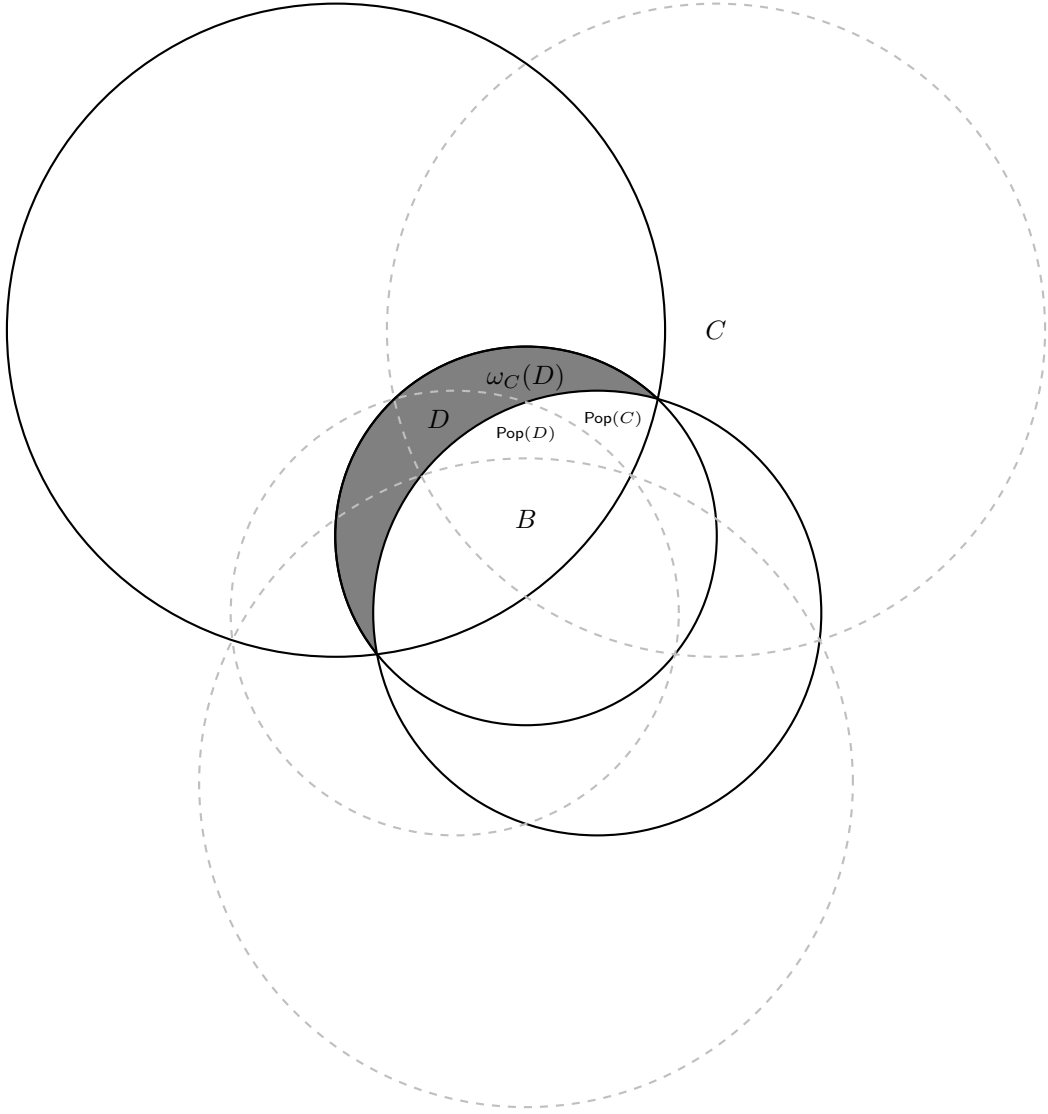
\begin{figure}[htbp]
\begin{tikzpicture}[scale=2.5]
\draw[black, thick,fill=gray] (0,-1/11) circle ({sqrt(122/121)});
\draw[black, thick,fill=white] (3/8,-1/2) circle ({sqrt(89/64)});
\draw[black, thick] (0,-1/11) circle ({sqrt(122/121)});
\draw[black, thick] (-1,1) circle ({sqrt(3)});
\draw[black, thick, lightgray,dashed] (1,1) circle ({sqrt(3)});
\draw[black, thick, lightgray,dashed] (-3/8,-1/2) circle ({sqrt(89/64)});
\draw[black, thick, lightgray,dashed] (0,-7/5) circle ({sqrt(74/25)});
\node at (1,1) {$C$};
\node at (.46,.53) {\tiny $\pop(C)$};
\node at (-.46,.53) {$D$};
\node at (0,.45) {\tiny $\pop(D)$};
\node at (0,.75) {$\omega_C(D)$};
\node at (0,0) {$B$};
\end{tikzpicture}
\caption{The braid arrangement $\HH$ of type $A_3$, drawn by intersecting hyperplanes with a sphere and applying a stereographic projection. Illustrating~\Cref{thm:two_ways}, we have indicated a region $C$, the region $\pop(C)$, a region $D\preceq C$, and the region $\omega_C(D)$.  The black hyperplanes are in $\HH_C$, while the gray dashed hyperplanes are not.  The gray shading indicates the region $\iota_C(D) \in \RR_C$.}
\label{fig:a3}
\end{figure}

\begin{lemma}\label{lem:omega_increasing}
Let $C,D\in\RR$ be such that $D\preceq C$. For every region $E\in[\pop(D),D]$, we have $E\leq\omega_C(E)$. 
\end{lemma}

\begin{proof}
It follows from \Cref{thm:two_ways} that $\covsha(D)\subseteq\covsha(\omega_C(D))$, so \[D=\bigvee_{\Sigma\in\covsha(D)}J_\Sigma\leq\bigvee_{\Sigma\in\covsha(\omega_C(D))}J_\Sigma=\omega_C(D)\] by \Cref{thm:shards}. 
Now let $E\xrightarrow{e_1}\cdots\xrightarrow{e_k}D$ be a positive minimal gallery from $E$ to $D$. Then \[\omega_C(E)\xrightarrow{\omega_C(e_1)}\cdots\xrightarrow{\omega_C(e_k)}\omega_C(D)\] is a positive minimal gallery from $\omega_C(E)$ to $\omega_C(D)$. 
For each $1\leq i\leq k$, we have $H_{e_i}=H_{\omega_C(e_i)}$ because $\Sigma(e_i)=\Sigma(\omega_C(e_i))$ by \Cref{thm:two_ways}. It follows that \[\inv(E)=\inv(D)\setminus\{H_{e_1},\ldots,H_{e_k}\}\subseteq\inv(\omega_C(D))\setminus\{H_{e_1},\ldots,H_{e_k}\}=\inv(\omega_C(E)).\] so $E\leq\omega_C(E)$.
\end{proof}

The map $(\iota_C^{\pop})^{-1}\colon\RR_C\to[\pop(C),C]$ is defined on regions in $\HH_C$; since a positive gallery $g$ in $\Sal(\HH_C)$ can be viewed as a sequence of regions, it makes sense to apply $(\iota_C^{\pop})^{-1}$ to all of $g$, thereby yielding a sequence of regions in the interval $[\pop(C),C]$. 

\begin{lemma}\label{lem:lift_up}
For each $C\in\RR$, there is an injective group homomorphism
\begin{align*}
    \Phi_C: \pi_1(\Sal(\HH_C),B_C) &\hookrightarrow \pi_1(\Sal(\HH),B) \\
    g &\mapsto \gal(B,\pop(C)) \cdot (\iota_C^\pop)^{-1}(g) \cdot \gal(B,\pop(C))^{-1}.
\end{align*}
If $D \preceq C$, then \[\Phi_C(\Crackle_C(\iota_C(D)))=\Crackle(D).\]
\end{lemma}
\begin{proof}
The first statement is immediate. Now suppose $D\preceq C$. Let \[\pop(D)=E_0\xrightarrow{e_1}E_{1}\xrightarrow{e_{2}}\cdots\xrightarrow{e_{k-1}}E_{k-1}\xrightarrow{e_k}E_k=D\] be a positive minimal gallery in $\Sal(\HH)$ from $\pop(D)$ to $D$, and let \[\omega_C(\pop(D))=\omega_C(E_0)\xrightarrow{\omega_C(e_1)}\omega_C(E_{1})\xrightarrow{\omega_C(e_{2})}\cdots\xrightarrow{\omega_C(e_{k-1})}\omega_C(E_{k-1})\xrightarrow{\omega_C(e_k)}\omega_C(E_k)=\omega_C(D)\] be the corresponding positive minimal gallery obtained by applying the isomorphism $\omega_C$ from \Cref{thm:two_ways}. We have \[(\iota_C^\pop)^{-1}(B_C)=\pop(C),\quad(\iota_C^\pop)^{-1}(\pop_C(\iota_C(D)))=\omega_C(\pop(D)),\quad\text{and}\quad(\iota_C^\pop)^{-1}(\iota_C(D))=\omega_C(D).\] Consequently, since $\Crackle_C(\iota_C(D))$ is equal to \[\gal(B_C,\pop_C(\iota_C(D)))\cdot\gal(\pop_C(\iota_C(D)),\iota_C(D),\pop_C(\iota_C(D)))\cdot\gal(B_C,\pop_C(\iota_C(D)))^{-1} \] by definition, we can write $(\iota_C^\pop)^{-1}(\Crackle_C(\iota_C(D)))$ as \[\gal(B,\omega_C(\pop(D)))\cdot\gal(\omega_C(\pop(D)),\omega_C(D),\omega_C(\pop(D)))\cdot\gal(B,\omega_C(\pop(D)))^{-1}.
    \] Setting $E'=\omega_C(\pop(D))$ and $E=\omega_C(D)$ in \Cref{lem:crackle_shard_generators} allows us to rewrite this gallery as \[\tt_{\Sigma(\omega_C(e_k))}\tt_{\Sigma(\omega_C(e_{k-1}))}\cdots\tt_{\Sigma(\omega_C(e_1))},\] and we know by \Cref{thm:two_ways} that this is equal to \[\tt_{\Sigma(e_k)}\tt_{\Sigma(e_{k-1})}\cdots\tt_{\Sigma(e_1)}.\] This completes the proof since \Cref{cor:crackle_shard_generators} tells us that $\tt_{\Sigma(e_k)}\tt_{\Sigma(e_{k-1})}\cdots\tt_{\Sigma(e_1)}=\Crackle(D)$.  
\end{proof}

\subsection{Proof of \texorpdfstring{\Cref{thm:crackle_hom}}{Theorem 1.3}}

Recall that we use \Cref{thm:shards_and_join_irrs} to write $J_\Sigma$ for the join-irreducible region in $\Weak(\HH,B)$ corresponding to the shard $\Sigma$.

\begin{proposition}\label{lem:all_shards}
Let $\Sigma\in\Sha(\HH,B)$ be a shard, and let $C\in\RR$ be a region. Then $J_\Sigma\preceq C$ if and only if there exists an $\TTsha$-word representing $\Crackle(C)$ that uses the shard loop $\tt_\Sigma$. 
\end{proposition}

\begin{proof}
First suppose $J_\Sigma\preceq C$. By \Cref{prop:shard_intersection_via_labels}, there is an edge $D'\xrightarrow{e}D$ with $\pop(C)\leq D'\lessdot D\leq C$ such that $\Sigma(e)=\Sigma$. We can find a positive minimal gallery from $\pop(C)$ to $C$ that uses $e$, so it follows from~\Cref{cor:crackle_shard_generators} that there is an $\TTsha$-word representing $\Crackle(C)$ that uses $\tt_\Sigma$. 

To prove the converse, suppose we can write $\Crackle(C)=\tt_{\Sigma_1}\cdots\tt_{\Sigma_m}$, where $\Sigma_i=\Sigma$ for some $1\leq i\leq m$. Let $\pop(C)=E_0\xrightarrow{e_1}E_{1}\xrightarrow{e_{2}}\cdots\xrightarrow{e_{k-1}}E_{k-1}\xrightarrow{e_k}E_k=C$ be a positive minimal gallery in $\Sal(\HH)$ from $\pop(C)$ to $C$; by \Cref{cor:crackle_shard_generators}, we have $\Crackle(C)=\tt_{\Sigma(e_k)}\cdots\tt_{\Sigma(e_1)}$. When we pass to the abelianization $H_1(\PP(\HH),B)$, we obtain the equality $\overline\tt_{H_{\Sigma_1}}\cdots\overline\tt_{H_{\Sigma_m}}=\overline\tt_{H_{\Sigma(e_1)}}\cdots\overline\tt_{H_{\Sigma(e_m)}}$. \Cref{thm:abelianization} tells us that $H_1(\PP(\HH),B)$ is a free abelian group with free generating set $\{\overline\tt_H\}_{H\in\HH}$, so $H_{\Sigma_i}$ must be one of the hyperplanes $H_{\Sigma(e_1)},\ldots,H_{\Sigma(e_k)}$. In other words, there exists $1\leq j\leq k$ such that the shards $\Sigma=\Sigma_i$ and $\Sigma(e_j)$ belong to the same hyperplane $H_{\Sigma}$. It follows from \Cref{prop:shard_intersection_via_labels} that $J_{\Sigma(e_j)}\preceq C$.

Suppose by way of contradiction that $J_\Sigma\not\preceq C$. Since $\covsha(J_\Sigma)=\{\Sigma\}$, it follows from the definition of the shard intersection order that there exists a point $x\in\bigcap\covsha(C)\setminus\Sigma$. On the other hand, since $J_{\Sigma(e_j)}\preceq C$, we have $x\in\Sigma(e_j)$. The shard $\Sigma$ is a polyhedral cone defined as the intersection of $H_\Sigma$ with some closed half-spaces, where each closed half-space is determined by a hyperplane that cuts $H_\Sigma$. Because $x\not\in\Sigma$, at least one of these closed half-spaces does not contain $x$. In other words, there exists a hyperplane $H\in\HH$ that cuts $H_\Sigma$ such that $\Sigma$ and $x$ are on opposite sides of $H$ and $x\not\in H$. There exists a region $A\in\RR$ such that $\HH_A$ is the full rank-$2$ subarrangement of $\HH$ containing $H_\Sigma$ and $H$. Then $H_\Sigma$ is a basic hyperplane of $\HH_A$; let $H'$ be the other basic hyperplane of $\HH_A$. Then $H'$ also cuts $H_\Sigma$, and $\Sigma$ and $x$ are on opposite sides of $H'$ (with $x\not\in H'$). Since $x\in\Sigma(e_j)$, we find that $\Sigma$ and $\Sigma(e_j)$ are on opposite sides of $H$ and on opposite sides of $H'$. Either $H$ or $H'$ is in $\inv(J_\Sigma)$; without loss of generality, say $H\in\inv(J_\Sigma)$. Then $\Sigma$ and $B$ are on opposite sides of $H$, so $x$ and $B$ must be on the same side of $H$. Because $x\in\bigcap\covsha(C)$, this means that $C$ and $B$ are on the same side of $H$; that is, $H\not\in\inv(C)$. 

Because $H_\Sigma\in\HH_A$ and $\Sigma$ and $\Sigma(e_j)$ are contained in different shards of $\HH_A$, it follows from \Cref{thm:shard_generators} and \Cref{lem:subarrangement_quotient} that $\iotab_A(\tt_{\Sigma})$ and $\iotab_A(\tt_{\Sigma(e_j)})$ are generators in $\TTsha(\HH_A,B_A)$ and that 
\begin{equation}\label{eq:distinct}
\iotab_A(\tt_{\Sigma})\neq\iotab_A(\tt_{\Sigma(e_j)}).
\end{equation} Since $J_{\Sigma(e_j)}\preceq C$, we have $H_\Sigma\in\inv(C)$ and $H\not\in\inv(C)$, so $\iota_A(C)$ is join-irreducible in $\Weak(\HH_A,B_A)$. This means that $\Crackle_A(\iota_A(C))$ is a single generator in $\TTsha(\HH_A,B_A)$. According to \Cref{lem:subarrangement_quotient}, $\iotab_A(\Crackle(C))$ is a single generator in $\TTsha(\HH_A,B_A)$. We have assumed that there is an expression $\Crackle(C)=\tt_{\Sigma_1}\cdots\tt_{\Sigma_m}$ with $\Sigma=\Sigma_i$; if we apply the quotient map $\iotab_A$ to this expression, then, by \Cref{lem:quotient_sec4}, we obtain an expression for $\iotab_A(\Crackle(C))$ as a product of elements of $\TTsha(\HH_A,B_A)\cup\{\id\}$ such that one of the elements is $\iotab_A(\tt_\Sigma)$. Hence, $\iotab_A(\tt_\Sigma)=\iotab_A(\Crackle(C))$. On the other hand, $J_{\Sigma(e_j)}\preceq C$, so we know by the first paragraph of the proof that there is an $\TTsha$-word for $\Crackle(C)$ using $\tt_{\Sigma(e_j)}$. Applying $\iotab_A$ to this expression yields an expression for $\iotab_A(\Crackle(C))$ as a product of elements of $\TTsha(\HH_A,B_A)\cup\{\id\}$ such that one of the elements is $\iotab_A(\tt_{\Sigma(e_j)})$. Hence, $\iotab_A(\tt_{\Sigma(e_j)})=\iotab_A(\Crackle(C))$. This shows that $\iotab_A(\tt_{\Sigma})=\iotab_A(\tt_{\Sigma(e_j)})$, which contradicts \Cref{eq:distinct}. 
\end{proof}

We can now complete the proof of~\Cref{thm:crackle_hom}. We saw in \Cref{cor:crackle_shard_generators} that $\Crackle$ maps $\Shard(\HH,B)$ into the interval $[\id,\Delta^2]_{\PPP^+}$; we are left to show that it is a poset embedding. 

Suppose first that $C,D\in\RR$ are such that $D\preceq C$. We can restrict to the subarrangement $\HH_C$ and note that $\iota_C(C)=-B_C$. By \Cref{cor:crackle_shard_generators}, $\Crackle_C(\iota_C(D)) \leq \Crackle_C(\iota_C(C))$ in the pure shard monoid $\PPP^+(\HH_C,B_C)$. Thus, there exist $\tt_{e_1},\ldots,\tt_{e_k},\tt_{e_{k+1}},\ldots,\tt_{e_m}$ in
$\TTsha(\HH_C,B_C)$ such that \[\tt_{e_1}\cdots\tt_{e_k}=\Crackle_C(\iota_C(D))\quad\text{and}\quad\tt_{e_1}\cdots\tt_{e_k}\tt_{e_{k+1}}\cdots\tt_{e_m}=\Crackle_C(\iota_C(C)).\] For 
$1\leq i\leq m$, note that $\Phi_C(\tt_{e_i})=\gal(B,\pop(C))\cdot(\iota_C^\pop)^{-1}(\tt_{e_i})\cdot\gal(B,\pop(C))^{-1}$ is one of the shard loops in $\TTsha(\HH,B)$. Because $D\preceq C$, \Cref{lem:lift_up} tells us that \[\Crackle(D)=\Phi_C(\Crackle_C(\iota_C(D)))=\Phi_C(\tt_{e_1})\cdots\Phi_C(\tt_{e_k})\] and \[\Crackle(C)=\Phi_C(\Crackle_C(\iota_C(C)))=\Phi_C(\tt_{e_1})\cdots\Phi_C(\tt_{e_k})\Phi_C(\tt_{e_{k+1}})\cdots\Phi_C(\tt_{e_m}).\] This proves that $\Crackle(D)\leq \Crackle(C)$ in $\PPP^+(\HH,B)$. 

To prove the converse, assume $\Crackle(D) \leq \Crackle(C)$ in $\PPP^+(\HH,B)$. Then every $\TTsha$-word representing $\Crackle(D)$ can be extended to an $\TTsha$-word representing $\Crackle(C)$.  According to~\Cref{lem:all_shards}, we have \[\{\Sigma\in\Sha:J_\Sigma\preceq D\}\subseteq\{\Sigma\in\Sha:J_{\Sigma}\preceq C\}.\] It follows from \Cref{prop:shard_intersection_via_labels} that \[\{\Sigma\in\Sha:J_\Sigma\preceq D\}=\Sigma([\pop(D),D])\quad\text{and}\quad\{\Sigma\in\Sha:J_\Sigma\preceq C\}=\Sigma([\pop(C),C]).\] Therefore, \Cref{prop:shard_intersection_via_labels} tells us that $D\preceq C$, as desired. 

\section{Snap}\label{sec:snap}

We now specialize to the setting of reflection arrangements of finite Coxeter groups and prove~\Cref{thm:shard_generators2} and \Cref{thm:snap}.  We will also collect corollaries specializing these results to sortable elements and noncrossing partitions.

\subsection{Coxeter Groups and Braid Groups}
Let $\HH$ be the reflection arrangement of a finite Coxeter group $W$. Then $\HH$ is simplicial, and $W$ can be seen as the group of orthogonal transformations of $\mathbb R^n$ generated by the reflections through the hyperplanes in $\HH$. There is bijection between $W$ and $\RR$ that maps an element $w\in W$ to the region $w(B)$, where $B$ is the fixed base region of $\HH$. We use this bijection to identify regions of $\HH$ with elements of $W$. 

Let $S$ be the set of simple reflections of $W$. Then $W$ has a presentation of the form \[\langle S:(ss')^{m(s,s')}=\id\text{ for all }s,s'\in S\rangle,\] where $\id$ is the identity element of $W$, $m(s,s)=1$ for all $s\in S$, and $m(s,s')=m(s',s)\in\{2,3,\ldots\}\cup\{\infty\}$ for all distinct $s,s'\in S$. Given symbols $\alpha,\beta$ and a nonnegative integer $r$, we write $[\alpha\vert\beta]_r$ for the word $\alpha\beta\alpha\cdots$ of length $r$ that starts with $\alpha$ and alternates between $\alpha$ and $\beta$. Thus, the braid relations of $W$ can be written as $[s\vert s']_{m(s,s')}=[s'\vert s]_{m(s,s')}$. For each $s\in S$. let ${\bf s}$ be a formal copy of $s$. The braid group $\BB(W)$ has a generating set ${\bf S}=\{{\bf s}:s\in S\}$ and presentation \[\langle {\bf S}:[\s\vert \s']_{m(s,s')}=[\s'\vert \s]_{m(s,s')}\text{ for all distinct }\s,\s'\in{\bf S}\rangle.\] Thus, the generators in $\bf S$ have infinite order in $\BB(W)$, while the elements of $S$ are involutions in $W$. The \defn{positive braid monoid} of $W$ is the monoid $\BB^+(W)$ generated by ${\bf S}$. There is a natural homomorphic quotient map $\varphi\colon \BB(W)\to W$ defined by $\varphi({\bf s})=s$ for all $s\in S$. The \defn{pure braid group} of $W$, denoted $\PPP(W)$, is the kernel of this map: $\PPP(W)=\ker(\varphi)$. 

A \defn{reduced word} for an element $w\in W$ is a word in the alphabet $S$ that represents $w$ and has minimum length among all such words. The length of a reduced word for $w$ is called the \defn{length} of $w$ and is denoted by $\len(w)$. The (right) \defn{weak order} on $W$ is defined by saying $u\leq v$ if there is a reduced word for $v$ that contains a reduced word for $u$ as a prefix. This defines a poset $\Weak(W)$, which coincides with $\Weak(\HH,B)$ under the identification of $W$ with $\RR$. For each $w\in W$, there is a natural \defn{lift} ${\bf w}\in \BB^+(W)$ obtained by taking a reduced word for $w$ and replacing each simple reflection $s$ with the corresponding generator ${\bf s}$. We often denote the lift of an element of $W$ using bold, but we will also sometimes write $\lift(w)$ for the lift of $w$ when we are wary about the possibility of confusion. An \defn{${\bf S}$-word} for an element $\w\in \BB^+(W)$ is a word in the alphabet ${\bf S}$ that represents $\w$. We can also define the \defn{weak order} on $\BB^+(W)$ by saying ${\bf u}\leq {\bf v}$ if there is an ${\bf S}$-word for ${\bf v}$ that contains an ${\bf S}$-word for ${\bf u}$ as a prefix. We write $\Weak(\BB^+(W))$ for the poset $(\BB^+(W),\leq)$. For ${\bf u},{\bf v}\in \BB^+(W)$ with ${\bf u}\leq{\bf v}$, we write $[{\bf u},{\bf v}]_{\BB^+}$ for the interval between $\bf u$ and $\bf v$ in $\Weak(\BB^+(W))$. 

A \defn{reflection} in $W$ is an element of the form $usu^{-1}$, where $u\in W$ and $s\in S$. Reflections are precisely the elements that, when viewed as orthogonal transformations of $\mathbb R^n$, are reflections through hyperplanes in $\HH$. An \defn{inversion} of an element $w\in W$ is a reflection $t$ in $W$ such that $\len(tw)<\len(w)$; such an inversion is called a \defn{cover reflection} of $w$ if $tw$ is covered by $w$ in the weak order.  Let $\inv(w)$ and $\cov(w)$ be the set of inversions and the set of cover reflections of $w$, respectively. Write $\langle \cov(w)\rangle$ for the parabolic subgroup of $W$ generated by $\cov(w)$. 

For $u,v\in W$, we have $u\leq v$ in the weak order if and only if $\inv(u)\subseteq\inv(v)$. Furthermore, a reformulation of \Cref{thm:shard_intersection_via_hyperplanes} in this context states that \begin{equation}\label{eq:shard_reformulation}
    u \preceq v\quad\text{if and only if}\quad \inv(u) \subseteq \inv(v)\text{ and }\langle \cov(u)\rangle \subseteq \langle \cov(v)\rangle. 
\end{equation}

\subsection{Pop and Crackle for Coxeter Groups}

A \defn{descent} of an element $w\in W$ is a simple reflection $s$ such that there is a reduced word for $w$ that ends with $s$. Similarly, a descent of an element ${\bf w}\in \BB^+(W)$ is a generator ${\bf s}\in{\bf S}$ such that there is an ${\bf S}$-word for ${\bf w}$ that ends with ${\bf s}$. We write $\des(w)$ and $\des({\bf w})$ for the set of descents of $w$ and the set of descents of ${\bf w}$, respectively. Given a subset $J\subseteq S$, we write $w_\circ(J)$ for the longest element of the parabolic subgroup of $W$ generated by $J$. Since $\Weak(W)$ and $\Weak(\BB^+(W))$ are locally finite meet-semilattices, they come equipped with pop-stack sorting operators; we denote both of these operators by $\pop$. If $w\in W$ and ${\bf w}=\lift(w)$, then $\des({\bf w})$ is the set of lifts of descents of $w$. We have $\pop(w)=ww_\circ(\des(w))$, so it follows that \[\pop({\bf w})={\bf w} \cdot ({\bf w}_\circ(\des(w)))^{-1}=\lift(\pop(w)),\] where ${\bf w}_\circ(\des(w))=\lift(w_\circ(\des(w)))$.

There is a natural action of $W$ on $\mathbb C^n\setminus\HH_{\mathbb C}$, and the braid group $\BB(W)$ is isomorphic to the fundamental group $\pi_1((\CC^n \setminus \HH_\CC)/W,x_B)$. The pure braid group $\PPP(W)$ is isomorphic to $\pi_1(\CC^n \setminus \HH_\CC,x_B)$. The identification of $W$ with $\RR$ allows us to naturally label the edges of $\PP(\HH)$ by the generators in $\mathbf{S}$. More precisely, if $w\in W$ and $s\in S\setminus\des(w)$, then we label the edges $w\xrightarrow{e} ws$ and $ws\xrightarrow{e^*}w$ with the generator ${\bf s}\in{\bf S}$. This labeling allows us to rephrase notions and results about crackle maps from \Cref{sec:crackle} in the language of braid groups. For example, if $w\in W$ and $\w=\lift(w)$, then 
\begin{align*}\Crackle(w) &= \pop(\w)\cdot  \left(\w_\circ(\des(w))\right)^2 \cdot \pop(\w)^{-1} \in \PPP^+(W) \subseteq \BB(W),\end{align*} where 
$\w_\circ(\des(w))=\lift(w_\circ(\des(w)))$. 
\begin{proposition}
The pure shard monoid $\PPP^+(W)$ is the submonoid of $\PPP(W)$ generated by \[\{\Crackle(j) : j \in W, |\des(j)|=1\}.\] 
\end{proposition}
\begin{proof}
    An element $j\in W$ has exactly 1 descent if and only if it is join-irreducible in $\Weak(W)$. \Cref{thm:shards_and_join_irrs} tells us that there is bijection $\Sigma\mapsto J_\Sigma$ from $\Sha(\HH,B)$ to the set of join-irreducible elements of $W$ (which we identify with regions of $\HH$). Moreover, $\Crackle(J_\Sigma)=\tt_\Sigma$. Hence, the result follows from the original definition of the pure shard monoid. 
\end{proof}

By \Cref{thm:crackle_hom}, $\Crackle$ is a poset embedding from $\Shard(W)$ into $\PPP^+(W)$.

\subsection{Snap}
Let $w_\circ$ be the long element of $W$. The standard lift map $\lift\colon W\to \BB^+(W)$ is a poset isomorphism from $\Weak(W)$ to the interval $[\id,\w_\circ]_{\BB^+}$ in $\Weak(\BB^+(W))$, where $\w_\circ=\lift(w_\circ)$.  We now define a nonstandard lift from $W$ to $\BB^+(W)$.

\begin{definition}\label{def:snap}
The \defn{snap map} $\Snap\colon W \to \BB^+(W)$ is defined by
\begin{align*} \Snap(w)  &:= \pop(\w) \cdot (\w_\circ(\des(w)))^2, \end{align*} where $\w=\lift(w)$ and $\w_\circ(\des(w))=\lift(w_\circ(\des(w)))$.
\end{definition}

Observe that $\Snap(w)=\w\cdot\w_\circ(\des(w))$ and that $\varphi(\Snap(w))=\pop(w)$.  Our aim in this section is to prove \Cref{thm:snap}, which states that $\Snap$ is a poset embedding from $\Shard(W)$ into $[\id,\Delta^2]_{\BB^+}$.  This is illustrated in~\Cref{fig:e_to_delta_B}, which shows $[\id,\Delta^2]_{\BB^+}$ when $W=I_2(4)$ is the dihedral group of order $8$.  

The first step in this endeavor is the following lemma, which will allow us to consider inversion multisets of elements of $\BB^+(W)$. 

\begin{lemma}\label{lem:inv_well_defined}
Let ${\bf w}\in \BB^+(W)$, and let ${\bf s}_{1}\cdots{\bf s}_{r}$ and ${\bf s}_{1}'\cdots{\bf s}_{r}'$ be two ${\bf S}$-words for $\w$. The multisets \[\{s_{1}\cdots s_{k-1}s_{k}s_{k-1}\cdots s_{1}:1\leq k\leq r\}\quad\text{and}\quad\{s_{1}'\cdots s_{k-1}'s_{k}'s_{k-1}'\cdots s_{1}':1\leq k\leq r\}\] are equal.
\end{lemma}

\begin{proof}
It suffices to prove the result when the two ${\bf S}$-words differ by a braid move. In fact, it suffices to prove the result when the braid move changes the entire first word into the entire second word. Thus, we can assume the first word is $[{\bf s}\vert {\bf s}']_{m(s,s')}$ and the second word is $[{\bf s}'\vert {\bf s}]_{m(s,s')}$. In this case, the result follows from the observation that $[s\vert s']_{k}([s\vert s']_{k-1})^{-1}=[s'\vert s]_{m(s,s')+1-k}([s'\vert s]_{m(s,s')-k})^{-1}$ for all $1\leq k\leq m(s,s')$. 
\end{proof}

For each $\w\in \BB^+(W)$, \Cref{lem:inv_well_defined} allows us to define the multiset \[\Inv(\w)=\{s_{1}\cdots s_{k-1}s_{k}s_{k-1}\cdots s_{1}:1\leq k\leq r\},\] where ${\bf s}_{1}\cdots{\bf s}_{r}$ is an ${\bf S}$-word for $\w$. If $\w=\lift(w)$ for some $w\in W$, then $s_{1}\cdots s_{r}$ is a reduced word for $w$; in this case, the multiset $\Inv(\w)$ is actually a set, and it is equal to the inversion set $\inv(w)$. 

\begin{lemma}\label{lem:Inv}
Let $w\in W$. Let $s_{1}\cdots s_{m}$ be a reduced word for $\pop(w)$, and let $s_{m+1}\cdots s_{r}$ be a reduced word for $w_\circ(\des(w))$ so that $s_{1}\cdots s_{r}$ is a reduced word for $w$. Let $t_k=s_{1}\cdots s_{k-1}s_{k}s_{k-1}\cdots s_{1}$. Then $\Inv(\Snap(w))=\{t_1,\ldots,t_m,t_{m+1},\ldots,t_r,t_r,\ldots,t_{m+1}\}$. Furthermore, $\inv(w)=\{t_1,\ldots,t_r\}$, and $t_{m+1},\ldots,t_r$ are precisely the reflections appearing in the parabolic subgroup $\langle \cov(w)\rangle$. 
\end{lemma}

\begin{proof}
Since $w_\circ(\des(w))$ is an involution, it also has $s_{r}\cdots s_{m+1}$ as a reduced word, so \[\s_{1}\cdots\s_{m}\s_{m+1}\cdots\s_{r}\s_{r}\cdots\s_{m+1}\] is an ${\bf S}$-word for $\Snap(w)$. Therefore, to prove the first statement, we just need to show that for each $m+1\leq k\leq r$, the reflection $t_k$ is equal to $s_{1}\cdots s_{r}s_{r}\cdots s_{k+1}s_{k}s_{k+1}\cdots s_{r}s_{r}\cdots s_{1}$. We can group the terms in this expression as \[
    s_{1}\cdots s_{k-1}(s_{k}\cdots s_{r})(s_{r}\cdots s_{k})(s_{k+1}\cdots s_{r})(s_{r}\cdots s_{k+1})s_{k}\cdots s_{1}\] to make it clear that it is indeed equal to $t_k$.
    
It is well known that $\inv(w)=\{t_1,\ldots,t_r\}$. For the last statement, note that, in the language of hyperplanes, the reflections $t_{m+1},\ldots,t_r$ correspond to the hyperplanes that separate the region $\pop(w)$ from the region $w$; these are the reflections in $\langle\cov(w)\rangle$. 
\end{proof}

\begin{example}
Suppose $W$ is the symmetric group $\mathfrak S_4$ so that $S=\{s_1,s_2,s_3\}$, where $s_i$ is the simple transposition $(i\,\, i+1)$. We have $m(s_1,s_2)=m(s_2,s_3)=3$ and $m(s_1,s_3)=2$. Let $w=s_1s_2s_3s_2$. Then $\des(w)=\{s_2,s_3\}$, $w_\circ(\des(w))=s_2s_3s_2$, and $\pop(w)=s_1$. Now, \[\Snap(w)=\w\,\w_\circ(\des(w))=\s_1\s_2\s_3\s_2\s_2\s_3\s_2,\] so \begin{align*}
\Inv(\Snap(w))&=\{s_1,s_1s_2s_1, s_1s_2s_3s_2s_1, s_1s_2s_3s_2s_3s_2s_1,s_1s_2s_3s_2s_2s_2s_3s_2s_1,s_1s_2s_3s_2s_2s_3s_2s_2s_3s_2s_1, \\ &\hspace{0.63cm}s_1s_2s_3s_2s_2s_3s_2s_3s_2s_2s_3s_2s_1\} \\ 
&= \{s_1,s_1s_2s_1,s_1s_2s_3s_2s_1,s_1s_2s_3s_2s_3s_2s_1,s_1s_2s_3s_2s_3s_2s_1,s_1s_2s_3s_2s_1,s_1s_2s_1\} \\ 
&= \{(1\,\,2),(1\,\,3),(1\,\,4),(3\,\,4),(3\,\,4),(1\,\,4),(1\,\,3)\}.
\end{align*}
\end{example}

\subsection{Proof of \texorpdfstring{\Cref{thm:snap}}{Theorem 1.5}}

We are now going to prove that $\Snap$ is a poset embedding of $\Weak(W)$ into $\Weak(\BB^+(W))$. We know that $\Snap(w_\circ)=\w_\circ^2=\Delta^2$, where $w_\circ$ is the long element of $W$, $\w_\circ$ is the lift of $w_\circ$, and $\Delta^2$ is the full twist. Therefore, it will follow immediately that $\Snap(W)$ is contained in the interval $[\id,\Delta^2]_{\BB^+}$.
 
Suppose first that $u,v\in W$ are such that $\Snap(u)\leq\Snap(v)$. Then there is an ${\bf S}$-word for $\Snap(v)$ that contains an ${\bf S}$-word for $\Snap(u)$ as a prefix. This implies that the multiset $\Inv(\Snap(u))$ is contained in the multiset $\Inv(\Snap(v))$, so it follows from the last sentence in \Cref{lem:Inv} that $\inv(u)\subseteq\inv (v)$ and $\langle\cov(u)\rangle\subseteq\langle\cov(v)\rangle$. According to \Cref{eq:shard_reformulation}, we have $u\preceq v$. 

To prove the converse, let us assume that $u,v\in W$ are such that $u\preceq v$. Let \[\pop(u)=u_0\lessdot u_1\lessdot\cdots\lessdot u_k=u\] be a saturated chain in $\Weak(W)$ from $\pop(u)$ to $u$. Let $t_i=u_iu_{i-1}^{-1}$, and let $p=t_k\cdots t_1=u\,\pop(u)^{-1}$. Each $t_i$ is the reflection through the hyperplane $H_{\Sigma(u_{i-1}\lessdot u_i)}$. Identifying $v$ with a region of $\HH$, we can consider the poset isomorphism $\omega_v$ from \Cref{thm:two_ways}. That theorem tells us that \[H_{\Sigma(u_{i-1}\lessdot u_i)}=H_{\Sigma(\omega_v(u_{i-1})\lessdot\omega_v(u_i))},\] so $t_i=\omega_v(u_i)\omega_v(u_{i-1})^{-1}$. This shows that $\omega_v(u)=p\,\omega_v(\pop(u))$, so
\begin{equation}\label{eq:snap_proof1}
u\,\pop(u)^{-1}=\omega_v(u)(\omega_v(\pop(u)))^{-1}.
\end{equation}

Since $\omega_v$ is a poset isomorphism by \Cref{thm:two_ways}, we have $\omega_v(\pop(u))\leq \omega_v(u)$. This means there exists $x\in W$ such that $\omega_v(u)=\omega_v(\pop(u))\,x$ and
\begin{equation}\label{eq:snap_proof2}
\lift(\omega_v(u))=\lift(\omega_v(\pop(u)))\,{\bf x}.
\end{equation}
Note that $\len(x)=\len(\omega_v(u))-\len(\omega_v(\pop(u)))=\len(u)-\len(\pop(u))=\len(w_\circ(\des(u)))$. \Cref{lem:omega_increasing} tells us that $\pop(u)\leq\omega_v(\pop(u))$, so there exists $y\in W$ such that $\omega_v(\pop(u))=\pop(u)\,y$ and
\begin{equation}\label{eq:snap_proof3}
\lift(\omega_v(\pop(u)))=\lift(\pop(u))\,{\bf y}.
\end{equation}
Then 
\begin{align*}
\pop(u)yxy^{-1}\pop(u)^{-1}&=\omega_v(\pop(u))x(\omega_v(\pop(u)))^{-1} \\ 
&=\omega_v(u)(\omega_v(\pop(u)))^{-1} \\
&=u\,\pop(u)^{-1} \\
&=\pop(u)w_\circ(\des(u))\pop(u)^{-1},
\end{align*}
where we have used \Cref{eq:snap_proof1}. Rearranging this equation yields $yx=w_\circ(\des(u))y$. Because $\len(\pop(u)yx)=\len(\omega_v(u))=\len(\omega_v(\pop(u)))+\len(x)=\len(\pop(u))+\len(y)+\len(x)$, we have $\len(yx)=\len(y)+\len(x)=\len(y)+\len(w_\circ(\des(u)))$. It follows that \[{\bf yx}=\lift(yx)=\lift(w_\circ(\des(u))y)={\bf w}_\circ(\des(u)){\bf y},\] where $\w_\circ(\des(u))=\lift(w_\circ(\des(u)))$.  Consequently, 
\begin{align*}
\Snap(u) &= {\bf u}\,{\bf w}_\circ(\des(u)) \\
&= \lift(\pop(u))\w_\circ(\des(u))^2 \\
&\leq \lift(\pop(u))\w_\circ(\des(u))^2\,{\bf y} \\
&=\lift(\pop(u))\w_\circ(\des(u))\,{\bf yx} \\
&=\lift(\pop(u))\,{\bf yx}^2.
\end{align*}
Invoking \Cref{eq:snap_proof2,eq:snap_proof3}, we find that $\lift(\pop(u))\,{\bf yx}^2=\lift(\omega_v(\pop(u)))\,{\bf x}^2=\lift(\omega_v(u))\,{\bf x}$. Therefore, to prove that $\Snap(u)\leq\Snap(v)$, we just need to show that $\lift(\omega_v(u))\,{\bf x}\leq\Snap(v)$. 

We defined $x$ via the equation $\omega_v(u)=\omega_v(\pop(u))\,x$; since $\omega_v(\pop(u))$ and $\omega_v(u)$ are both in the interval $[\pop(v),v]$ (by the definition of $\omega_v$ in \Cref{thm:two_ways}), this implies that $x$ is in the parabolic subgroup $\langle\des(v)\rangle$. It follows that ${\bf x}\leq{\bf w}_\circ(\des(v))$, so $\lift(\omega_v(u))\,{\bf x}\leq\lift(\omega_v(u))\,\w_\circ(\des(v))$. Therefore, the proof will be complete if we can demonstrate that $\lift(\omega_v(u))\,\w_\circ(\des(v))\leq\Snap(v)$. 

Because $\omega_v(u)\leq v$, there exists $z\in \langle\des(v)\rangle$ such that $v=\omega_v(u)\,z$ and ${\bf v}=\lift(\omega_v(u))\,{\bf z}$. It is well known that for every $s\in\des(v)$, there exists $s'\in\des(v)$ such that $\s\,\w_\circ(\des(v))=\w_\circ(\des(v))\,\s'$. It follows that there exists ${\bf z}'\in \BB^+(W)$ such that ${\bf z}\,\w_\circ(\des(v))=\w_\circ(\des(v))\,{\bf z}'$. Finally, 
\begin{align*}
\lift(\omega_v(u))\,\w_\circ(\des(v))&\leq \lift(\omega_v(u))\,\w_\circ(\des(v))\,{\bf z}' \\ 
&=\lift(\omega_v(u))\,{\bf z}\,\w_\circ(\des(v)) \\
&={\bf v}\,\w_\circ(\des(v)) \\
&=\Snap(v),
\end{align*}
as desired. 

\subsection{Catalan Combinatorics}
\label{sec:noncrossing}

A \defn{standard Coxeter element} $c$ of $W$ is a product of the simple reflections in some order; fix a reduced expression for $c$ as $s_1s_2\cdots s_n$.  The \defn{$c$-sorting word} $\w(c)$ of $\w \in \BB^+(W)$ is the lexicographically minimal subword of the $\bf S$-word $(\s_1\s_2\cdots\s_n)^\infty$ that represents $\w$~\cite{reading2007sortable,stump2015cataland}.  Write $\w(c) = \w_1 \w_2 \cdots \w_k$, with each $\w_i$ a subword of $\s_1 \s_2\cdots \s_n$.  We say $\w$ is \defn{$c$-sortable} if the set of letters appearing in $\w_i$ contains the set of letters appearing in $\w_{i+1}$ for all $1\leq i\leq k-1$. Let $m$ be a positive integer. Following~\cite{reading2007sortable,stump2015cataland}, we write $\Sort^m(W,c)$ for the set of all $c$-sortable elements of $\BB^+(W)$ in the interval $[\id,\w_\circ^m]_{\BB^+}$. Then $(\Sort^m(W,c),\leq)$ is the $m$-th \defn{$c$-Fuss--Cambrian lattice} of type $W$---in particular~\cite{galashin2022rational}, \[\left| \Sort^m(W,c)\right| = \mathrm{Cat}^m(W) := \prod_{i=1}^n \frac{mh+d_i}{d_i},\] where $h$ is the \emph{Coxeter number} of $W$ and $d_1,\ldots,d_n$ are the \emph{degrees} of $W$. Write ${\Sort(W,c)=\Sort^1(W,c)}$; under the usual identification of the interval $[\id,\w_\circ]_{\BB^+}$ in $\Weak(\BB^+(W))$ with $\Weak(W)$, the set $\Sort(W,c)$ forms a sublattice of $\Weak(W)$.

Continuing to identify $[\id,\w_\circ]_{\BB^+}$ with $\Weak(W)$, we find that the $c$-corting word $\w_\circ(c)$ naturally defines a positive minimal gallery from $\id$ (identified with $B$) to $\w_\circ$ (identified with $-B$); we say that a shard is \defn{$c$-noncrossing} if it labels one of the edges in this gallery.
The \defn{$c$-noncrossing partition lattice} $\NC(W,c)$ is the interval from $\id$ to $c$ in the absolute order on $W$. In~\cite{reading2007clusters}, Reading gave a beautiful bijection $\mathrm{nc}_c: \Sort(W,c) \to \NC(W,c)$, which is a poset isomorphism from $(\Sort(W,c),\preceq)$ to $\NC(W,c)$~\cite[Theorem 8.5]{reading2011noncrossing}.

We obtain the following corollaries of~\Cref{thm:crackle_hom,thm:snap}, establishing similar relationships between the noncrossing partition lattice and the images of the set of sortable elements under the crackle and snap maps.

\begin{corollary}\label{cor:nc2} The map $\Crackle$ restricts to a poset embedding of $(\Sort(W,c),\preceq)$ into $[\id,\Delta^2]_{\PPP^+}$.  In particular, the subposet $(\Crackle(\Sort(W,c)),\leq)$ of $[\id,\Delta^2]_{\PPP^+}$ is isomorphic to $\NC(W,c)$.  
\end{corollary}

\Cref{cor:nc2} is illustrated in~\Cref{fig:shard}, where $s$ is the reflection through $H_8$, $t$ is the reflection through $H_1$, and $c=st$.  In this example, the noncrossing shards are $H_1,\Sigma_6,\Sigma_7$ and $H_8$, with corresponding $c$-sortable regions $B$, $R_1$, $R_6$, $R_7$, $R_8$, and $-B$.

\begin{corollary}\label{cor:nc1} The map $\Snap$ restricts to a poset embedding of $(\Sort(W,c),\preceq)$ into $[\id,\Delta^2]_{\BB^+}$.  In particular, the subposet $(\Snap(\Sort(W,c)),\leq)$ of $[\id,\Delta^2]_{\BB^+}$ is isomorphic to $\NC(W,c)$.  
\end{corollary}

\Cref{cor:nc1} is illustrated in~\Cref{fig:e_to_delta_B}.  We now link this result to the discussion in~\Cref{subsec:snap_intro}.  Recall that for $\s \in{\bf S}$ and a positive braid $\w=\s_{i_1} \cdots \s_{i_k}\in \BB^+(W)$ with projection $w=\varphi(\w) \in W$, we let $\s^\w = (t,j)$, where $t=s^w=w^{-1}sw$ and $j$ counts the number of times $t$ appears in the sequence $s^{s_{i_k}},s^{s_{i_k}s_{i_{k-1}}},\ldots,s^{s_{i_k}s_{i_{k-1}}\cdots s_1}$.

In~\cite{galashin2022rational}, a new set of noncrossing $W$-Catalan objects was introduced as the set of subwords of $\mathbf{c}^{h+1}$ that represent the full twist $\Delta^2=\w_\circ^2$ and satisfy an additional \emph{Deodhar condition}. When interpreted in the positive braid monoid $\BB^+(W)$, this Deodhar condition is equivalent to restricting to the subset of $\Sort^2(W,c)$ consisting of the $c$-sortable elements $\w$ with the property that for each $\s \in \des(\w)$, we have $\s^\w = (t,j)$ with $j$ even.  Thus, \Cref{cor:nc1} is equivalent to the observation in~\Cref{subsec:snap_intro} that the $c$-noncrossing partition lattice appears as the subposet of the $m$-th $c$-Fuss--Cambrian lattice for $m=2$ consisting of the $c$-sortable elements whose corresponding subword complex facets satisfy the Deodhar conditions. 

\section{Pow}\label{sec:pow}

In this section, $\HH$ is an arbitrary finite central irreducible real hyperplane arrangement. We will introduce a map $\Pow$ from the set $\RR$ of regions to the pure shard monoid $\PPP^+(\HH,B)$.  Just as $\Crackle$ embeds the ``short and wide'' poset $\Shard(\HH,B)$ into the ``tall and wide'' interval $[\id,\Delta^2]_{\PPP^+}$, \Cref{thm:pow_hom} states that $\Pow$ embeds the ``tall and slender'' poset $\Weak(\HH,B)$ inside $[\id,\Delta^2]_{\PPP^+}$; this theorem is illustrated in~\Cref{fig:e_to_delta_P2}.  Thus, the interval $[\id,\Delta^2]_{\PPP^+}$ simultaneously contains $\Shard(\HH,B)$ and $\Weak(\HH,B)$.

We require a short argument to show that $\Pow$ will be well-defined.

\begin{proposition}\label{prop:pow_well_defined}
Let 
\[B=C_0\xrightarrow{e_1}C_{1}\xrightarrow{e_{2}} \cdots \xrightarrow{e_{k-1}}C_{k-1}\xrightarrow{e_k}C_k= C\quad\text{and}\quad B=C_0'\xrightarrow{e_1'}C_{1}'\xrightarrow{e_{2}'} \cdots \xrightarrow{e_{k-1}'}C_{k-1}'\xrightarrow{e_k'}C_k'= C
\]
be two positive minimal galleries from the base region $B$ to a region $C \in \RR$.  Then \[\tt_{\Sigma(e_k)}\tt_{\Sigma(e_{k-1})}\cdots \tt_{\Sigma(e_1)} = \tt_{\Sigma(e'_k)}\tt_{\Sigma(e'_{k-1})}\cdots \tt_{\Sigma(e'_1)}.\]
\end{proposition}

\begin{proof}
Any two minimal galleries from $B$ to $C$ are homotopic by a succession of homotopies across $2$-cells of $\Sal(\HH)$~\cite[Lemma 11]{salvetti1987topology}. Hence, it suffices to prove the proposition when $\HH$ has rank $2$ and $C=-B$. In this case, the result is immediate from \Cref{lem:reverse_order_pow}. 
\end{proof}

\begin{definition}\label{def:pow}
Let $C \in \RR$, and let \[B=C_0\xrightarrow{e_1}C_{1}\xrightarrow{e_{2}} \cdots \xrightarrow{e_{k-1}}C_{k-1}\xrightarrow{e_k}C_k= C\] be a positive minimal gallery from $B$ to $C$.  Define \[\Pow(C):=\tt_{\Sigma(e_k)}\tt_{\Sigma(e_{k-1})}\cdots \tt_{\Sigma(e_1)}.\]  
\end{definition}

\subsection{Proof of \texorpdfstring{\Cref{thm:pow_hom}}{Theorem 1.6}}

Suppose $C,D\in\RR$ are such that $D\leq C$ in $\Weak(\HH,B)$. Let \[B=C_0\xrightarrow{e_1}C_{1}\xrightarrow{e_{2}} \cdots \xrightarrow{e_{k-1}}C_{k-1}\xrightarrow{e_k}C_k= C\] be a positive minimal gallery from $B$ to $C$, where $D=C_m$. Let \[-C=-C_k\xrightarrow{f_k} -C_{k-1}\xrightarrow{f_{k-1}}\cdots \xrightarrow{f_{2}}-C_{1}\xrightarrow{f_1} -C_0= -B\] be the corresponding positive minimal gallery from $-C$ to $-B$ (with $-D=-C_m$).  Using \Cref{lem:reverse_order_pow}, we find that 
\[\Pow(C)=\tt_{\Sigma(e_k)}\tt_{\Sigma(e_{k-1})}\cdots \tt_{\Sigma(e_1)}=\tt_{\Sigma(f_1)}\tt_{\Sigma(f_2)}\cdots \tt_{\Sigma(f_k)}\] and 
\[\Pow(D)=\tt_{\Sigma(e_m)}\tt_{\Sigma(e_{m-1})}\cdots \tt_{\Sigma(e_1)}=\tt_{\Sigma(f_1)}\tt_{\Sigma(f_2)}\cdots \tt_{\Sigma(f_m)}.\] This demonstrates that $\Pow(D)\leq \Pow(C)$ in $\PPP^+(\HH,B)$. 

Now consider the natural homomorphism from $\pi_1(\PP(\HH),B)$ to its abelianization $H_1(\PP(\HH),B)$. If $C\in\RR$ and we apply this homomorphism to $\Pow(C)$, then, by \Cref{thm:abelianization}, we obtain $\prod_{H\in\inv(C)}\overline{\tt}_H$. It follows that if $C,D\in\RR$ are such that $\Pow(D)\leq\Pow(C)$ in $\PPP^+(\HH,B)$, then $\inv(D)\subseteq\inv(C)$, so $D\leq C$ in $\Weak(\HH,B)$.

\begin{comment}\b\begin{proposition}\label{prop:crackle_max}
For $\HH$ simplicial, $\Crackle(C)$ is the unique maximal element of the image of $\Crackle$ less than $\Pow(C)$.
\end{proposition}
\begin{proof}
We know that $\Crackle(C)$ is a prefix of $\Pow(C)$.

Assume that $\Crackle(D) \leq \Pow(C)$ in $\PPP^+$.  We want to show that
$\Crackle(D) \leq \Crackle(C)$ in $\PPP^+$, which is equivalent to $D \preceq C$.

\end{proof}

\begin{corollary}\label{cor:shard_is_weaker}
For $\HH$ simplicial, $\Shard(\HH,B)$ is a weakening of $\Weak(\HH,B)$.
\end{corollary}
\begin{proof}
This follows from~\Cref{prop:crackle_max}.
\end{proof}

The following is immediate from~\Cref{cor:shard_is_weaker}.
\begin{corollary}
For $W$ a finite Coxeter group with Coxeter element $c$, $\NC(W,c)$ is a weakening of $\Sort(W,c)$.
\end{corollary}
\end{comment}

\section{Future Work}
\label{sec:future}

\subsection{Noncrossing Pure Braid Presentations} In future work, we will combine our \Cref{thm:shard_generators} with Salvetti's~\Cref{thm:relations}, Coxeter--Catalan combinatorics, Cambrian lattices, and noncrossing shards to write explicit presentations of the pure braid groups of finite Coxeter groups.  In the special case of the symmetric group and the Tamari lattice, our method will recover Artin's original presentation of the pure braid group~\cite{artin1925theorie,artin1947theory}.

\subsection{The Pure Shard Monoid} 
The pure shard monoid is an interesting algebraic and order-theoretic structure that deserves further study---in particular, we would like to better understand the elements and the maximal chains in the interval $[\id,\Delta^2]_{\PPP^+}$. These maximal chains correspond to $\TTsha$-words representing the full twist $\Delta^2$.  When $\HH$ is an arrangement of rank 2, we characterized these words for $\Delta^2$ as the cyclic rotations of unimodal $\TTsha$-words (\Cref{prop:rank2maxchain}). It would already be interesting to better understand $[\id,\Delta^2]_{\PPP^+}$ for special cases, such as rank-3 arrangements or reflection arrangements of type-$A$ Coxeter groups.

\subsection{Infinite Arrangements}
Throughout this paper, we have assumed that $\HH$ is finite. It is natural to ask what aspects of the above theory generalize to arrangements with infinitely many hyperplanes.

\subsection{Bubbles, Blossom, Buttercup (and Bliss)}

We view the bubble sort operator $\mathsf{Bubbles}$ as the 0-Hecke action of any reduced word for the long element $w_\circ$ in the symmetric group.  The \emph{higher Bruhat order} is a partial order defined on these reduced words~\cite{manin1989arrangements}.  We wonder if there are similar ``higher Bruhat orders'' built from the $\TTsha$-words for $\Delta^2$. One might expect to find relevant maps $\mathsf{Blossom}$ and $\mathsf{Buttercup}$ in this theory. We recommend this subsection's title as the logical name for this proposed work.

\section*{Acknowledgements}
Colin Defant was supported by the National Science Foundation under Award No.\ 2201907 and by a Benjamin Peirce Fellowship at Harvard University.  Nathan Williams was partially supported by the National Science Foundation under Award No.\ 2246877.  This work benefited from computations in \texttt{Sage}~\cite{sagemath}, the combinatorics features developed by the \texttt{Sage-Combinat} community~\cite{Sage-Combinat}, as well as CHEVIE~\cite{GH96}.  We thank Nathan Reading for the clarity of his exposition and Ariel Williams for advice.

Nathan Williams is very grateful to Jon McCammond for introducing him to this area of investigation, and in particular for showing him the Salvetti complex and the part of $[\id,\Delta^2]_{\PPP^+}$ generated by noncrossing shards.

\bibliographystyle{alpha}
\bibliography{literature}

\end{document}